\documentclass[a4paper,12pt]{article}
\usepackage{amsmath}
\usepackage{amssymb}
\usepackage{amsthm}
\usepackage{amscd}
\usepackage{amsfonts}
\usepackage{graphicx}%
\usepackage{empheq}
\usepackage{morefloats}
\usepackage[dvipsnames]{xcolor}

\usepackage[font={small,sf}]{caption}

\usepackage{enumerate}
\usepackage{enumitem}

\newcommand{\be}{\begin{equation}}
\newcommand{\ee}{\end{equation}}
\newcommand{\eps}{\varepsilon}
\newcommand{\mb}{\mathbf}
\newcommand{\bs}{\boldsymbol}
\newcommand{\Real}{\mathbb R}

\renewcommand{\Re}{\operatorname{Re}}
\renewcommand{\Im}{\operatorname{Im}}

\newcommand{\Wcal}{\mathcal W}
\newcommand{\mc}{\mathcal}

\usepackage{enumerate}
\usepackage{amsfonts}

\usepackage{amsthm}
\usepackage{psfrag}
\usepackage{epsfig}
\usepackage{subfigure}
\usepackage{graphicx}

\usepackage{amssymb}
\usepackage{epstopdf}
\usepackage{amsbsy}
\usepackage{latexsym}
\usepackage{moreverb}
\usepackage{textcomp}
\usepackage{verbatim}

\usepackage[T1]{fontenc}
\usepackage[utf8]{inputenc}
\usepackage{authblk}
\usepackage{amsbsy}

\begin{document}

\theoremstyle{remark} \newtheorem*{remark}{Remark}
\theoremstyle{plain} \newtheorem{theorem}{Theorem}[section]
\theoremstyle{plain} \newtheorem{proposition}[theorem]{Proposition}
\theoremstyle{plain}	\newtheorem{lemma}[theorem]{Lemma}
\theoremstyle{plain}	\newtheorem{cor}{Corollary}
\theoremstyle{definition}	\newtheorem{definition}{Definition}
\theoremstyle{corollary}	\newtheorem{corollary}{Corollary}
\theoremstyle{remark}	\newtheorem{example}{Example}

%\titlehead{%\centering\includegraphics[width=3.5cm]{figures/KTH}}
\title{Stochastic regularity of general quadratic observables of high frequency waves}
\author{G. Malenov\'a and O. Runborg}
%\date{\today}
\date{}

\maketitle

\begin{abstract}
 We consider the wave equation with uncertain initial data and medium, 
 when
the wavelength $\varepsilon$ of the solution is short compared to the
distance traveled by the wave. 
We are interested in the statistics for quantities of interest (QoI),
defined as functionals of the wave solution,
given the probability distributions of the uncertain parameters
in the wave equation.
Fast methods to compute this statistics
require considerable smoothness in the mapping from parameters to the QoI,
which is typically not present in the high frequency case,
as the oscillations on the $\varepsilon$ scale
in the wave field is inherited by the QoIs.
The main contribution of this work is to 
identify certain non-oscillatory quadratic QoIs and show
$\varepsilon$-independent estimates for the derivatives of the QoI
with respect to the parameters, when the wave solution is
replaced by a Gaussian beam approximation.
\end{abstract}

\section{Introduction}

Many physical phenomena can be described by propagation of high-frequency waves
with stochastic parameters. For instance, an earthquake
where seismic waves with uncertain epicenter travel through the layers of the Earth
with uncertain soil characteristics represents one such problem stemming from geophysics. Similar problems arise e.g. in optics, acoustics
or oceanography. By high frequency we understand that the wavelength is very
short compared to the distance traveled by the wave.

As a simplified model of the wave propagation, we use the scalar wave equation 
\begin{subequations}\label{waveeq}
\begin{align}
&u^\eps_{tt}(t,\mb{x},{\mb y})= c(\mb{x},{\mb y})^2 \, \Delta u^\eps(t,\mb{x},{\mb y}), \qquad &&\text{in  } [0,T] \times {\Real}^n \times \Gamma, \label{wave_1st}\\
&u^\eps(0,\mb{x},{\mb y})=B_0(\mb{x},{\mb y}) \, e^{i \, \varphi_0(\mb{x},{\mb y})/\eps}, &&\text{in } {\Real}^n \times \Gamma,\\
&u^\eps_t(0,\mb{x},{\mb y})=\eps^{-1}B_1(\mb{x},{\mb y}) \, e^{i \, \varphi_0(\mb{x},{\mb y})/\eps},
&&\text{in } {\Real}^n \times \Gamma, \label{wave_2nd}
\end{align}
\end{subequations}
with highly oscillatory initial data, represented by the small wavelength $\eps\ll 1$, and a stochastic parameter ${\mb y} \in \Gamma \subset\Real^N$ which models the uncertainty. 
For realistic problems,
the dimension $N$ of the stochastic space can be fairly large. Two sources of uncertainty are considered: the local speed, $c = c(\mb x,\mb y)$, and the initial data, $B_0=B_0(\mb x,\mb y)$, $B_1=B_1(\mb x,\mb y)$, $\varphi_0=\varphi_0(\mb x,\mb y)$. The solution is therefore also a function of the random parameter, $u^\eps=u^\eps(t,\mb x,\mb y)$. 

The focus of this work is on the regularity of certain
nonlinear functionals of the solution $u^\eps$ with respect
to the random parameters ${\mb y}$. 
Our motivation for the study comes from the field of
uncertainty quantification (UQ), where the functionals represent
{\em quantities of interest} (QoI). We will denote them
generically by ${\mathcal Q}({\bf y})$.
The aim in (forward) UQ is to compute the statistics of 
${\mathcal Q}$, typically
the mean and the variance, given the probability distribution of
${\bf y}$. This is often  done by random sample based methods
like Monte--Carlo \cite{MC}, % and Quasi Monte--Carlo (QMC).
which, however, has a rather slow convergence rate;
the error decays as $O(N^{-1/2})$
for $N$ samples.
%for these methods are however
%slow: $N^{-1/2}$ for MC and $N^{-1}$ for QMC, where $N$
%is the number of samples.
Grid based methods like 
Stochastic Galerkin (SG) 
\cite{Ghanem,Xiu_Karniadakis:02,BTZ:05,Todor_Schwab}
and Stochastic Collocation (SC) 
%\cite{Xiu_Hesthaven,BNT:07,NTW:08,NTW:08_2,Motamed_etal:12,Ernst_Sprung:14,Gunzburger_etal:14}
\cite{Xiu_Hesthaven,BNT:07,NTW:08}
can achieve much faster convergence rates, even spectral rates 
where the error decays faster than $N^{-p}$ for all $p>0$.
They rely on smoothness of ${\mathcal Q}({\bf y})$
with respect to ${\bf y}$.
This smoothness is referred to as the {\em stochastic
regularity} of the problem. When ${\bf y}$ is a high-dimensional
vector, SG and SC must be performed on sparse grids
\cite{Bungartz_Griebel:04,Griebel_Knapek:09}
to break the curse of dimension. This typically
requires even stronger stochastic regularity.

To show the fast convergence of SG and SC, analysis of the
stochastic regularity has been carried out
for many different PDE problems. Examples include
 elliptic problems 
\cite{BabuskaNobileTempone:2010,CohenDevoreSchwab:2011,NobileTempone:2009},
the wave equation \cite{MotamedNobileTempone:2013}, Maxwell equations 
\cite{LiFangLin:2018}
and various kinetic equations
\cite{JinLiuMa:2017,LiWang:2017,LiuJin:2017,JinZhu:2018,ShuJin:2018}.

In the high frequency case, which is the subject of this article,
the main question is how the ${\mb y}$-derivatives of ${\mathcal Q}$ 
depend on
the wave length $\varepsilon$. The solution $u^\eps$ oscillates
with period $\varepsilon$ and these oscillations are
often inherited by ${\mathcal Q}$. If this is the case, 
SG and SC will not work
well, as the 
derivatives of ${\mathcal Q}$ grow rapidly with $\varepsilon$.
Special choices of ${\mathcal Q}$ can, however, have better
properties, as we discuss below. 
A further complication is that the direct
numerical solution of \eqref{waveeq}
becomes infeasible as $\varepsilon\to 0$, as the computational cost 
to approximate $u^\eps$
is of order $O(\eps^{-n-1})$. Asymptotic methods based on
e.g. {\em geometrical optics} \cite{EngRun:03,Runborg07} or 
\emph{Gaussian beams} (GB) \cite{Cerveny_etal:1982,Ralston82}
must therefore be used.
% For those methods the cost 
%grows much slower with $\varepsilon$.

%QoIs based on this approximation could
%potentially have different stochastic regularity than the
%exact solution, although this is
In \cite{MaMoRuTe16} we
identified a non-oscillatory quadratic QoI, 
\be\label{QoIsmall0}
 \widetilde{\mc Q}(t,\mb y):=  \int_{\Real^n} |u^\eps (t,\mb x,\mb y)|^2\psi(t,\mb x)\, d\mb x,
 \qquad  \psi\in C_c^\infty(\Real\times\Real^n),
\ee
and introduced a GB solver for $u^\eps$
coupled with SC on sparse grids to approximate it.
A big advantage of the GB method is that it approximates the solution to the PDE \eqref{waveeq} via solutions
to a set of $\eps$-independent ODEs instead.
In \cite{malenova2017stochastic} we also showed 
rigorously that all derivatives of $\widetilde{\mathcal Q}$
are bounded independently of $\varepsilon$ when the
wave solution $u^\eps$ is approximated by Gaussian beams,
\[%be\label{QoIQoI}
	\sup_{\mb y\in \Gamma}\left|\frac{
	\widetilde{\mc Q}(t,\mb y)}{\partial \mb y^{\bs \sigma}}\right|\leq C_{\bs\sigma}, \qquad \forall \bs\sigma\in \mathbb N_0^N,
\]%ee
where $C_{\bs\sigma}$ are independent of $\varepsilon$.
A related study is found in
\cite{JinLiuRussoZhou:2019}.

In this article we generalize the result in 
\cite{malenova2017stochastic}
and consider QoIs
which include higher order derivatives of the solution
and also averaging in time. More precisely, we study
\be\label{QoIbig}%be%\label{QoIfirst}
 \mc Q^{p,\bs\alpha}(\mb y) = \eps^{2(p+|\bs\alpha|)}\int_{\Real}\int_{\Real^n} g(t,\mb x,\mb y)|\partial_t^p\partial_{\mb x}^{\bs\alpha}u^\eps (t,\mb x,\mb y)|^2\psi(t,\mb x)\, d\mb x\, dt,
\ee
with $g\in C^\infty(\Real\times\Real^n\times \Gamma)$,
$p$ a non-negative integer and ${\bs\alpha}$ a multi-index.
Many physically relevant QoIs can be written on this form.
%The article generalizes the results in \cite{malenova2017stochastic} where the QoI
%\be\label{QoIsmall0}
% \widetilde{\mc Q}(t,\mb y):= \widetilde{\mc Q}^{0,\mb 0}(t,\mb y) = \int_{\Real^n} |u^\eps (t,\mb x,\mb y)|^2\psi(t,\mb x)\, d\mb x,
%\ee
%was studied.
The simplest case in \eqref{QoIbig},
\be\label{QoIbig0}
\mc Q(\mb y):=\mc Q^{0,\mb 0}(\mb y) = \int_\Real \int_{\Real^n} |u^\eps(t,\mb x,\mb y)|^2 \psi(t,\mb x) \, d\mb x\, dt,
\ee
represents the weighted average intensity of the wave. If the solution $u^\eps$ to \eqref{waveeq} describes the pressure, then $\mc Q$ represents the acoustic potential energy. 
Another significant example is the weighted total energy of the wave,
\[
E(\mb y) = \eps^2 \int_\Real \int_{\Real^n} (|u_t^\eps(t,\mb x,\mb y)|^2 + c^2(\mb x,\mb y) |\nabla u^\eps(t,\mb x,\mb y)|^2 )\psi(t,\mb x)\, d\mb x\, dt,
\]
which can be decomposed into terms of type \eqref{QoIbig}.
An additional example is the weighted and averaged version of the Arias intensity,
\[
 I(\mb y) = \eps^4 \int_{\Real}\int_{\Real^n} |u_{tt}^\eps(t,\mb x,\mb y)|^2 \psi(t,\mb x)\, d\mb x\, dt,
\]
which represents the total energy per unit mass and is used to
measure the
strength of ground motion during an earthquake, see \cite{hansen1970seismic}.
%stored by a set of undamped oscillators 
%at the end of an earthquake

In this work we show that also the QoI \eqref{QoIbig}
is non-oscillatory
when $u^\eps$ is replaced by the GB approximation $\tilde{u}$. Indeed, under the assumptions given in 
Section~\ref{sec:ass}
we then prove that
%
%In this article we generalize the result in 
%
%A necessary condition for this method to converge is that the QoIs are regular in the stochastic space. That is, we require that 
for all compact $\Gamma_c\subset \Gamma$ and all $\bs \sigma\in \mathbb N_0^N$,
\be\label{QoIQoI}
	\sup_{\mb y\in \Gamma_c}\left|\frac{\partial^{\bs\sigma} \mc Q^{p,\bs\alpha}(\mb y)}{\partial \mb y^{\bs \sigma}}\right|\leq C_{\bs\sigma},
\ee
for some constants $C_{\bs\sigma}$, uniformly in $\eps$.

The full GB approximation $\tilde u$ features two modes, $\tilde u= \tilde u^++\tilde u^-$, satisfying two different sets of ODEs. In certain cases, it is possible to approximate $u^\eps$ by one of 
the modes only, i.e. either $\tilde u = \tilde u^+$ or $\tilde u =\tilde u^-$. We can then examine a QoI that, in contrast to \eqref{QoIbig}, 
is only integrated in space,
\be\label{QoIsmall}
 \widetilde{\mc Q}^{p,\bs\alpha}(t,\mb y) = \eps^{2(p+|\bs\alpha|)}\int_{\Real^n} g(t,\mb x,\mb y)|\partial_t^p\partial_{\mb x}^{\bs\alpha}u^\eps (t,\mb x,\mb y)|^2\psi(t,\mb x)\, d\mb x,
\ee
%with $\psi\in C^\infty_c(\Real^n)$ and $g\in C^\infty(\Real\times\Real^n\times\Gamma)$ 
and show
a stronger regularity result, 
\be\label{est1}
	\sup_{\substack{\mb y\in \Gamma_c\\t\in[0,T]}}\left|\frac{\partial^{\bs\sigma} \widetilde{\mc Q}^{p,\bs\alpha}(t,\mb y)}{\partial \mb y^{\bs \sigma}}\right|\leq C_{\bs\sigma},  \qquad \forall \bs\sigma\in \mathbb N_0^N,
\ee
uniformly in $\eps$, when $u^\eps$ is replaced by $\tilde u^\pm$.
In fact, this one-mode case, with $p={\bs\alpha}=0$, was the one considered in
\cite{malenova2017stochastic}. 
%The result here generalizes
%that result to general $p$ and ${\ms \alpha}$.
%The simplest case \eqref{QoIsmall0} was considered in , % where the QoI was studied in \cite{malenova2017stochastic}. 
%including the proof of the regularity of the type \eqref{est1} for $u^\eps$ approximated by a one-mode GB solution $\tilde u$.

The layout of this article is as follows: we briefly introduce our assumptions in Section \ref{sec:ass} and then present the Gaussian beam method in Section \ref{sec:GB}. The one-mode QoI 
\eqref{QoIsmall}
with $u^\eps$ approximated by $\tilde u=\tilde{u}^\pm$ is regarded in Section \ref{sec:onefamily}. The stochastic regularity \eqref{est1} is shown in Theorem~\ref{th:new}. 
This serves as a stepping stone for the proof of regularity of the 
general two-mode QoI \eqref{QoIbig}
with
$u^\eps$ approximated by
 $\tilde u = \tilde u^++\tilde u^-$,
which is the subject of Section \ref{sec:twofamily} where
the final stochastic regularity \eqref{QoIQoI} is shown in 
Theorem~\ref{th:main}.

%%%%%%%%%%%%%%%%%%%%%%%%%%%%%%%%%%%%%%%%%%%%%%%%%%%%%%%%%%%%%%%%%%%%%%%%%%%%%%%%%%%%%%%%%%%%%%%%%%%%%%%%%%%
\section{Assumptions and preliminaries}\label{sec:ass}
Let us consider the Cauchy problem \eqref{waveeq}. By $t\in [0,T]\subset \Real$ we denote the time, $\mb x = (x_1,\ldots,x_n)\in \Real^n$ is the spatial variable and the uncertainty in the model is described by the random variable $\mb y = (y_1,\ldots y_N)\in \Gamma$ where $\Gamma\subset \Real^N$ is an open set. %The independent random variables $y_1,\ldots y_N$ have a bounded joint probability density $\rho(\mb y) = \prod_{n=1}^N \rho_n(y_n):\Gamma\to\Real_+$.
By ${\mc B}_\mu$ we will denote the $n$-dimensional closed ball around 0
of radius $\mu$, i.e. the set 
${\mc B}_\mu := \{\mb x\in \Real^n: |\mb x|\leq \mu\}$,
with the convention that ${\mc B}_\infty = \Real^n$.

We make the following precise assumptions.
\begin{enumerate}[label=(A\arabic*)]
\item \label{ASSUM1} Strictly positive, smooth and bounded speed of propagation,
%$c\in C^\infty(\Real^n\times\Gamma)$, such that
$$
c\in C^\infty(\Real^n\times\Gamma),\qquad
0<c_{\rm min} \le c(\bold{x},{\bf y}) \le c_{\rm max} < \infty, \qquad \forall \, {\bf x} \in {\mathbb R}^n, \quad \forall \, {\bf y} \in \Gamma.
$$
and for each multi-index pair $\bs\alpha$, $\bs\beta$ there is a constant $C_{\bs\alpha,\bs\beta}$ such that
\[
\left|\partial_{\bf x}^{\bs \alpha}\partial_{\bf y}^{\bs \beta} c({\bf x},{\bf y}) \right|\leq
C_{\bs\alpha,\bs\beta}, \qquad
\forall\, {\bf x} \in {\mathbb R}^n, \quad \forall \, {\bf y} \in \Gamma.
\]
\item \label{ASSUM2} Smooth and (uniformly) compactly supported initial amplitudes,
$$
B_{{\ell}}\in C^\infty(\Real^n\times\Gamma),\qquad
 {\rm supp}\,B_{{\ell}}(\,\cdot\,,{\bf y})\subset K_0,
 \qquad \ell=0,1,\quad \forall \, \mb y\in\Gamma,
$$
where $K_0\subset \Real^n$ is a compact set.
\item \label{ASSUM3} Smooth initial phase with non-zero gradient,
$$
  \varphi_0\in C^\infty(\Real^n\times\Gamma),\qquad
|\nabla\varphi_0({\bf x},{\bf y})|>0,\qquad \forall \, {\bf x} \in {\mathbb R}^n, \quad \forall \, {\bf y} \in \Gamma.
$$
%\item \label{ASSUM4}Compact support in random space: the set $\Gamma$ is bounded.
\item \label{ASSUM5}High frequency,
$$
0<\varepsilon\leq 1.
$$
\item \label{ASSUM6} %Smooth and compactly supported QoI test function,
%\[\psi \in C_c^\infty(\mathbb R^n), \qquad \text{supp}\, \psi\subset K_1,\]
%where $K_1\subset \Real^n$ is a compact set.
%\item[\ref{ASSUM6}'] \label{ASSUM6p}
Smooth and compactly supported QoI test function,
\[
 \psi\in C^\infty_c(\Real\times \Real^n),\qquad \text{supp}\, \psi \subset [0,T]\times K_1,\]
 where $K_1\subset \Real^n$ is a compact set.
\end{enumerate}
Throughout the paper we will frequently
use the shorthand  $f\in C^\infty$
with the understanding that
$f$ is continuously differentiable infinitely
many times in each of its variables, over its entire
domain of definition,
typically 
$\Real\times\Real^n\times\Gamma\times \Real^n$
or $\Real\times\Real^n\times\Gamma\times\Real^n\times\Real^n$.

%Let us now introduce the following notation used throughout this paper.
%\begin{itemize}
%\item
%%By $B_\mu$ we will denote the $n$-dimensional closed ball around 0, i.e. the set 
%$B_\mu := \{\mb x\in \Real^n: |\mb x|\leq \mu\}$.
%\item $\mc X:= C^\infty(\Real\times\Real^n\times\Gamma\times \Real^n)$.
%\item $\mc Y := C^\infty(\Real\times\Real^n\times\Gamma\times\Real^n\times\Real^n).$
%\end{itemize}
%Note that $B_\infty = \Real^n$.
%%%%%%%%%%%%%%%%%%%%%%%%%%%%%%%%%%%%%%%%%%%%%%%%%%%%%%%%%%%%%%%%%%%%%%%%%%%%%%%%%%%%%%%%%%%%%%%%%%%%%%%%%%%%%%%%%%%%%%%%%%%%%%%%%%%%
\section{Gaussian beam approximation}\label{sec:GB}
Solving \eqref{waveeq} directly requires a substantial number of numerical operations when the wavelength $\eps$ is small. In particular, to maintain a given accuracy for a fixed $\mb y$, we need at least $O(\eps^{-n})$ discretization points in $\mb x$ and $O(\eps^{-1})$ time steps resulting into the computational cost $O(\eps^{-n-1})$.
To avoid the high cost we employ asymptotic methods arising from geometrical optics. In particular, the Gaussian beam (GB) method provides a powerful tool, see \cite{Cerveny_etal:1982, LRT:13, Ralston82, Runborg07, Tanushev08}. %Let us now define the $k$-th order GB solution to the wave equation.

Individual Gaussian beams are asymptotic solutions to the wave equation \eqref{waveeq} that
concentrate around a central ray in space-time.
Rays are bicharacteristics of the wave equation \eqref{waveeq}. They are denoted by $(\mb q^\pm,\mb p^\pm)$ where $\mb q^\pm(t,\mb y,\mb z)$ represents the position and $\mb p^\pm(t,\mb y,\mb z)$ the direction, respectively, and $\mb z\in K_0$ is the starting point so that $\mb q^\pm(0,\mb y,\mb z) = \mb z$ for all $\mb y\in \Gamma$. From each $\mb z$, the ray propagates in two opposite directions, here distinguished by the superscript $\pm$. These corresponds to the two modes of the wave
equation and leads to two different GB solutions, one
for each mode.
% The GB solution also features two modes 
%
%distinguished by the $\pm$ superscript.
We denote the two $k$-th order Gaussian beams starting at $\mb z\in K_0$ by $v_k^\pm(t,\mb x,\mb y,\mb z)$ and define it as
\begin{equation}\label{vdef}
v_{k}^\pm(t,\mb x,\mb y,\mb z) =
A_{k}^\pm(t,\mb x-\mb q^\pm(t,\mb y,\mb z),\mb y,\mb z) e^{i\Phi_{k}^\pm(t,\mb x-\mb q^\pm(t,\mb y,\mb z),\mb y,\mb z) /\varepsilon},
\end{equation}
where

\begin{equation}\label{phidef}
  \Phi_{k}^\pm(t,\mb x,\mb y,\mb z) =
  \phi_{0}^\pm(t,\mb y,\mb z) + \mb x^T \mb p^\pm(t,\mb y,\mb z) + \frac12\,\mb x^T  M^\pm(t,\mb y,\mb z) \mb x +
  \sum_{|\bs\beta|=3}^{k+1} \frac1{\bs\beta!}\phi_{\bs\beta}^\pm(t,\mb y,\mb z) \mb x^{\bs\beta},
\end{equation}
is the $k$-th order phase function
and
\begin{equation}\label{Adef1}
 % A_{k}^\pm(t,\mb x,\mb y,\mb z)= \sum_{j=0}^{\lceil \frac{k}{2} \rceil -1} \varepsilon^j \bar{a}_{j,k}^\pm(t,\mb x,\mb y,\mb z),
  A_{k}^\pm(t,\mb x,\mb y,\mb z)= \sum_{j=0}^{\lceil \frac{k}{2} \rceil -1} \varepsilon^j \sum_{|\bs\beta|=0}^{k-2j-1} \frac1{\bs\beta!}a_{j,\bs\beta}^\pm(t,\mb y,\mb z) \mb x^{\bs\beta},  
\end{equation}
is the $k$-th order amplitude function. 
The higher the order $k$, the more accurately $v_k^\pm$ approximates the solution to \eqref{waveeq} in terms of $\eps$.
%, with
%\begin{equation}\label{Adef2}
%  \bar{a}_{j,k}^\pm(t,\mb x,\mb y,\mb z)  = \sum_{|\bs\beta|=0}^{k-2j-1} \frac1{\bs\beta!}a_{j,\bs\beta}^\pm(t,\mb y,\mb z) \mb x^{\bs\beta} \ .
%\end{equation}
The variables $\phi_0^\pm, \mb q^\pm, \mb p^\pm, M^\pm, \phi_{\bs\beta}^\pm, a_{j,\bs\beta}^\pm$ are given by a set of ODEs, the simplest ones being
\begin{subequations}\label{IC}
\begin{align}
    \dot \phi_0^\pm & = 0,\\
    \dot{\mb q}^\pm & = \pm c(\mb q^\pm) \frac{\mb p^\pm}{|\mb p^\pm|},\\
    \dot{\mb p}^\pm & = \mp \nabla c(\mb q^\pm) |\mb p^\pm|,\\
    \dot M^\pm & = \mp(D^\pm + (B^\pm)^T M^\pm + M^\pm B^\pm + M^\pm C^\pm M^\pm),\\
    \dot a_{0,\mb 0}^\pm & = \pm \frac{1}{2 |\mb p^\pm|} \left(-c(\mb q^\pm) \text{Tr}(M^\pm) + \nabla c(\mb q^\pm)^T  \mb p^\pm + \frac{c(\mb q^\pm) (\mb p^\pm)^T M^\pm \mb p^\pm}{|\mb p^\pm|^2}\right) a_{0,\mb 0}^\pm,
\end{align}
\end{subequations}
where
\[
  B^\pm = \frac{\mb p^\pm \nabla c(\mb q^\pm)^T}{|\mb p^\pm|}, \qquad C^\pm = \frac{c(\mb q^\pm)}{|\mb p^\pm|} - \frac{c(\mb q^\pm)}{|\mb p^\pm|^3}  \mb p^\pm (\mb p^\pm)^T, \qquad D^\pm= |\mb p^\pm| \nabla^2 c(\mb q^\pm).
\]
For the ODEs determining $\phi_{\bs\beta}^\pm$ and $a_{j,\bs\beta}^\pm$ other than the leading term we refer the reader to \cite{Ralston82, Tanushev08}.

As mentioned above, the sign corresponds to GBs moving in opposite directions which means that they 
constitute two different modes that
are governed by two different sets of ODEs.
%{\bf EXPAND THIS DISCUSSION}
%We will say that GBs following the same set of ODEs (i.e. all either described by $v_k^+$ or $v_k^-$ in \eqref{vdef}) constitute \emph{the same mode}.
%
Single beams from the same mode with their starting points in $K_0$ are summed together to form the
$k$-th order {\em one-mode} solution $u_k^\pm(t,\mb x,\mb y)$,
\be\label{ukdef}
u^{\pm}_k(t,\mb x,\mb y) = 
\left(\frac{1}{2\pi \eps}\right)^{n/2} \int_{K_0}v_k^{\pm}(t,\mb x,\mb y,\mb z)\varrho_\eta(\mb x-\mb q^{\pm}(t,\mb y,\mb z))d\mb z.
\ee
%\begin{equation}\label{udef}
%u_{k}(t,\mb x,\mb y) =
%    \left(\frac{1}{2\pi\varepsilon}\right)^\frac{n}{2}
%   \int_{K_0} v_{k}(t,\mb x,\mb y,\mb z) \varrho_\eta(\mb x-\mb q(t,\mb y,\mb z)) d\mb z,
%\end{equation}
where the integration in $\mb z$ is over the support of the initial data $K_0\subset\Real^n$, which is independent of ${\mb y}$
by \ref{ASSUM2}.
Since the wave equation is linear, the superposition of beams is still an
asymptotic solution.
The function $\varrho_\eta\in C^\infty(\Real^n)$
is a real-valued {\em cutoff} function with radius $0<\eta\leq\infty$,
\be\label{cutoff}
	\varrho_\eta(\mb x) = \left\{\begin{array}{lll}
							1, & \text{if } |\mb x|\leq \eta, & \text{for } 0<\eta<\infty,\\
							0, & \text{if } |\mb x|\geq 2\eta,& \text{for } 0<\eta<\infty,\\
							1, & & \text{for } \eta = \infty.
			\end{array}\right.
\ee
For first order GBs, $k=1$, one can choose $\eta = \infty$, i.e. no $\varrho_\eta$, see below.

Each GB $v_k^\pm$ requires initial values for all its coefficients. An appropriate choice makes $u_k^\pm(0,\mb x,\mb y)$ converge asymptotically as $\eps\to 0$ to the initial conditions in \eqref{waveeq}. As shown in \cite{LRT:13}, the initial data are to be chosen as follows:
\begin{subequations}\label{GBini}
\begin{align}
 \mb q^\pm(0,\mb y,\mb z) &= \mb z, \\
 \mb p^\pm(0,\mb y,\mb z) &= \nabla \varphi_0(\mb z,\mb y), \\
 \phi^\pm_0(0,\mb y,\mb z) &= \varphi_0(\mb z,\mb y), \\
 M^\pm(0,\mb y,\mb z) &= \nabla^2\varphi_0(\mb z,\mb y) + i\ I_{n\times n},\\
 \phi^\pm_{\bs\beta}(0,\mb y,\mb z) &= \partial_{\mb x}^{\bs\beta}\varphi_0(\mb z,\mb y), \qquad |\bs\beta|= 3,\ldots,k+1 ,\\
a_{0,\mb 0}^\pm(0,\mb y,\mb z) &= \frac12\left(B_0(\mb z,\mb y) \pm \frac{B_1(\mb z,\mb y)}{ic(\mb z,\mb y)|\nabla\varphi_0(\mb z,\mb y)|}\right),
%a_0^- = (B_0 - B_1/(ic|\nabla\varphi_0|))/2
 %a_{0,\mb 0}^\pm(0,\mb y,\mb z) &= B_0(\mb z,\mb y),
\end{align}
\end{subequations}
where $I_{n\times n}$ denotes the identity matrix of size $n$.
The initial data for the higher order amplitude coefficients are given in \cite{LRT:13}.
The following proposition shows that all these variables are smooth and $a_{j,\bs\beta}^\pm$ remain supported in $K_0$ for all times $t$ and random variables $\mb y\in\Gamma$.
\begin{proposition}\label{prop:smooth}%\label{prop:Thetasmooth}
Under assumptions \ref{ASSUM1}--\ref{ASSUM3}, the coefficients $\phi_0^\pm, \mb q^\pm, \mb p^\pm, M^\pm, \phi_{\bs\beta}^\pm, a_{j,\bs\beta}^\pm$
all belong to $C^\infty(\Real\times\Gamma\times\Real^n)$ and
\[
\text{\rm supp} (a_{j,\bs\beta}^\pm(t,\mb y,\cdot))\subset K_0, \qquad \forall\ t\in \Real, \, \mb y\in \Gamma.
\]
Consequently, $\Phi_k^\pm\in C^\infty$.
%$\Phi_k^\pm\in \mc X.$ 
%C^\infty(\Real\times\Real^n\times\Gamma\times\Real^n)$.
\end{proposition}
\begin{proof}
Existence and regularity of the solutions follow from 
standard ODE theory and a
result in \cite[Section 2.1]{Ralston82} which ensures that the non-linear Riccati equations for $M^\pm(t, {\mb y}; {\mb z})$ have solutions for all times and parameter values, with the given initial data.
That $\text{supp} (a_{j,\bs\beta}^\pm(t,\mb y,\cdot))$ stays in
$K_0$ for all times is a consequence of the
form of the ODEs for the amplitude coefficients, given in 
\cite{Ralston82}.
% and the fact that initial data is compactly
%supported in $K_0$, the property (23) follows.
%
%From Lemma 1 together with (A3) it follows that unique solutions q(t, y; z) and p(t, y; z) exist
%for all times, for all parameter values y 2 ?? and z 2 Rn. Moreover, the choice of initial data
%and a  The remaining coefficient functions are well-defined
%since they satisfy linear ODEs with variable, continuous, coefficients.
%The smoothness follows from the fact that initial data for all coefficient functions, as well
%as the speed of propagation, are smooth in z and y by (A1), (A2) and (A3).  This shows (P1).
%(P2) is a direct consequence of (P1) and the definition of k in (6).
%
%
%
%	Properties (P1) and (P2) of Proposition 1 in \cite{malenova2017stochastic}; the proof is in 
%	\cite{malenova2016uncertainty}.
	\end{proof}
Finally, the $k$-th order GB superposition solution is defined as a sum of the two modes in \eqref{ukdef},
\be\label{umode}
u_k(t,\mb x,\mb y) = u^+_k(t,\mb x,\mb y) + u^-_k(t,\mb x,\mb y).
\ee
Approximating $u^\eps$ with $u_{k}$ we can define the GB quantity of interest corresponding to \eqref{QoIbig} as
\be\label{QoIbigGB}
	\mc Q_{\text{GB}}^{p,\bs\alpha}(\mb y) = \eps^{2(p+|\bs\alpha|)}\int_\Real\int_{\Real^n} g(t,\mb x,\mb y)|\partial_t^p\partial_{\mb x}^{\bs\alpha}u_k(t,\mb x,\mb y)|^2 \psi(t,\mb x) d\mb x\, dt,
\ee
where $\psi$ is as in \ref{ASSUM6} and $g\in C^\infty(\Real\times\Real^n\times\Gamma)$.

We note that 
for numerical computations with SG or SC combined with GB it is 
indeed the
stochastic regularity of ${\mc Q}_{\text{GB}}^{p,\bs\alpha}$ rather
than of the exact ${\mc Q}^{p,\bs\alpha}$ that is relevant.
Moreover, since 
$u_k$ approximates the exact solution $u^\eps$ 
well, ${\mc Q}_{\text{GB}}^{p,\bs\alpha}$ will also be a good approximation
of ${\mc Q}^{p,\bs\alpha}$. 
For instance, when $p=0$ and $\alpha\neq 0$ one can use the
Sobolev estimate 
$||u_k-u^\eps||_{H^s}\leq C\eps^{k/2-s}$, for $s\geq 1$,  
shown in \cite{LRT:16},
to derive the error bound $|{\mc Q}_{\text{GB}}^{0,\bs\alpha}-{\mc Q}^{0,\bs\alpha}|\leq C\eps^{k/2}$
in the same way as in  \cite{malenova2017stochastic}, where
the case $\alpha=0$ was discussed.
Also, in some cases, like in one dimension with constant speed 
$c(x,y) = c(y)$, the GB solution is exact if the
initial data is exact. Then $\mc Q_{\text{GB}}^{p,\alpha} = \mc Q^{p,\alpha}$.

%It was argued in \cite{malenova2017stochastic} that $\mc Q^{p,\bs\alpha}$ is well-approximated by $\mc Q_{\text{GB}}^{p,\bs\alpha}$ as $u_k$ is a good approximation of $u^\eps$.
%%which represents a weighted local average of the wave strength.
%
%Notably, in the 1D case with constant speed $c(x,y) = c(y)$ the GB QoI \eqref{QoIbigGB} is exact, $\mc Q_{\text{GB}}^{p,\alpha} = \mc Q^{p,\alpha}$, provided we choose $u^\eps(0,x,y)$ in \eqref{waveeq} identical to $u_k(0,x,y)$, and $u_t^\eps(0,x,y) = 0$. This is because $p^\pm, M^\pm, a_{j,\beta}^\pm, \phi_\beta^\pm$ are constant, and $q^\pm = z \pm c(y)\text{sign}(p^\pm)t$ due to \eqref{IC}.

%{\color{magenta} More on why is it ok to approximate it by GB.}

%%%%%%%%%%%%%%%%%%%%%%%%%%%%%%%%%%%%%%%%%%%%%%%%%%%%%%%%%%%%%%%%%%%%%%%%%%%%%%%%%%%%%%%%%%%%%%%%%%%%%%%%%%%%%%%%%%%%%%%%%%%%%%%%%%%%

\section{One-mode quantity of interest}\label{sec:onefamily}

Before considering the QoI \eqref{QoIbigGB} it is advantageous to first focus on its one-mode counterpart with $u_k$ consisting of either $u_k = u_k^+$ or $u_k = u_k^-$ only,
as given in \eqref{QoIsmall}.
In the present article, 
this is partly due to the fact that the one-mode QoI will be a stepping stone 
for our analysis of the full two-mode QoI.
However, its examination is also important in its own right.
As the two wave modes propagate in opposite directions
they separate and parts of the domain will mainly be covered
by waves belonging to only one of the modes.
As a simple example, in one dimension with constant speed,
%
%
% as
%in certain cases, a one mode solution suffices to approximate the full solution $u^\eps$ in \eqref{waveeq}. 
%For example, in $n=1$ with a constant speed, 
%
the d'Alembert solution to the 
wave equation is a superposition of a left and a right going wave. 
In the general case, the effect is
more pronounced in the high-frequency regime,
when the wave length is significantly smaller than the
curvature of the wave front
\cite{EngRun:03,Runborg07}. Discarding one of the modes
then amounts to discarding reflected waves and waves that initially
propagate away from the domain of interest. 
The solution will nevertheless contain waves
going in different directions.
For example, if $B_1$ in \eqref{waveeq} is chosen such that $u^\eps$ essentially propagates in one direction, then merely one mode, either $u_k^+$ or $u_k^-$, is sufficient to approximate $u^\eps$. 
The approximation is similar to, but not the same as, 
using the paraxial wave equation 
instead of
the full wave equation, which is a common strategy in areas like
seismology, plasma physics, 
underwater acoustics and optics \cite{BamEtAl:88}.
%geophysics where the waves traveling downwards into
%the earth is mainly waves from one and the same mode 
%It is not the case that a one-family solution only propagates in one direction: waves moving in different directions can also be approximated by one GB family. 
%For a more detailed discussion and examples, see Section \ref{sec:whatcouldgowrong}.

%In this section, we will only consider a one-mode GB solution, so either $u_k = u_k^+$ or $u_k = u_k^-$ in \eqref{umode}. It is not important which one we choose 
%as long as the choice is consistent. 
%We 
%and henceforth omit superscripts of all variables.

Let us thus define the GB-approximated version of the QoI in \eqref{QoIsmall},
\be\label{QoIsmallGB}
	\widetilde{\mc Q}^{p,\bs\alpha}_{\text{GB}}(t,\mb y) = \eps^{2(p+|\bs\alpha|)}\int_{\Real^n} g(t,\mb x,\mb y)|\partial_t^p \partial_{\mb x}^{\bs\alpha} u_k(t,\mb x,\mb y)|^2 \psi(t,\mb x) d\mb x,
\ee
with $\psi\in C_c^\infty(\Real\times \Real^n)$ and $g\in C^\infty(\Real\times\Real^n\times \Gamma)$. %As in the previous subsection, we will omit the superscripts everywhere. 
Here 
$u_k = u_k^+$ or $u_k = u_k^-$ in \eqref{umode}. It is not important which one we choose 
%as long as the choice is consistent. 
%We 
and henceforth omit superscripts of all variables.

To introduce the terminology used in this section, we will need the following proposition.
\begin{proposition}\label{prop:admissible}
Assume \ref{ASSUM1}--\ref{ASSUM3} hold. Then for all $T>0$, beam order $k$ and compact $\Gamma_c\subset\Gamma$, there is a GB cutoff width $\eta>0$ and constant $\delta>0$ such that for all $\mb x\in {\mc B}_{2\eta}$,
	\be\label{propphi}
		\Im \Phi_k(t,\mb x,\mb y,\mb z) \geq \delta |\mb x|^2, \quad \forall t\in [0,T], \, \mb y\in \Gamma_c, \, \mb z\in K_0.
	\ee
For the first order GB, $k=1$, we can take $\eta = \infty$ and \eqref{propphi} is valid for all $\mb x\in \Real^n$.
\end{proposition}
\begin{proof}
	Property (P4) in Proposition 1 in \cite{malenova2017stochastic}. The proof is in \cite{malenova2016uncertainty}.
\end{proof}
Note that $\eta$ is the width of the cutoff function $\varrho_\eta$ in \eqref{cutoff} used in the GB superposition \eqref{ukdef}.
\begin{definition}\label{def:admissible}
	The cutoff width $\eta$ used for the GB approximation is called admissible for a given $T$, $k$ and $\Gamma_c$ if it is small enough in the sense of Proposition \ref{prop:admissible}.
\end{definition}
We will prove the following main theorem.

\begin{theorem}\label{th:new}
Assume \ref{ASSUM1}--\ref{ASSUM6} hold and consider a one-mode GB solution. Moreover, let $\eta$ be admissible 
%in the sense of Definition \ref{def:admissible} 
for $T>0$, $k$ and a compact $\Gamma_c\subset \Gamma$. Then for all $p\in \mathbb N$ and $\bs\alpha\in \mathbb N_0^N$, there exist $C_{\bs\sigma}$ such that
\[
	\sup_{\substack{\mb y\in\Gamma_c\\ t\in [0,T]}} \left|\frac{\partial^{\bs\sigma} \widetilde{\mc Q}^{p,\bs\alpha}_{\text{GB}}(t,\mb y)}{\partial \mb y^{\bs\sigma}}\right|\leq C_{\bs\sigma}, \quad \forall \bs\sigma \in \mathbb N_0^N,
\]
where $C_{\bs\sigma}$ is independent of $\eps$ but depends on $T,k$ and $\Gamma_c$.
\end{theorem}
The proof of Theorem \ref{th:new} is presented in Section \ref{sec:th:new}.

%%%%%%%%%%%%%%%%%%%%%%%%%%%%%%%%%%%%%%%%%%%%%%%%%%%%%%%%%%%%%%%%%%%%%%%%%%%%%%%%%

%%%%%%%%%%%%%%%%%%%%%%%%%%%%%%%%%%%%%%%%%%%%%%%%%%%%%%%%%%%%%%%%%%%%%%%%%%%%%%%%%%%%%%%%%%%%%%%%%%%%%%%%%%%%%%%
%\subsection{Known results}
Let us also recall the known results regarding the simplest version of the QoI \eqref{QoIsmallGB}, 
\be\label{QoIsmall0GB}
	\widetilde{\mc Q}_{\text{GB}}:= \widetilde{\mc Q}_{\text{GB}}^{0,\mb 0} = \int_{\Real^n} |u_k(t,\mb x,\mb y)|^2 \psi(t,\mb x) d\mb x,
\ee
which were
obtained in \cite{malenova2017stochastic}. % for the QoI \eqref{QoIsmall0GB}. % under the assumption of \eqref{umode} only consisting of one mode.
%We can prove the following theorem \cite[Theorem 1]{malenova2017stochastic}.
\begin{theorem}[{\cite[Theorem 1]{malenova2017stochastic}}]\label{th:old}
Assume \ref{ASSUM1}--\ref{ASSUM6} hold and consider a one-mode GB solution. Moreover, let $\eta$ be admissible 
%in the sense of Definition \ref{def:admissible} 
for $T>0$, $k$ and a compact $\Gamma_c\subset \Gamma$. Then there exist $C_{\bs\sigma}$ such that
\[
	\sup_{\substack{\mb y\in\Gamma_c\\ t\in[0,T]}} \left|\frac{\partial^{\bs\sigma} \widetilde{\mc Q}_{\text{GB}}(t,\mb y)}{\partial \mb y^{\bs\sigma}}\right|\leq C_{\bs\sigma}, \quad \forall \bs\sigma \in \mathbb N_0^N,
\]
where $C_{\bs\sigma}$ is independent of $\eps$ but depends on $T$, $k$ and $\Gamma_c$.
\end{theorem}
\begin{remark}
This is a minor
generalization of Theorem~1 
 in \cite{malenova2017stochastic}.
 In particular we here allow $\psi$ to also depend on $t$
 and have an estimate that is uniform in $t$. Moreover,
 instead of
assuming $\Gamma$ to be the closure of a bounded open set, as in \cite{malenova2017stochastic}, we consider compact subsets $\Gamma_c$ of an open set $\Gamma$. These modifications do not affect the proof
in a significant way.

\end{remark}

\begin{remark}
One can note that the stochastic regularity in $\mb y$
shown in Theorem~\ref{th:new} also implies stochastic regularity
in $t$ for the same QoI. Indeed, upon defining
$$
   v^\eps(t,\mb x,\mb y,y_0):= u^\eps(ty_0,\mb x,\mb y),
$$
$v^\eps$ will satisfy the same wave equation as
$u^\eps$, with $c(\mb x,\mb y)$ replaced by $y_0c(\mb x,\mb y)$
and $B_1(\mb x,\mb y)$ replaced by $y_0B_1(\mb x,\mb y)$. One can verify that 
with these alterations, the Gaussian beam approximations of $u^\eps$
and $v^\eps$ also satisfy the same equations.
Moreover, for a fixed $t$,
time derivatives of the QoI based on $u^\eps$ corresponds to
partial derivatives in $y_0$ for the QoI based on $v^\eps$,
which is covered by the theory above.
However, making this observation precise, we leave for
future work.
\end{remark}

%%%%%%%%%%%%%%%%%%%%%%%%%%%%%%%%%%%%%%%%%%%%%%%%%%%%%%%%%%%%%%%%%%%%%%%%%%%%%%%%%%%%%%%%%%%%%%%%%%%%%%%%%%%%%%%%%%%%%%%%%
\subsection{Preliminaries}\label{sec:onemodeprel}

In this section we introduce functions spaces and
derive some preliminary results for the
main proof of Theorem \ref{th:new}.
However, we start with a note on the case $\eta=\infty$, which is sometimes an admissible cutoff width in the sense of Proposition~\ref{prop:admissible}. In particular, it is always admissible when $k=1$. It amounts to removing the cutoff functions $\varrho_\eta$ in \eqref{ukdef} altogether. This is convenient in
computations, but there are some technical issues with having
$\eta=\infty$ in the proofs below. We note, however, that, in any finite
time interval $[0,T]$ and compact $\Gamma_c\subset\Gamma$, the Gaussian beam superposition \eqref{umode} with no
cutoff is identical to the one with a large enough cutoff, because of
the compact
support of the test function $\psi(t,\mb x)$. Indeed, suppose 
$\text{supp }\psi(t,\cdot)\subset {\mc B}_{R}$, for $t\in[0,T]$.
Then for $|\mb x|\leq R$ we have
$$
|{\bf x}- {\bf q}(t,\mb y,\mb z)| \leq |{\bf x}|+|{\bf q}(t,\mb y,\mb z)|\leq
R+|{\bf q}(t,\mb y,\mb z)|, \qquad \forall t\in[0,T],\ \forall \mb y\in\Gamma,\ \forall \mb z, \in K_0.
$$
Hence, for $\bar\eta = R+\sup_{t\in[0,T],\mb y\in\Gamma_c, \mb z\in
K_0}|\mb q(t,\mb y,\mb z)|$ we will have
$$
   \psi(t,\mb x)=
   \varrho_{\bar\eta}({\bf x}-{\bf q}(t,\mb y,\mb z))
   \varrho_{\bar\eta}({\bf x}-{\bf q}(t,\mb y,\mb z'))\psi(t,\mb x),\qquad
\forall t\in[0,T],\ \forall \mb y\in\Gamma_c,\ \forall \mb z,\mb z'\in K_0.
$$
We can therefore, without loss of
generality, assume that $\eta<\infty$. 

%We will use the following preliminaries to prove Theorem \ref{th:new}. 
Let us now define a shorthand for the following sets:
\begin{itemize}
%\item $\mc P:= \{P\in \mc X: P(t,\mb x,\mb y,\mb z) = \sum_{|\bs\alpha|=0}^M a_{\bs\alpha}(t,\mb y,\mb z) \: \mb x^{\bs\alpha}, \; \text{where } \: a_{\bs\alpha}\in C^\infty(\Real\times\Gamma\times\Real^n), \: \forall \bs\alpha\}$
%\item 
%$$
%{\mc P}_\mu := 
%\begin{cases}
%\{p\in C^\infty: p(t,\mb x,\mb y,\mb z) = \sum_{|\bs\alpha|=0}^M a_{\bs\alpha}(t,\mb x,\mb y,\mb z) \: \mb x^{\bs\alpha},\:  \text{where} \: a_{\bs\alpha}\in C^\infty, \: \forall \bs\alpha\}, & \mu=\infty,
%\\
%\{p\in \mc P_\infty: \; \text{supp}\, a_{\bs\alpha}(t,\,\cdot\,,\mb y,\mb z)\subset {\mc B}_{2\mu}, \: \forall \bs\alpha,\: t\in \Real,\: \mb y\in \Gamma,\: \mb z\in \Real^n\}& 0<\mu<\infty,
%\end{cases}
%$$.
%
%
\item 
$
\mc P_\mu := \Bigl\{p\in C^\infty: p(t,\mb x,\mb y,\mb z) = \sum_{|\bs\alpha|=0}^M a_{\bs\alpha}(t,\mb x,\mb y,\mb z) \: \mb x^{\bs\alpha},\:  \text{where}\: a_{\bs\alpha}\in C^\infty, \\ 
\text{\hspace{13 mm}and supp}\, a_{\bs\alpha}(t,\,\cdot\,,\mb y,\mb z)\subset {\mc B}_{2\mu}, \: \forall \bs\alpha,\: t\in \Real,\: \mb y\in \Gamma,\: \mb z\in \Real^n\Bigr\},
$

%\item $\mc P_\infty := \{p\in C^\infty: p(t,\mb x,\mb y,\mb z) = \sum_{|\bs\alpha|=0}^M a_{\bs\alpha}(t,\mb x,\mb y,\mb z) \: \mb x^{\bs\alpha},\:  \text{where} \: a_{\bs\alpha}\in C^\infty, \: \forall \bs\alpha\}$,
%
%\item $\mc P_\mu := \{p\in \mc P_\infty: \; \text{supp}\, a_{\bs\alpha}(t,\,\cdot\,,\mb y,\mb z)\subset {\mc B}_{2\mu}, \: \forall \bs\alpha,\: t\in \Real,\: \mb y\in \Gamma,\: \mb z\in \Real^n\}$,

\item $\mc S_\mu := \Bigl\{f\in C^\infty: f(t,\mb x,\mb y,\mb z) = \sum_{j=0}^L \eps^j p_j(t,\mb x,\mb y,\mb z) e^{i\Phi_k(t,\mb x,\mb y,\mb z)/\eps}, \; \text{where}\: p_j\in \mc P_\mu, \: \forall j\Bigr\}$.
%\item $\mc P_\infty := \{P\in \mc X: P(t,\mb x,\mb y,\mb z) = \sum_{|\bs\alpha|=0}^M a_{\bs\alpha}(t,\mb x,\mb y,\mb z) \: \mb x^{\bs\alpha},\:  \text{ where } \: a_{\bs\alpha}\in \mc X, \: \forall \bs\alpha\}$.
%\item $\mc P_\mu := \{P\in \mc P_\infty: \; \text{supp}\, a_{\bs\alpha}(t,\cdot,\mb y,\mb z)\subset {\mc B}_{2\mu}, \: \forall \bs\alpha,\: t\in \Real,\: \mb y\in \Gamma,\: \mb z\in \Real^n\},$ \quad for $0<\mu<\infty $.
%\item $\mc S_\mu := \{f\in \mc X: f(t,\mb x,\mb y,\mb z) = \sum_{j=0}^L \eps^j P_j(t,\mb x,\mb y,\mb z) e^{i\Phi_k(t,\mb x,\mb y,\mb z)/\eps}, \; \text{ where }\: P_j\in \mc P_\mu, \: \forall j\}$.
\end{itemize}
Note that these sets are also defined for $\mu=\infty$,
in which case there is no restriction on the support
of the coefficient functions $a_{\bs\alpha}$
since ${\mc B}_\infty = \Real^n$.
The phase $\Phi_k$ in the definition of $\mc S_\mu$ is as in \eqref{phidef}. %The following lemma proves that $\Phi_k\in \mc P_\infty$ as well as $\partial_t\Phi_k, \partial_{x_\ell}\Phi_k\in \mc P_\infty.$
%\begin{lemma}\label{lemma:PhiinP}
%For the phase $\Phi_k$ in \eqref{phidef}, $\Phi_k, \partial_t\Phi_k, \partial_{x_\ell}\Phi_k\in \mc P_\infty$ for $\ell=1,\ldots n$.
%\end{lemma}
%\begin{proof}
%By \eqref{phidef}, $\Phi_k$
By Proposition \ref{prop:smooth}, it can be written as $\Phi_k(t,\mb x,\mb y,\mb z) = \sum_{|\bs\alpha|=0}^{k+1}d_{\bs\alpha}(t,\mb y,\mb z)\, \mb x^{\bs\alpha}$, with $d_{\bs\alpha}\in C^{\infty}(\Real\times \Gamma\times \Real^n)$ and hence $\Phi_k\in \mc P_\infty$. 
The following properties hold for the sets defined above.
\begin{lemma}\label{lemma:PSprop}
Let $r\in \mc P_\infty$, $p_1, p_2 \in \mc P_\mu$ and $w_1, w_2 \in \mc S_\mu$. Then, for $0<\mu\leq \infty$,
\begin{enumerate}
	\item\label{lem1} $p_1+p_2\in \mc P_\mu$.
	\item\label{lem2} $w_1+w_2\in \mc S_\mu$.
	\item\label{lem3} $r p_1\in \mc P_\mu$.
%\sum_{|\bs\gamma|=0}^{M}c_{\bs\gamma}\,\mb x^{\bs\gamma}\in \mc P_\mu$, for $c_{\bs\gamma}\in \mc X$. %, in particular, $P_1 P_2 \in \mc P$.
	\item\label{lem8}	$r w_1 \in \mc S_\mu$.
%	\item\label{lem4} $\partial_t P_1 \in \mc P_\mu$.
%	\item\label{lem5} \eps \partial_t w_1
	\item\label{lem46} $\partial_s p_1\in \mc P_\mu$, for $s\in \{t,{x_\ell},\: \ell = 1,\ldots n\}$.
	\item\label{lem7} $\eps \partial_{s} w_1 \in \mc S_\mu,$ for $s\in \{t,x_\ell,\: \ell = 1,\ldots n\}$.
\end{enumerate}
\end{lemma}
\begin{proof}
We will denote 
\begin{align*}
	p_m(t,\mb x,\mb y,\mb z) &= \sum_{|\bs\alpha|=0}^{M_m} a_{m,\bs\alpha}(t,\mb x,\mb y,\mb z)\, \mb x^{\bs\alpha}, 
	&& w_m(t,\mb x,\mb y,\mb z) = \sum_{j=0}^{L_m} \eps^j q_{m,j}(t,\mb x,\mb y,\mb z) e^{i\Phi_k(t,\mb x,\mb y,\mb z)/\eps},\\
%P_2(t,\mb x,\mb y,\mb z) &= \sum_{|\bs\beta|=0}^{M_2} b_{\bs\beta}(t,\mb x,\mb y,\mb z)\,\mb x^{\bs\beta}, && w_2(t,\mb x,\mb y,\mb z) = \sum_{\ell=0}^{L_2} \eps^\ell Q_\ell(t,\mb x,\mb y,\mb z) e^{i\Phi_k(t,\mb x,\mb y,\mb z)/\eps}\\
r(t,\mb x,\mb y,\mb z) &= \sum_{|\bs\gamma|=0}^M c_{\bs\gamma}(t,\mb x,\mb y,\mb z) \: \mb x^{\bs\gamma},&& m \in \{1,2\}.
\end{align*}
Let us assume without loss of generality that $M_2\geq M_1$ and $L_2\geq L_1$.
	\begin{enumerate}
			\item  The sum $p_1+p_2$ can be rewritten as 
			$p_1+p_2 = \sum_{|\bs\beta|=0}^{M_2} b_{\bs\beta}(t,\mb x,\mb y,\mb z)\, \mb x^{\bs\beta}$, where $b_{\bs\beta}$ is such that
\[
	b_{\bs\beta} = \left\{
				\begin{array}{ll}
						a_{1,\bs\beta} + a_{2,\bs\beta}, & \text{for } \, |\bs\beta|\leq M_1, \\
						a_{2,\bs\beta}, & \text{for } \, M_1<|\bs\beta|\leq M_2.
				\end{array}\right.
\]
Hence $b_{\bs\beta}\in C^\infty$ and $\text{supp } b_{\bs\beta}(t,\cdot,\mb y,\mb z)\subset {\mc B}_{\mu}$, for all $t\in \Real, \: \mb y\in \Gamma, \: \mb z\in \Real^n$. Therefore $p_1+p_2\in \mc P_\mu$.
\item The sum $w_1+w_2$ can be rewritten as $w_1+w_2 = \sum_{j=0}^{L_2} \eps^j q_j(t,\mb x,\mb y,\mb z) e^{i\Phi_k(t,\mb x,\mb y,\mb z)/\eps}$, where $q_{j}$ is such that
\[
	q_{j} = \left\{
				\begin{array}{ll}
						q_{1,j} + q_{2,j}, & \text{for } \, j\leq L_1, \\
						q_{2,j}, & \text{for } \, L_1<j\leq L_2.
				\end{array}\right.
\]
By point \ref{lem1} we have that $q_j\in \mc P_\mu$ for all $j$ and therefore $w_1+w_2\in \mc S_\mu$.

\item We have
\begin{align*}
	r(t,\mb x,\mb y,\mb z) p_1(t,\mb x,\mb y,\mb z) &= \sum_{|\bs\gamma|=0}^{M} c_{\bs\gamma}(t,\mb x,\mb y,\mb z) \mb x^{\bs\gamma}\sum_{|\bs\alpha|=0}^{M_1} a_{1,\bs\alpha}(t,\mb x,\mb y,\mb z) \mb x^{\bs\alpha}  \\
	&= \sum_{|\bs\delta| = 0}^{M_1+M} d_{\bs\delta}(t,\mb x,\mb y,\mb z) \mb x^{\bs\delta},
\end{align*}
where $d_{\bs\delta} = \sum_{\bs\alpha+\bs\gamma = \bs\delta} a_{1,\bs\alpha}c_{\bs\gamma}\in C^\infty$.
Since $\text{supp} \, a_{1,\bs\alpha}(t,\cdot,\mb y,\mb z)\subset {\mc B}_{\mu}$, we also have \break $\text{supp}\, d_{\bs\delta}(t,\cdot,\mb y,\mb z)\subset {\mc B}_{\mu}$ for all $t\in \Real,\: \mb y\in \Gamma, \: \mb z\in \Real^n$ and therefore $rp_1\in \mc P_\mu$.

\item We have
\[
r(t,\mb x,\mb y,\mb z) w_1(t,\mb x,\mb y,\mb z) = \sum_{j=0}^{L_1} \eps^j r(t,\mb x,\mb y,\mb z)q_{1,j}(t,\mb x,\mb y,\mb z) e^{i\Phi_k(t,\mb x,\mb y,\mb z)/\eps},
\]
where $rq_{1,j}\in \mc P_\mu$ by point \ref{lem3} for all $j$. Therefore $r w_1\in \mc S_\mu$.

\item The time derivative of $p_1$ reads	
$\partial_t p_1(t,\mb x,\mb y,\mb z) = \sum_{|\bs\alpha|=0}^{M_1} \partial_t  a_{1,\bs\alpha}(t,\mb x,\mb y,\mb z) \: \mb x^{\bs\alpha},$ % = \sum_{|\bs\alpha|=0}^{M_1} b_{\bs\alpha}(t,\mb x,\mb y,\mb z) \mb x^{\bs\alpha},
%for $b_{\bs\alpha} = \partial_t a_{\bs\alpha}$ 
and since $\text{supp}\, \partial_t  a_{1,\bs\alpha}(t,\cdot,\mb y,\mb z)\subset {\mc B}_{\mu}$ for all $t\in \Real,\: \mb y\in \Gamma, \: \mb z\in \Real^n$, we have $\partial_t p_1\in \mc P_\mu$.
%\begin{comment}
%\item The time derivative of $w_1$ reads
%\begin{align*}
%	\partial_t w_1(t,\mb x,\mb y,\mb z) &= \underbrace{\sum_{j=0}^{L_1} \eps^j \partial_t P_j(t,\mb x,\mb y,\mb z)e^{i\Phi_k(t,\mb x,\mb y,\mb z)/\eps}}_{\textcircled{\raisebox{-0.9pt}{1}}} \\
%&+ \underbrace{\sum_{j=0}^{L_1} P_j(t,\mb x,\mb y,\mb z) i\eps^{j-1} \partial_t \Phi_k(t,\mb x,\mb y,\mb z) e^{i\Phi_k(t,\mb x,\mb y,\mb z)/\eps}}_{\textcircled{\raisebox{-0.9pt}{2}}}.
%\end{align*}
%The term $\eps{\textcircled{\raisebox{-0.9pt}{1}}}$ reads
%\[
%	\eps{\textcircled{\raisebox{-0.9pt}{1}}} = \sum_{j=0}^{L_1+1} \eps^j Q_j(t,\mb x,\mb y,\mb z) e^{i\Phi_k(t,\mb x,\mb y,\mb z)/\eps},
%\]
%where 
%\[
%	Q_j = \left\{ \begin{array}{ll}
%				0, & \text{for } j=0,\\
%
%				\partial_t P_{j-1}, & \text{otherwise.}
%		\end{array}\right.
%\]
%Since $\partial_t P_j\in \mc P_\mu$, $\forall p$, by point \ref{lem4}, we obtain $\eps {\textcircled{\raisebox{-0.9pt}{1}}}\in \mc S_\mu$. 
%
%\ldots
%
% Then
%\[
%	\eps{\textcircled{\raisebox{-0.9pt}{2}}} = \sum_{j=0}^{L_1} \eps^{j} P_j(t,\mb x,\mb y,\mb z) i \sum_{|\bs\alpha|=0}^{k+1}d_{\bs\alpha}(t,\mb y,\mb z)\mb x^{\bs\alpha} e^{i\Phi_k(t,\mb x,\mb y,\mb z)/\eps}= \sum_{j=0}^{L_1} \eps^{j} R_j(t,\mb x,\mb y,\mb z) e^{i\Phi_k(t,\mb x,\mb y,\mb z)/\eps},
%\]
%with $R_j = i P_j \sum_{|\bs\alpha|=0}^{k+1}d_{\bs\alpha} \mb x^{\bs\alpha}\in \mc P_\mu$ by point \ref{lem3} and hence $\eps {\textcircled{\raisebox{-0.9pt}{2}}}\in \mc S_\mu.$
%This implies that $\eps \partial_t w_1 = \eps(\textcircled{\raisebox{-0.9pt}{1}}+\textcircled{\raisebox{-0.9pt}{2}})\in \mc S_\mu$.
%\end{comment}
Secondly, the derivative of $p_1$ with respect to $x_\ell$ reads
\[
	\partial_{x_\ell} p_1(t,\mb x,\mb y,\mb z) = \underbrace{\sum_{|\bs \alpha|=0}^{M_1} \partial_{x_\ell} a_{1,\bs\alpha}(t,\mb x,\mb y,\mb z) \, \mb x^{\bs\alpha}}_{\textcircled{\raisebox{-0.9pt}{1}}} + \underbrace{\sum_{|\bs\alpha|=0}^{M_1} a_{1,\bs\alpha} (t,\mb x,\mb y,\mb z) \alpha_\ell \,\mb x^{\bs\alpha-\mb e_\ell}}_{\textcircled{\raisebox{-0.9pt}{2}}}.
\]
Since $\text{supp}\,  \partial_{x_\ell} a_{1,\bs\alpha}(t,\cdot,\mb y,\mb z)\subset {\mc B}_{\mu}$ for all $t\in \Real,\: \mb y\in \Gamma, \: \mb z\in \Real^n$, we have ${\textcircled{\raisebox{-0.9pt}{1}}}\in \mc P_\mu$. For ${\textcircled{\raisebox{-0.9pt}{2}}}$, there exist $c_{\bs\gamma}\in C^\infty$ such that ${\textcircled{\raisebox{-0.9pt}{2}}} = \sum_{|\bs\gamma|=0}^{M_1-1} c_{\bs\gamma} (t,\mb x,\mb y,\mb z) \, \mb x^{\bs\gamma}$
with $\text{supp}\, c_{\bs\gamma}(t,\cdot,\mb y,\mb z)\subset {\mc B}_{\mu}$ for all $t\in \Real,\: \mb y\in \Gamma, \: \mb z\in \Real^n$ and hence $\textcircled{\raisebox{-0.9pt}{2}}\in \mc P_\mu$. By point \ref{lem1}, $\partial_{x_\ell} p_1 = {\textcircled{\raisebox{-0.9pt}{1}}}+{\textcircled{\raisebox{-0.9pt}{2}}}\in \mc P_\mu$.

\item The derivative $\partial_{s} w_1$ with respect to either of $s\in \{t,x_\ell, \: \ell=1,\ldots n\}$ reads
\begin{align*}
	\partial_{s} & w_1(t,\mb x,\mb y,\mb z) \\ &= \underbrace{\sum_{j=0}^{L_1} \eps^j \partial_{s} q_{1,j}(t,\mb x,\mb y,\mb z)e^{i\Phi_k(t,\mb x,\mb y,\mb z)/\eps}}_{\textcircled{\raisebox{-0.9pt}{1}}}\\
	&\ \ \ \ \ + \underbrace{\sum_{j=0}^{L_1}i\eps^{j-1} \partial_{s} \Phi_k(t,\mb x,\mb y,\mb z) q_{1,j}(t,\mb x,\mb y,\mb z) e^{i\Phi_k(t,\mb x,\mb y,\mb z)/\eps}}_{\textcircled{\raisebox{-0.9pt}{2}}}.
\end{align*}
We have $\eps {\textcircled{\raisebox{-0.9pt}{1}}} = \sum_{j=0}^{L_1+1}\eps^j q_j(t,\mb x,\mb y,\mb z)e^{i\Phi_k(t,\mb x,\mb y,\mb z)/\eps}$,
with
\[
	q_j = \left\{ \begin{array}{ll}
				0, & \text{for } j=0,\\
				\partial_{s} q_{1,j-1}, & \text{otherwise.}
		\end{array}\right.
\]
By point \ref{lem46}, $q_j\in \mc P_\mu$, and we therefore obtain $\eps \textcircled{\raisebox{-0.9pt}{1}}\in \mc S_\mu$. 
Since $\Phi_k\in \mc P_{\infty}$, we have by point \ref{lem46} that $\partial_s \Phi_k\in \mc P_\infty$ and therefore $\eps \textcircled{\raisebox{-0.9pt}{2}}\in \mc S_\mu$ by point \ref{lem8}. By point \ref{lem2}, we finally arrive at $\eps \partial_{s}w_1 = \eps \textcircled{\raisebox{-0.9pt}{1}} + \eps \textcircled{\raisebox{-0.9pt}{2}}\in \mc S_\mu$.
	\end{enumerate}
\end{proof}
As a consequence, we obtain the following corollary.
\begin{corollary}\label{cor:mixedder}
If $w \in \mc S_\mu$, all scaled mixed derivatives $\eps^{p+|\bs\alpha|}\partial_t^p \partial_{\mb x}^{\bs\alpha} w \in \mc S_\mu.$
\end{corollary}
\begin{proof}
Apply point \ref{lem7} of Lemma \ref{lemma:PSprop} repeatedly.
\end{proof}

%%%%%%%%%%%%%%%%%%%%%%%%%%%%%%%%%%%%%%%%%%%%%%%%%%%%%%%%%%%%%%%%%%%%%%%%%%%%%%%%%%%%%%%%%%%%%%%%%%%%%%%%%%%%%%%%%%%%%%%%%
\subsection{Proof of theorem \ref{th:new}}\label{sec:th:new}% {\color{orange} (just $p=1$ for now)}}
%\begin{itemize}
%\item
%$\Lambda_\mu = \Lambda_\mu(t,\mb y,\mb z,\mb z') := \{\mb x\in \Real^n: |\mb x-\mb q(t,\mb y,\mb z)|\leq 2\mu \quad \text{and } \quad |\mb x-\mb q(t,\mb y,\mb z')|\leq 2\mu\}$.
%\item $\mc T_\mu := \{f\in \mc Y: \text{supp } f(t,\cdot,\mb y,\mb z,\mb z')\subset \Lambda_\mu(t,\mb y,\mb z,\mb z'), \: \forall t\in \Real, \, \mb y\in \Gamma, \, \mb z,\mb z'\in \Real^n\}.$
%\end{itemize}
%Moreover, we define $\Lambda_\infty = \Real^n$ and $\mc T_\infty = \mc Y$. 
%Now on we fix $\eta$ so that it is admissible in the sense of Definition \ref{def:admissible}. 
%First assuming \ref{ASSUM6}, l
The QoI \eqref{QoIsmallGB} can be written
\begin{align}
 \widetilde{\mc Q}^{p,\bs\alpha}_{\text{GB}}(t,\mb y) &= \eps^{2(p+|\bs\alpha|)} \int_{\Real^n}  g(t,\mb x,\mb y)\partial_t^p \partial_{\mb x}^{\bs\alpha}u_k(t,\mb x,\mb y)^* \partial_t^p \partial_{\mb x}^{\bs\alpha}u_k(t,\mb x,\mb y) \psi(t,\mb x) d\mb x \nonumber \\
 & = \left(\frac{1}{2\pi\varepsilon}\right)^{n} \int_{K_0 \times K_0} I(t,\mb y,\mb z,\mb z') \, d\mb z \, d\mb z',\label{QItemp}
\end{align}
where
\begin{align}
 I(t,\mb y,\mb z,\mb z') = \eps^{2(p+|\bs\alpha|)}  \int_{\Real^n} & \partial_t^p \partial_{\mb x}^{\bs\alpha} (w_k(t,\mb x-\mb q(t,\mb y,\mb z),\mb y,\mb z))^* \partial_t^p \partial_{\mb x}^{\bs\alpha} (w_k(t,\mb x-\mb q(t,\mb y,\mb z'),\mb y,\mb z'))\nonumber \\
 &\times g(t,\mb x,\mb y) \psi(t,\mb x)\, d\mb x,\label{tt}
\end{align}
and
\be\label{wkdef}
w_k(t,\mb x,\mb y,\mb z) = A_{k}(t,\mb x,\mb y,\mb z) \varrho_\eta(\mb x) e^{i\Phi_k(t,\mb x,\mb y,\mb z)/\eps}.
\ee
The following lemma allows us to rewrite $I$ in \eqref{tt} %$\eps^{p+|\bs\alpha|} \partial_t^p \partial_{\mb x}^{\bs\alpha} w_k(t,\mb x-\mb q(t,\mb y,\mb z),\mb y,\mb z))$ 
in terms of functions belonging to
$\mc S_{\eta}$.
\begin{lemma}\label{lemma:partialwk}
Let $w_k$ be as in \eqref{wkdef}. Then for each $k\geq 1$, $p\geq 0$, $\bs\alpha\in \mathbb N_0^N$, there exists $s_k\in \mc S_{\eta}$ such that
	\[
		\eps^{p+|\bs\alpha|} \partial_t^p\partial_{\mb x}^{\bs\alpha} (w_k(t,\mb x-\mb q(t,\mb y,\mb z),\mb y,\mb z)) = s_k(t,\mb x-\mb q(t,\mb y,\mb z),\mb y,\mb z).
	\]
\end{lemma}
\begin{proof}
We note that from \eqref{Adef1},% and \eqref{Adef2},
\[
	w_k(t,\mb x,\mb y,\mb z) = \sum_{j=0}^{\lceil \frac{k}{2}\rceil-1} \eps^j \sum_{|\bs\beta|=0}^{k-2j-1}\frac{1}{\bs\beta!} a_{j,\bs\beta}(t,\mb y,\mb z) \varrho_\eta(\mb x)\, \mb x^{\bs\beta}e^{i\Phi_k(t,\mb x,\mb y,\mb z)/\eps},
\]
and since $\varrho_\eta$ is supported in ${\mc B}_{2\eta}$ then %for $k>1$, then $w_k\in \mc S_{\eta}$. For $k=1$, we have $\eta = \infty$ and $\varrho_\infty \equiv 1$ and hence $w_1 \in \mc S_\infty$. We conclude
%therefore that 
$w_k\in \mc S_{\eta}$. % for all $k\geq 1$.
We first differentiate 
$$\partial_{\mb x}^{\bs\alpha} (w_k(t,\mb x-\mb q(t,\mb y,\mb z),\mb y,\mb z)) = \partial_{\mb x}^{\bs\alpha} w_k(t,\mb x,\mb y,\mb z)\big|_{\mb x= \mb x -\mb q(t,\mb y,\mb z)},
$$ and note that by Corollary \ref{cor:mixedder},	
$
r_k:= \eps^{|\bs\alpha|} \partial_{\mb x}^{\bs\alpha} w_k \in \mc S_{\eta}.
$
Furthermore, the time derivative of $r_k(t,\mb x-\mb q(t,\mb y,\mb z),\mb y,\mb z)$ reads
\[
	\partial_t \left(r_k(t,\mb x-\mb q(t,\mb y,\mb z),\mb y,\mb z)\right) = \partial_t r_k(t,\mb x,\mb y,\mb z) -\partial_t \mb q(t,\mb y,\mb z) \cdot \nabla_{\mb x} r_k(t,\mb x,\mb y,\mb z)\Big|_{\mb x = \mb x-\mb q(t,\mb y,\mb z)}.
\]
From points \ref{lem2}, \ref{lem8} and \ref{lem7} in Lemma \ref{lemma:PSprop} and Proposition \ref{prop:smooth}, we have that $F r_k\in \mc S_{\eta}$, where $F$ is the operator $F = \eps (\partial_t -\partial_t \mb q \cdot \nabla_{\mb x})$. Repeated differentiation of $r_k(t,\mb x-\mb q(t,\mb y,\mb z),\mb y,\mb z)$ subject to an appropriate scaling with $\eps$ thus yields repeated application of the $F$ operator:
\[
\eps^p \partial_t^p \left(r_k(t,\mb x-\mb q(t,\mb y,\mb z),\mb y,\mb z)\right) = F^p r_k(t,\mb x,\mb y,\mb z)\Big|_{\mb x = \mb x-\mb q(t,\mb y,\mb z)}.
\]
Since $s_k := F^p r_k \in \mc S_{\eta}$ the proof is complete.
\end{proof}
The function $s_k\in \mc S_{\eta}$ can be rewritten recalling the definition of $\mc S_\eta$ as $s_k(t,\mb x,\mb y,\mb z) = \sum_{j=0}^L \eps^j p_j(t,\mb x,\mb y,\mb z) e^{i\Phi_k(t,\mb x,\mb y,\mb z)/\eps}$, with $p_j\in \mc P_\eta$, for all $j$. Then using Lemma~\ref{lemma:partialwk}, the quantity \eqref{tt} becomes
\begin{align*}
	I(t,\mb y,& \mb z,\mb z') = \int_{\Real^n} s_k^*(t,\mb x-\mb q(t,\mb y,\mb z),\mb y,\mb z) s_k(t,\mb x-\mb q(t,\mb y,\mb z'),\mb y,\mb z')g(t,\mb x,\mb y)\psi(t,\mb x)\, d\mb x \nonumber \\
	&= \sum_{j,\ell=0}^{L} \eps^{j+\ell} \int_{\Real^n} h_{j\ell} (t,\mb x,\mb y,\mb z,\mb z') e^{i\Theta_k(t,\mb x,\mb y,\mb z,\mb z')/\eps} \, d\mb x, %\label{tt2}
\end{align*}
where $\Theta_k$ is the $k$-th order GB phase
\be\label{Thetakdef}
	\Theta_k(t,\mb x,\mb y,\mb z,\mb z') = \Phi_k(t,\mb x-\mb q(t,\mb y,\mb z'),\mb y,\mb z')-\Phi_k^*(t,\mb x-\mb q(t,\mb y,\mb z),\mb y,\mb z),
\ee
and
\[
 h_{j\ell}(t,\mb x,\mb y,\mb z,\mb z') = p_j^*(t,\mb x-\mb q(t,\mb y,\mb z),\mb y,\mb z) \, p_\ell(t,\mb x-\mb q(t,\mb y,\mb z'),\mb y,\mb z')g(t,\mb x,\mb y)\psi(t,\mb x).
\]
Let us use the definition of $\mc P_{\eta}$ and write
$p_j(t,\mb x,\mb y,\mb z) = \sum_{|\bs\alpha|=0}^{M} a_{j,\bs\alpha}(t,\mb x,\mb y,\mb z) \, \mb x^{\bs\alpha}$, 
with \break $\text{supp}\, a_{j,\bs\alpha}(t,\cdot,\mb y,\mb z)\subset {\mc B}_{2\eta}$ for all $j,\bs\alpha,\: t\in \Real, \: \mb y\in \Gamma, \: \mb z\in\Real^n$.
We get %for a fixed $j$ and $\ell$,
\[%begin{align*}
	%P_j^*(t,\mb x-\mb q(t,\mb y,\mb z),\mb y,\mb z)\,& P_\ell(t,\mb x-\mb q(t,\mb y,\mb z'),\mb y,\mb z') \psi(t,\mb x)\\
h_{j\ell}(t,\mb x,\mb y,\mb z,\mb z') = \sum_{|\bs\alpha|,|\bs\beta|=0}^{M}c_{j,\ell,\bs\alpha,\bs\beta}(t,\mb x,\mb y,\mb z,\mb z') (\mb x-\mb q(t,\mb y,\mb z))^{\bs\alpha}(\mb x-\mb q(t,\mb y,\mb z'))^{\bs\beta},
\]%end{align*}
where $c_{j,\ell,\bs\alpha,\bs\beta}(t,\mb x,\mb y,\mb z,\mb z') = a_{j,\bs\alpha}^*(t,\mb x-\mb q(t,\mb y,\mb z),\mb y,\mb z)a_{\ell,\bs{\beta}}(t,\mb x-\mb q(t,\mb y,\mb z'),\mb y,\mb z')g(t,\mb x,\mb y)\psi(t,\mb x)$ implying that 
%$\text{supp}\, 
%c_{j,\ell,\bs\alpha,\bs\beta}\subset \Lambda_\mu(t,\mb y,\mb z,\mb z')$,
%where
$\text{supp}\, c_{j,\ell,\bs\alpha,\bs\beta}(t,\cdot,\mb y,\mb z,\mb z')\subset \Lambda_{\eta}(t,\mb y,\mb z,\mb z')$,
given by
$$
\Lambda_\eta(t,\mb y,\mb z,\mb z') := \{\mb x\in \Real^n: |\mb x-\mb q(t,\mb y,\mb z)|\leq 2\eta\:\text{ and }\:|\mb x-\mb q(t,\mb y,\mb z')|\leq 2\eta\}.
$$
%
% for all $t,\mb y,\mb z,\mb z'$ and $c_{j,\ell,\bs\alpha,\bs\beta}\in \mc Y$. 
%For $\eta = \infty$, $a_{j,\bs\alpha}\in \mc X$ and hence $c_{j,\ell,\bs\alpha,\bs\beta}\in \mc Y = \mc T_\infty$.
To summarize, the quantity \eqref{tt} can be written as
\[
	I(t,\mb y,\mb z,\mb z') = \sum_{j,\ell=0}^{L} \eps^{j+\ell} \sum_{|\bs\alpha|,|\bs\beta|=0}^{M} I_{j,\ell,\bs\alpha,\bs\beta}(t,\mb y,\mb z,\mb z'),
\]
with 
\[
I_{j,\ell,\bs\alpha,\bs\beta}(t,\mb y,\mb z,\mb z')= \int_{\Real^n} c_{j,\ell,\bs\alpha,\bs\beta}(t,\mb x,\mb y,\mb z,\mb z')(\mb x-\mb q(t,\mb y,\mb z))^{\bs\alpha}(\mb x-\mb q(t,\mb y,\mb z'))^{\bs\beta} e^{i\Theta_k(t,\mb x,\mb y,\mb z,\mb z')/\eps} \, d\mb x,
\]
such that $c_{j,\ell,\bs\alpha,\bs\beta}\in \mathcal T_\eta$, where 
%We note at this point that exchanging \ref{ASSUM6} for \ref{ASSUM6p}' preserves the form of $I$ and the property $c_{j,\ell,\bs\alpha,\bs\beta}\in \mathcal T_\eta$.
%
%
%We will consider the following set for $\mu<\infty$ in the course of the proof: %\label{def:Omega}
\begin{align*}
\mc T_\eta &:= \Bigl\{f\in C^\infty: \text{supp } f(t,\cdot,\mb y,\mb z,\mb z')\subset \Lambda_\eta(t,\mb y,\mb z,\mb z'), \: \forall t\in \Real, \, \mb y\in \Gamma, \, \mb z,\mb z'\in \Real^n\Bigr\}.
%&\ \ \ \text{where}\:
%\Lambda_\mu(t,\mb y,\mb z,\mb z') := \{\mb x\in \Real^n: |\mb x-\mb q(t,\mb y,\mb z)|\leq 2\mu\:\text{ and }\:|\mb x-\mb q(t,\mb y,\mb z')|\leq 2\mu\}.
\end{align*}
We will now utilize the following theorem.
\begin{theorem}\label{th:oldmodified}
Assume \ref{ASSUM1}--\ref{ASSUM6} hold. Let $\eta<\infty$ be admissible
for $T>0$, $k$ and a compact $\Gamma_c\subset \Gamma$.
Define
\be\label{I0}
	I_0(t,\mb y,\mb z,\mb z') = \int_{\Real^n} f(t,\mb x,\mb y,\mb z,\mb z') (\mb x-\mb q(t,\mb y,\mb z))^{\bs\alpha}(\mb x-\mb q(t,\mb y,\mb z'))^{\bs\beta} e^{i\Theta_k(t,\mb x,\mb y,\mb z,\mb z')/\eps}d\mb x,
\ee
where $\Theta_k$ is as in \eqref{Thetakdef} and $f\in \mc T_\eta$. %$g\in \mc Y$ with $\text{supp}\, g(t,\cdot,\mb y,\mb z,\mb z')\subset \Lambda_{\eta}(t,\mb y,\mb z,\mb z')$ for all $t\in \Real,\: \mb y\in \Gamma, \:\mb z\in \Real^n, \: \mb z'\in \Real^n$. 
Then there exist $C_{\bs\sigma,\bs\alpha,\bs\beta}$ such that
\[
	\sup_{\substack{\mb y\in \Gamma_c\\ t\in [0,T]}}\left(\frac{1}{2\pi\eps}\right)^{n} \int_{K_0\times K_0} \left|\partial_{\mb y}^{\bs \sigma} I_0(t,\mb y,\mb z,\mb z')\right| d\mb z \, d\mb z'\leq C_{\bs\sigma,\bs\alpha,\bs\beta},
\]
for all $\bs\sigma\in \mathbb N_0^N$ and $\bs\alpha,\bs\beta\in \mathbb N_0^n$, where $C_{\bs\sigma,\bs\alpha,\bs\beta}$ is independent of $\eps$ but depends on $T$, $k$ and $\Gamma_c$.%$k\geq \max(|\bs\alpha|,|\bs\beta|)+1$.
\end{theorem}
\begin{proof}
The proof is essentially the same as the proof of Theorem 1 in \cite{malenova2017stochastic}. We include shortened version 
in the Appendix.
\end{proof}
Since $I_{j,\ell,\bs\alpha,\bs\beta}$ is of the form \eqref{I0}, we can use Theorem \ref{th:oldmodified} (replacing the constant $C_{\bs\sigma,\bs\alpha,\bs\beta}$ with $C_{\bs\sigma,j,\ell,\bs\alpha,\bs\beta}$ to illustrate its dependence on $j$ and $\ell$ as well). Then recalling \eqref{QItemp} and \ref{ASSUM5} we get
\begin{align*}
	\sup_{\substack{\mb y\in \Gamma_c\\ t\in[0,T]}}\left|\frac{\partial^{\bs\sigma} \widetilde{\mc Q}^{p,\bs\alpha}_{\text{GB}}(t,\mb y)}{\partial \mb y^{\bs\sigma}}\right|&\leq \sup_{\substack{\mb y\in \Gamma_c\\ t\in[0,T]}}\left(\frac{1}{2\pi\eps}\right)^n \int_{K_0\times K_0} \left|\frac{\partial^{\bs\sigma} I(t,\mb y,\mb z,\mb z')}{\partial \mb y^{\bs\sigma}}\right|\, d\mb z\, d\mb z' \\
&\leq \sup_{\substack{\mb y\in \Gamma_c\\t\in[0,T]}}\left(\frac{1}{2\pi\eps}\right)^n \sum_{j,\ell=0}^{L} \eps^{j+\ell} \sum_{|\bs\alpha|,|\bs\beta|=0}^{M} \int_{K_0\times K_0} \left|\frac{\partial^{\bs\sigma} I_{j,\ell,\bs\alpha,\bs\beta}(t,\mb y,\mb z,\mb z')}{\partial \mb y^{\bs\sigma}}\right|\, d\mb z\, d\mb z'\\
&\leq \tilde C \sup_{j,\ell,\bs\alpha,\bs\beta} C_{\bs\sigma,j,\ell,\bs\alpha,\bs\beta}
\leq C_{\bs\sigma},
\end{align*}
where $C_{\bs\sigma}$ depends on $\eta, T, k, \Gamma_c, L,M$, but is independent of $\eps$, for all $\bs\sigma\in \mathbb N_0^N$. This concludes the proof of Theorem \ref{th:new}.

%%%%%%%%%%%%%%%%%%%%%%%%%%%%%%%%%%%%%%%%%%%%%%%%%%%%%%%%%%%%%%%%%%%%%%%%%%%%%%%%%%%%%%%%%%%%%%%%%%%%%%%%
\section{Two-mode quantity of interest}\label{sec:twofamily}

Let us consider a wave composed of both forward and backward propagating modes as defined in \eqref{umode}.
In this case, Theorem \ref{th:new} for the QoI \eqref{QoIsmallGB} is no longer necessarily true. In fact, $\widetilde{\mc Q}^{p,\bs\alpha}_{\text{GB}}$
can be highly oscillatory. We will therefore have to look at a slightly different QoI where
the averaging is also done in time, not just in space.

 %%%%%%%%%%%%%%%%%%%%%%%%%%%%%%%%%%%%%%%%%%%%%%%%%%%%%%%%%%%%%%%%%%%%%%%%%%%%%%%%%%%%%%%%%%%%%%%%%%%%%%%%%%%%%%%%%
\subsection{What could go wrong?}\label{sec:whatcouldgowrong}
Since $\widetilde {\mathcal Q}_{\text{GB}}$ in \eqref{QoIsmallGB} is a good approximation of $\widetilde{\mathcal Q}$ in \eqref{QoIsmall}, it is oscillatory if and only if the other one is,
and we will first show a simple example where $\widetilde{\mc Q}$ in \eqref{QoIsmall0} is oscillatory.
    	\begin{figure}
            \centering
                \begin{minipage}[b]{0.32\textwidth}
                    \includegraphics[width=\textwidth]{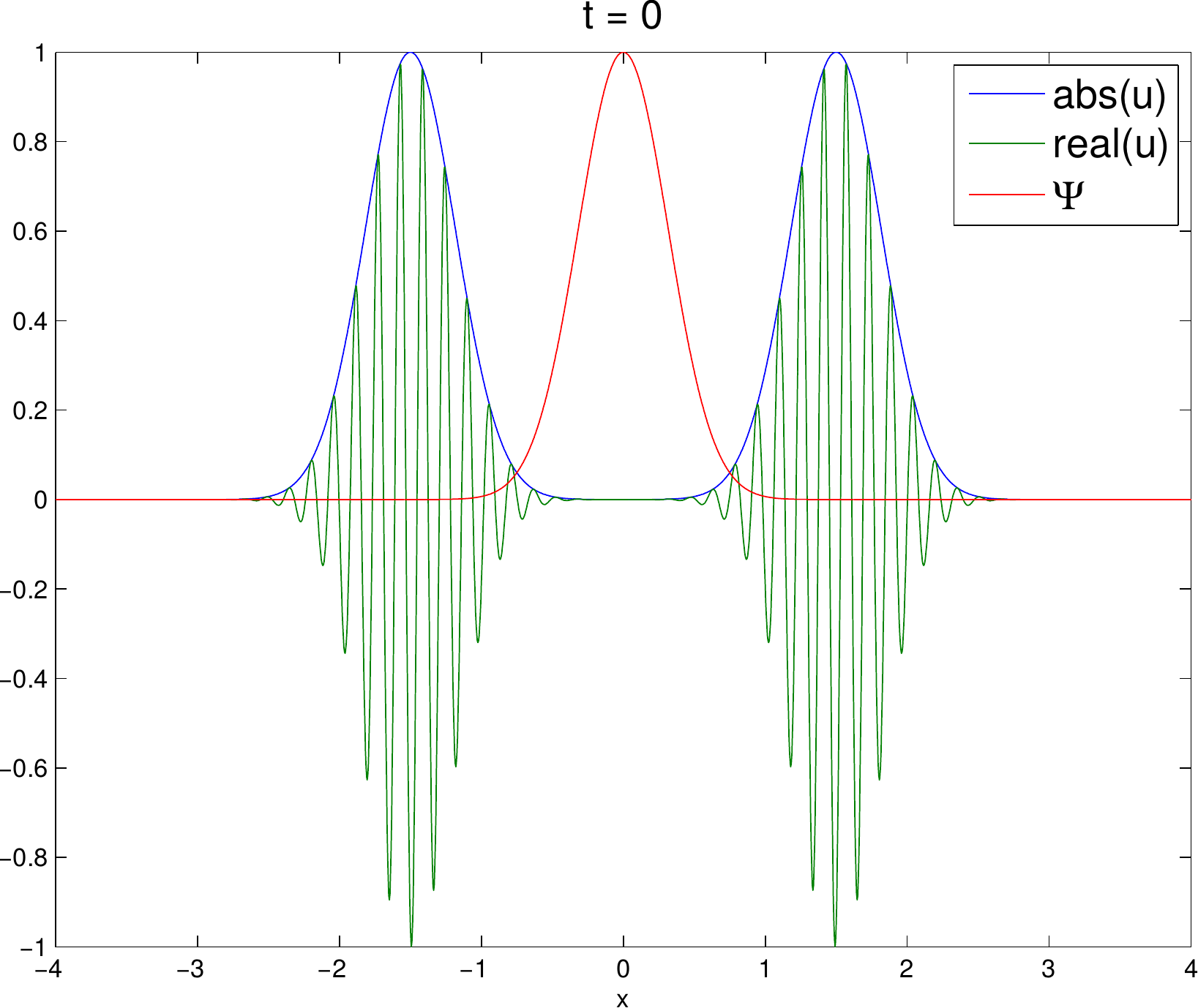}
                \end{minipage}
                \begin{minipage}[b]{0.32\textwidth}
                    \includegraphics[width=\textwidth]{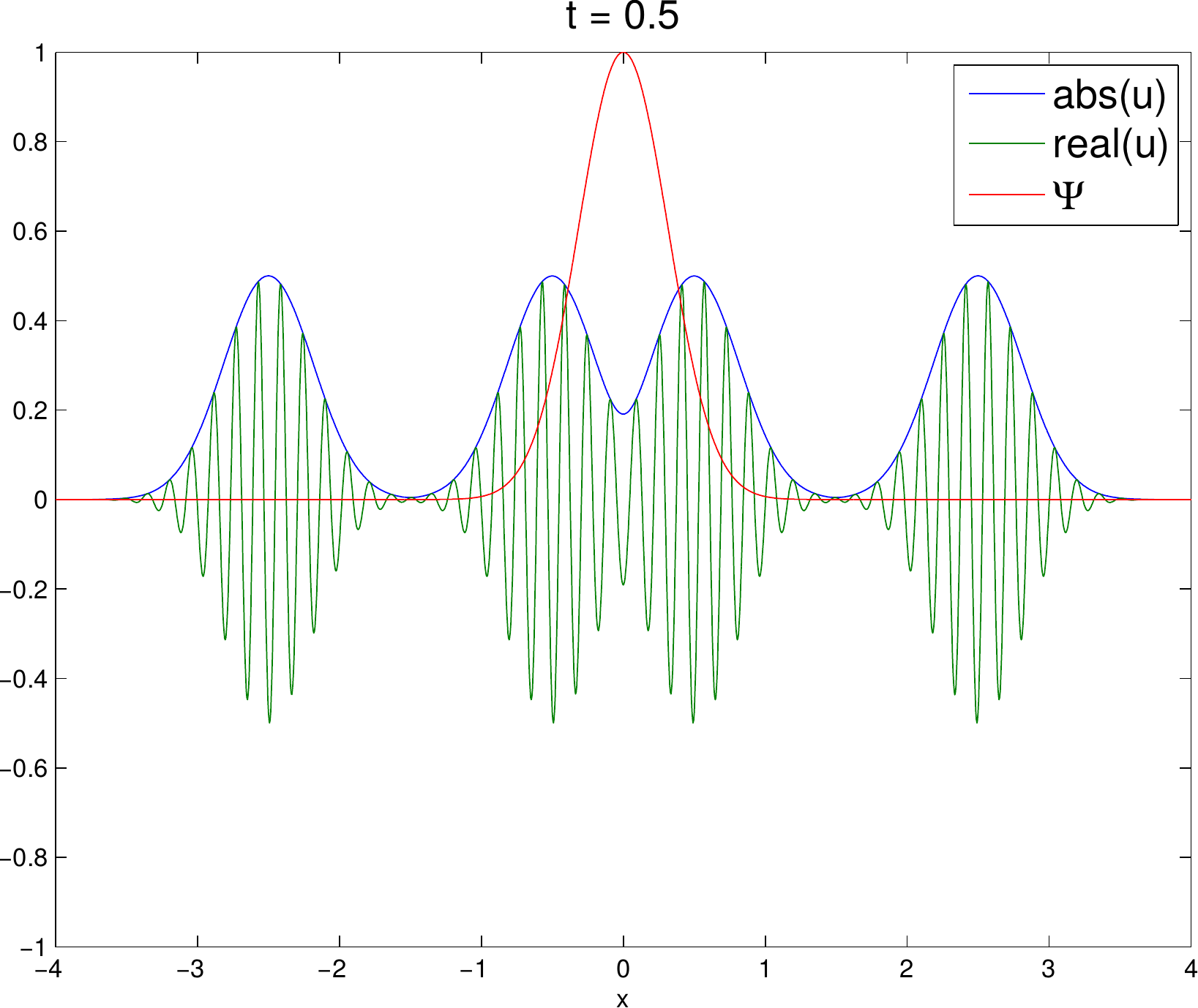}
                \end{minipage}
                %\begin{minipage}[b]{0.24\textwidth}
                %    \includegraphics[width=\textwidth]{figures/sol0_75.pdf}
                %\end{minipage}
                %\begin{minipage}[b]{0.3\textwidth}
                 %   \includegraphics[width=\textwidth]{figures/sol1.pdf}
                %\end{minipage}
                \begin{minipage}[b]{0.32\textwidth}
                    \includegraphics[width=\textwidth]{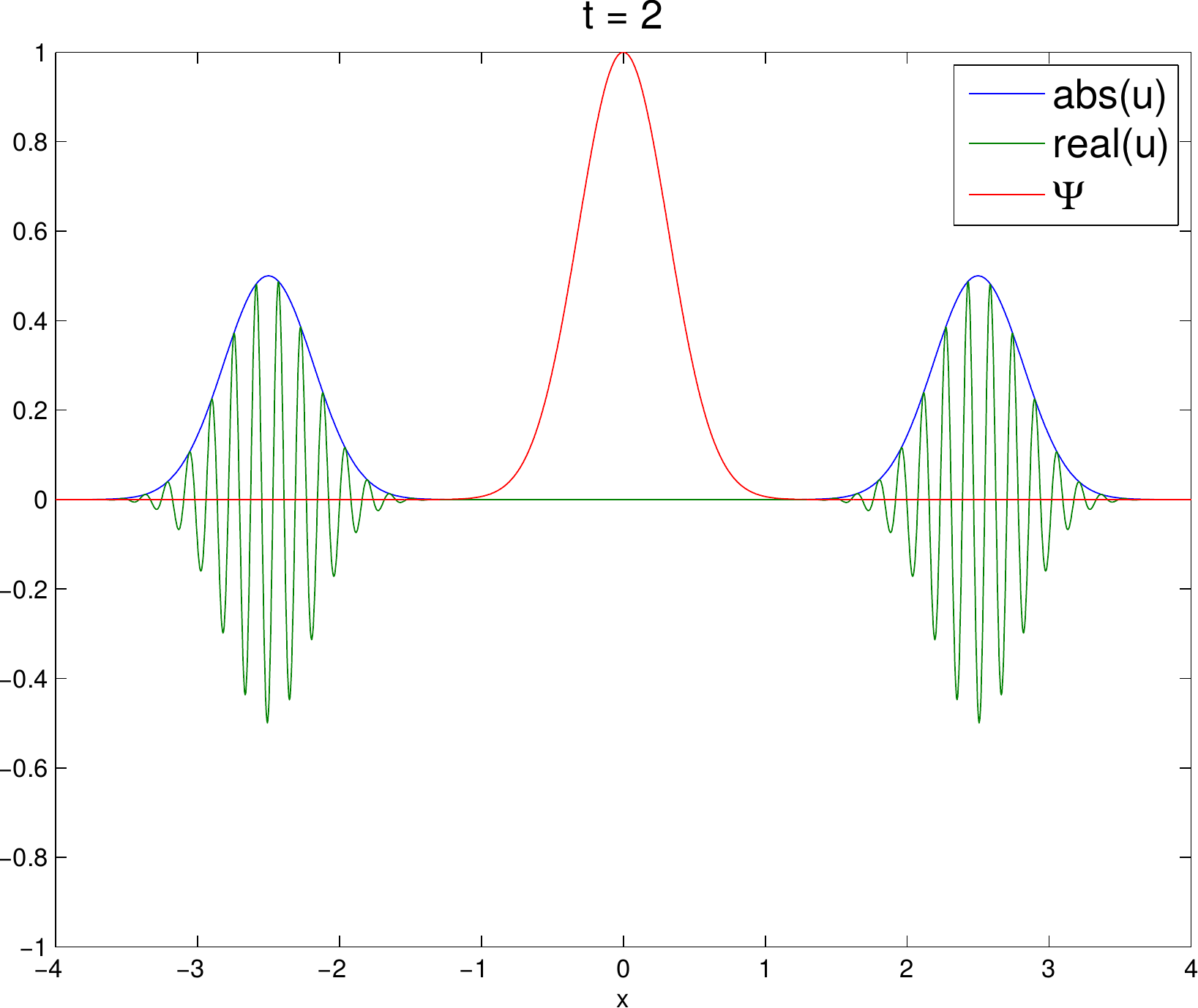}
                \end{minipage}
            \caption{d'Alembert solution with initial data \eqref{initdata} and \eqref{B0phi}.}
            \label{fig:sol}
		\end{figure}

Let us consider a 1D case with spatially constant speed $c(x,y) = c(y)$. The initial data to \eqref{waveeq},
\be\label{initdata}
	u^\eps(0,x,y) = B_0(x,y)e^{i\varphi_0(x,y)/\eps},\qquad u_t^\eps(0,x,y) = 0,
\ee
generate the d'Alembert solution
\be\label{dalemb}
	u^\eps(t,x,y) = u^+(t,x,y) + u^-(t,x,y), \qquad 
%%\frac12\left[B_0(x-c(y)t,y) e^{i\varphi_0(x-c(y)t,y)/\eps} + B_0(x+c(y)t,y)e^{i\varphi_0(x+c(y)t,y)/\eps}\right].
u^\pm(t,x,y) = \frac12 B_0(x\mp c(y)t,y) e^{i\varphi_0(x\mp c(y)t,y)/\eps}.
\ee
The QoI \eqref{QoIsmall0} therefore reads
\begin{align}
	\widetilde{\mc Q}(t,y) &= \int_\Real |u^+(t,x,y) + u^{-}(t,x,y)|^2 \psi(t,x) \, dx \nonumber \\
	&= \int_\Real \left(|u^+(t,x,y)|^2 + |u^-(t,x,y)|^2 + 2\Re (u^+(t,x,y)^* u^-(t,x,y))\right)\psi(t,x)\, dx\nonumber \\
	&=: \widetilde{Q}_+(t,y) + \widetilde{Q}_-(t,y) + \widetilde{Q}_0(t,y).\label{tt3}
\end{align}
The first two terms of $\widetilde{\mc Q}$ yield
\[
	\widetilde Q_{\pm}(t,y) = \int_\Real|u^\pm(t,x,y)|^2\psi(t,x) \, dx = \frac14 \int_{\Real}B_0^2(x\mp c(y)t,y)\psi(t,x)\, dx,
\]
where the integrand is smooth, compactly supported and independent of $\eps$, including all its derivatives in $y$. Therefore, the terms $\widetilde Q_{\pm}$ satisfy Theorem~\ref{th:new}.
The last term $\widetilde Q_0$ reads
\[
\widetilde Q_0(t,y)= \frac12 \int_\Real \cos\left(\frac{\varphi(t,x,y)}{\eps}\right) B_0(x+c(y)t,y)B_0(x-c(y)t,y)\psi(t,x)\, dx,
\]
where $\varphi(t,x,y):=\varphi_0(x+c(y)t,y)-\varphi_0(x-c(y)t,y)$.
This term could conceivably be problematic, depending on the choice of $B_0$ and $\varphi_0$.
Notably, the selection
\be\label{B0phi}
	B_0(x,y) = e^{-5(x+s)^2}+e^{-5(x-s)^2}, \qquad \varphi_0(x,y) = x, \qquad \psi(t,x) = e^{-5x^2},
\ee
produces two symmetric pulses centered at $\pm s$, each splitting into two waves traveling in opposite directions, see Figure \ref{fig:sol} where we set $s=1.5$ and $c=2$. The test function $\psi$ is compactly supported in $x$ for numerical purposes.
%Thus only the waves propagating towards each other contribute to the QoI, see Figure \ref{fig:sol}. 
%The first two terms have been shown in Theorem \ref{th:old} to be non-oscillatory. 
Let us also choose the speed $c(y) = y$ to be the stochastic variable. 
Then $\varphi(t,x,y) = 2yt$ and $\widetilde Q_0$ includes an oscillatory prefactor $\cos\left({2yt}/{\eps}\right)$ that does not depend on $x$ and hence cannot be damped by the test function $\psi$. Consequently, 
an $\eps^{-\sigma}$ term is produced when differentiating 
$\partial_y^\sigma\widetilde{\mc Q}(t,y)$. Thus $\widetilde{\mc Q}$ does not satisfy Theorem~\ref{th:new}.
The QoI \eqref{QoIsmall0} along with its first and second derivative in $y$ is depicted in Figure \ref{fig:QoI1}, left column, for varying $\eps = (1/40,1/80, 1/160)$. The plots display oscillations of growing amplitude with increasing $\sigma$ and decreasing $\eps$ as predicted. Here, we chose $y\in [1.5,2]$, $s=3$ and $t = 2$.

In general, for odd-order polynomial $\varphi_0$, there is a cosine prefactor independent of $x$ in $\widetilde Q_0$ which induces oscillations in $\eps$ of the QoI \eqref{QoIsmall0}.

    	\begin{figure}[h!]
            \centering
                \begin{minipage}[b]{0.3\textwidth}
                    \includegraphics[width=\textwidth]{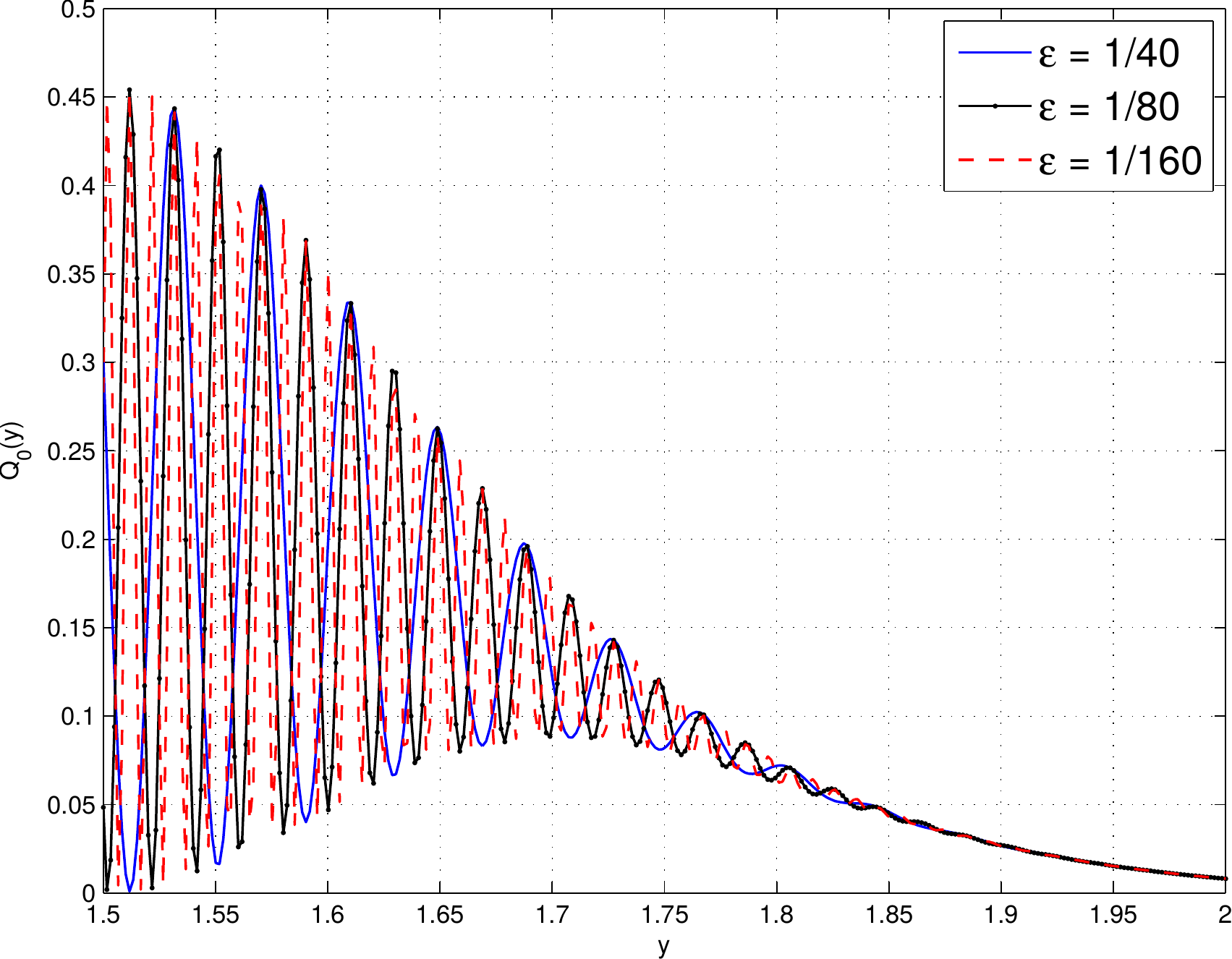}
                \end{minipage}
                \begin{minipage}[b]{0.3\textwidth}
                    \includegraphics[width=\textwidth]{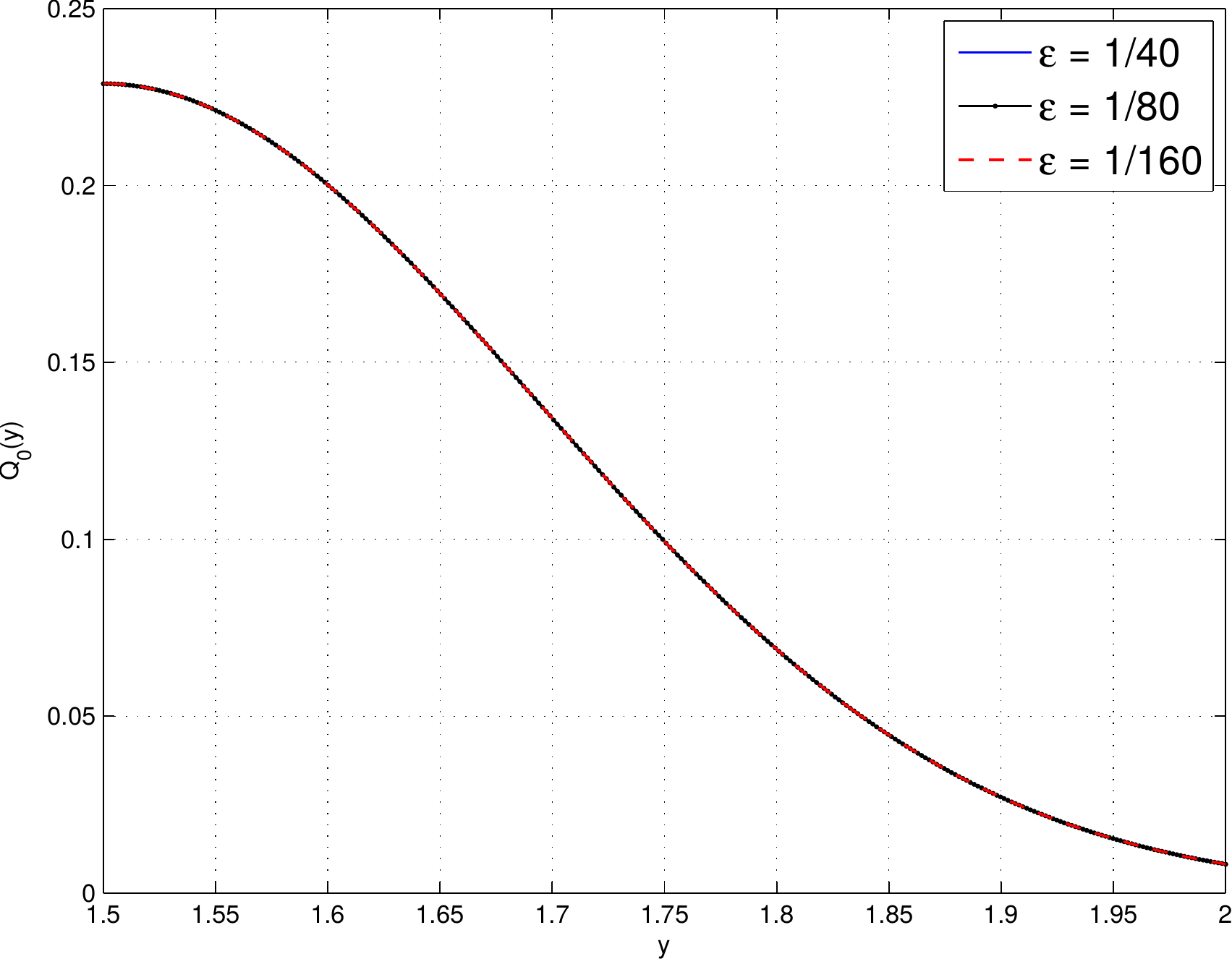}
                \end{minipage}
                \begin{minipage}[b]{0.3\textwidth}
                    \includegraphics[width=\textwidth]{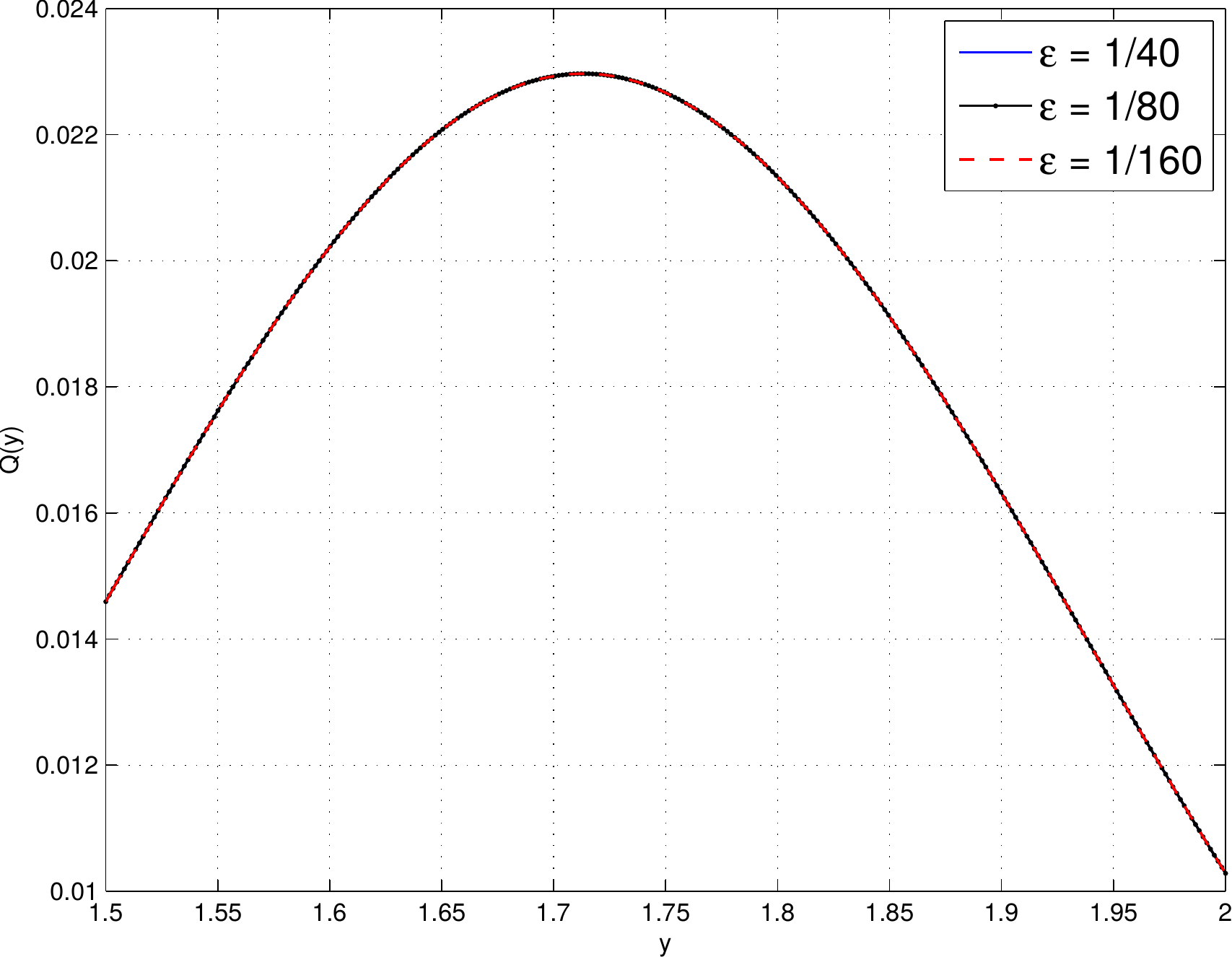}
                \end{minipage}
                \begin{minipage}[b]{0.3\textwidth}
                    \includegraphics[width=\textwidth]{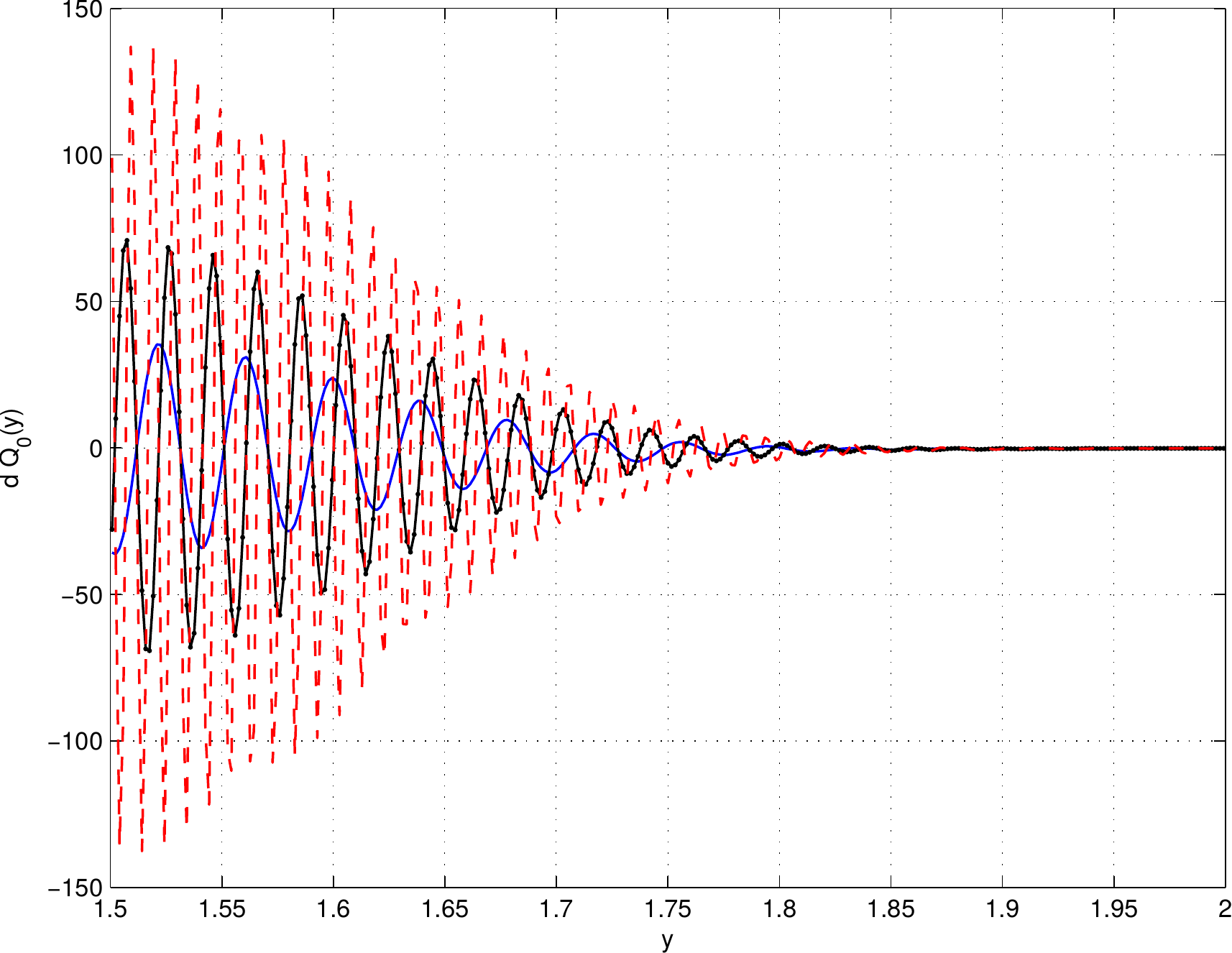}
                \end{minipage}
                \begin{minipage}[b]{0.3\textwidth}
                    \includegraphics[width=\textwidth]{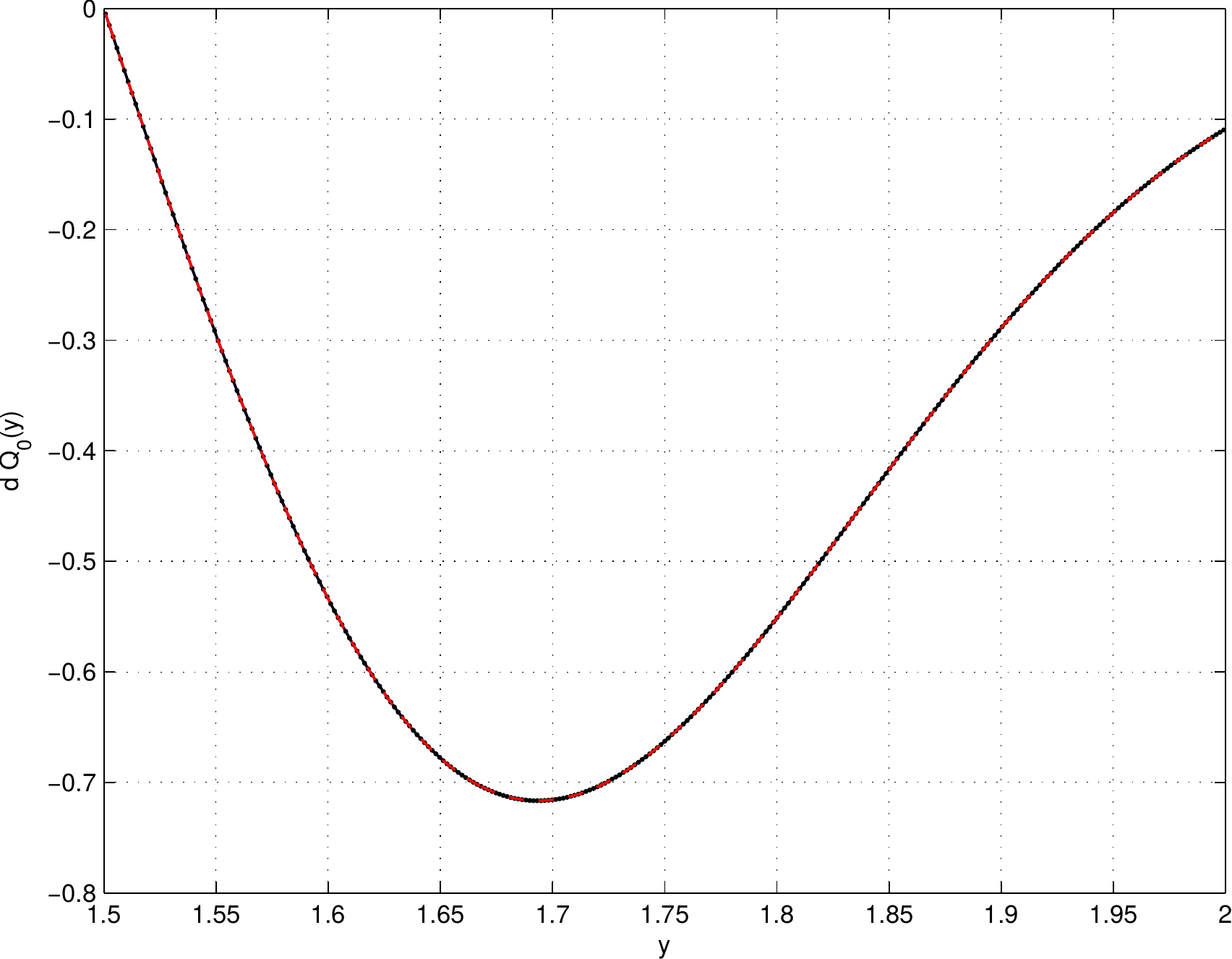}
                \end{minipage}
                \begin{minipage}[b]{0.3\textwidth}
                    \includegraphics[width=\textwidth]{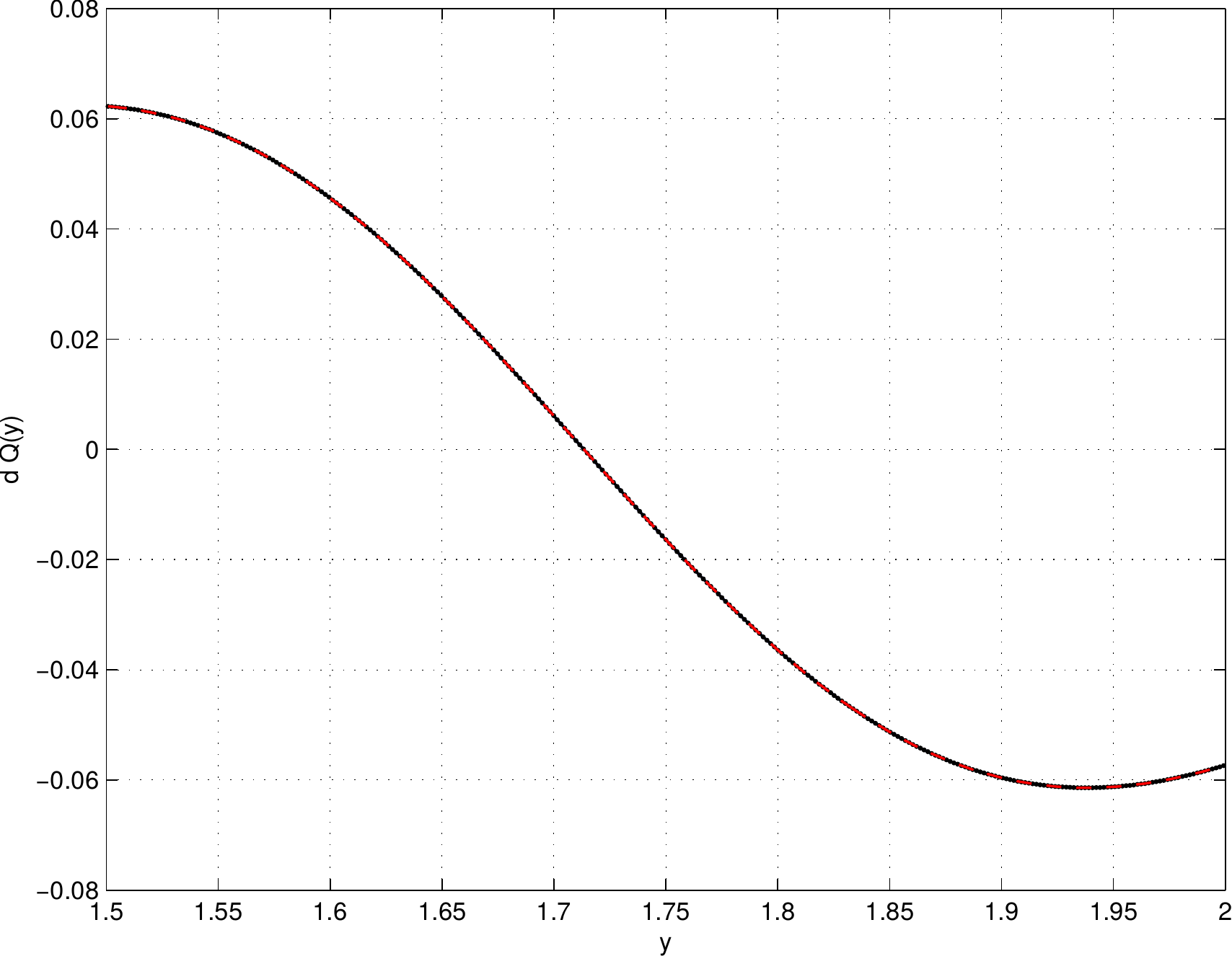}
                \end{minipage}
                \begin{minipage}[b]{0.3\textwidth}
                    \includegraphics[width=\textwidth]{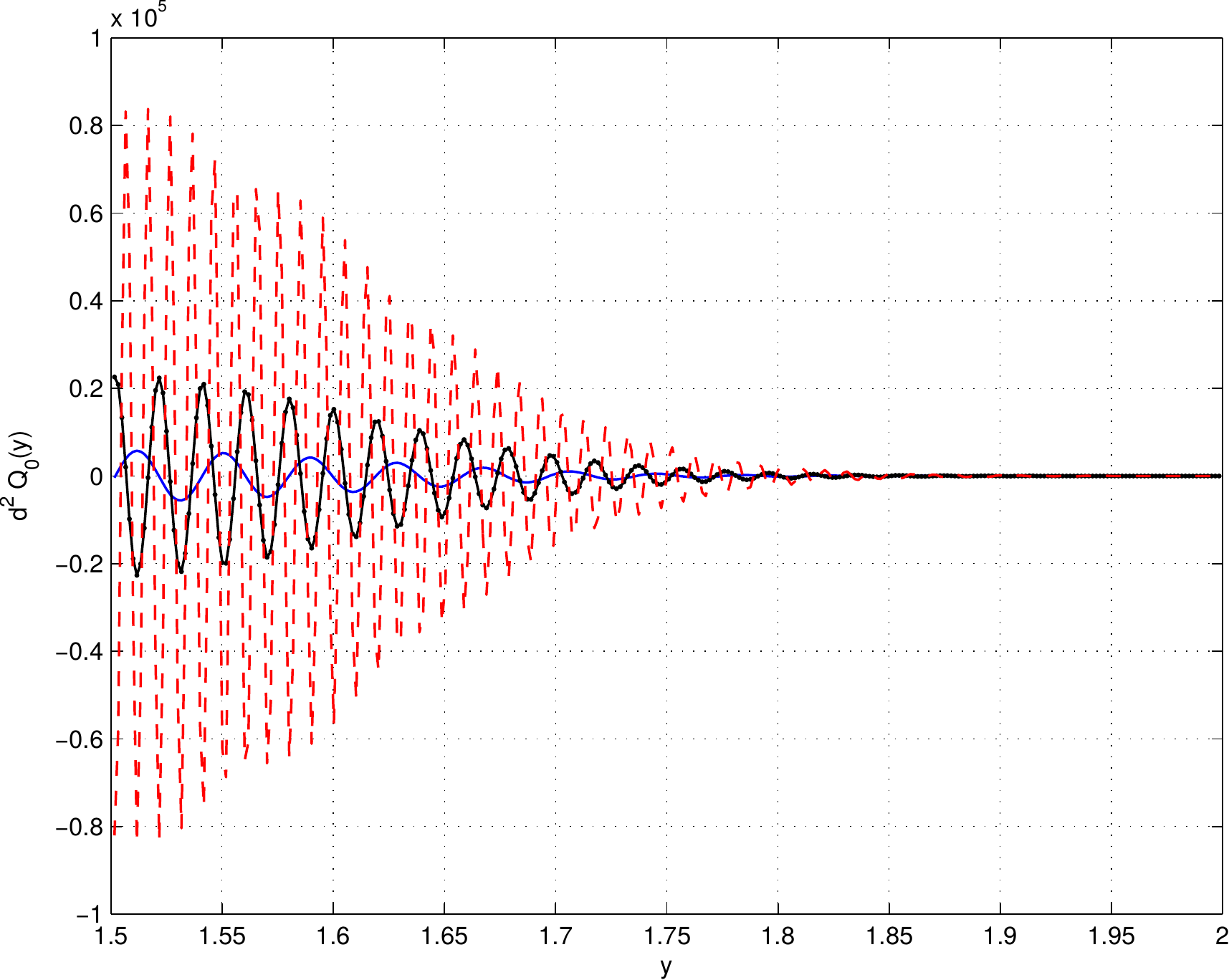}
                \end{minipage}
                \begin{minipage}[b]{0.3\textwidth}
                    \includegraphics[width=\textwidth]{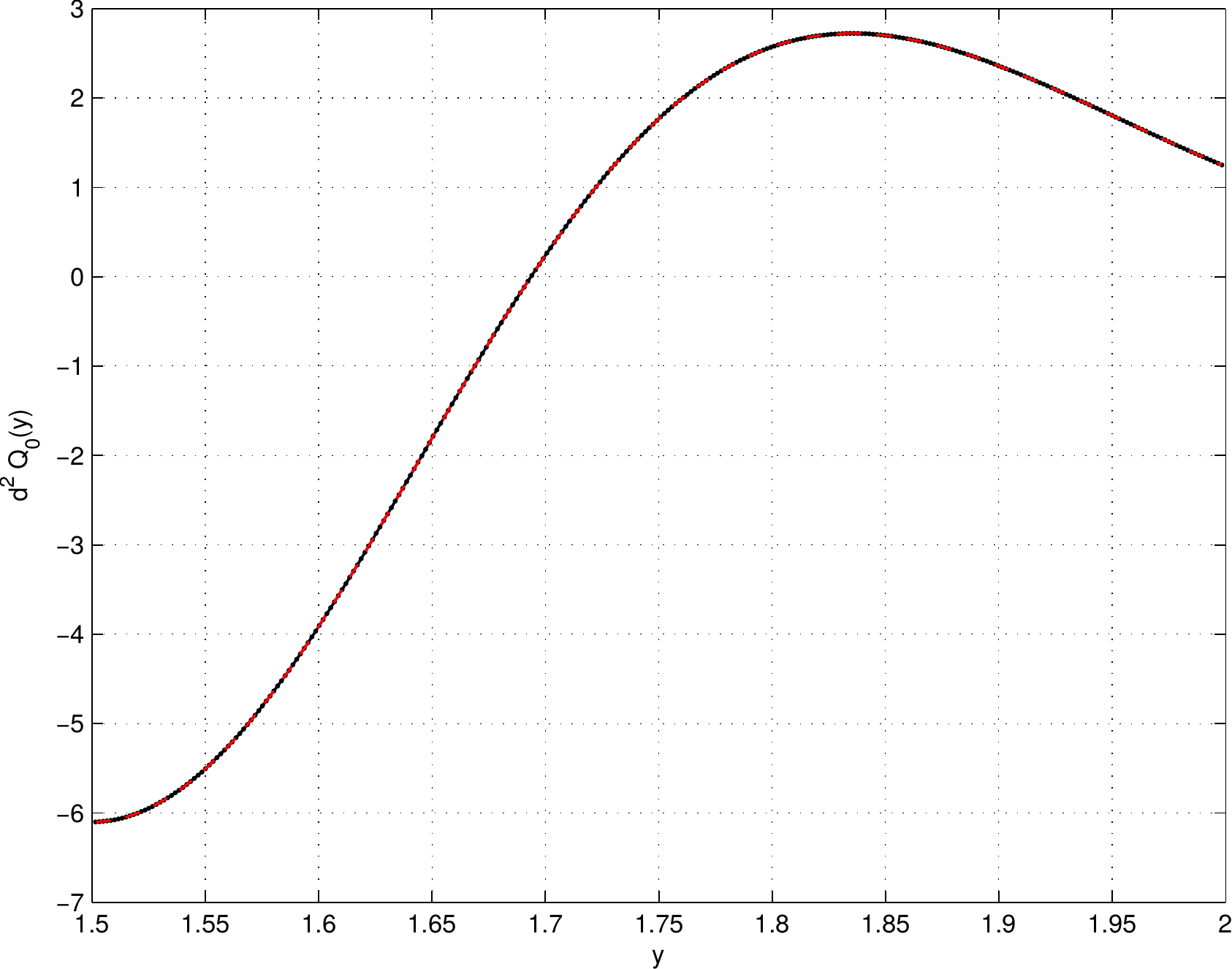}
                \end{minipage}
                \begin{minipage}[b]{0.3\textwidth}
                    \includegraphics[width=\textwidth]{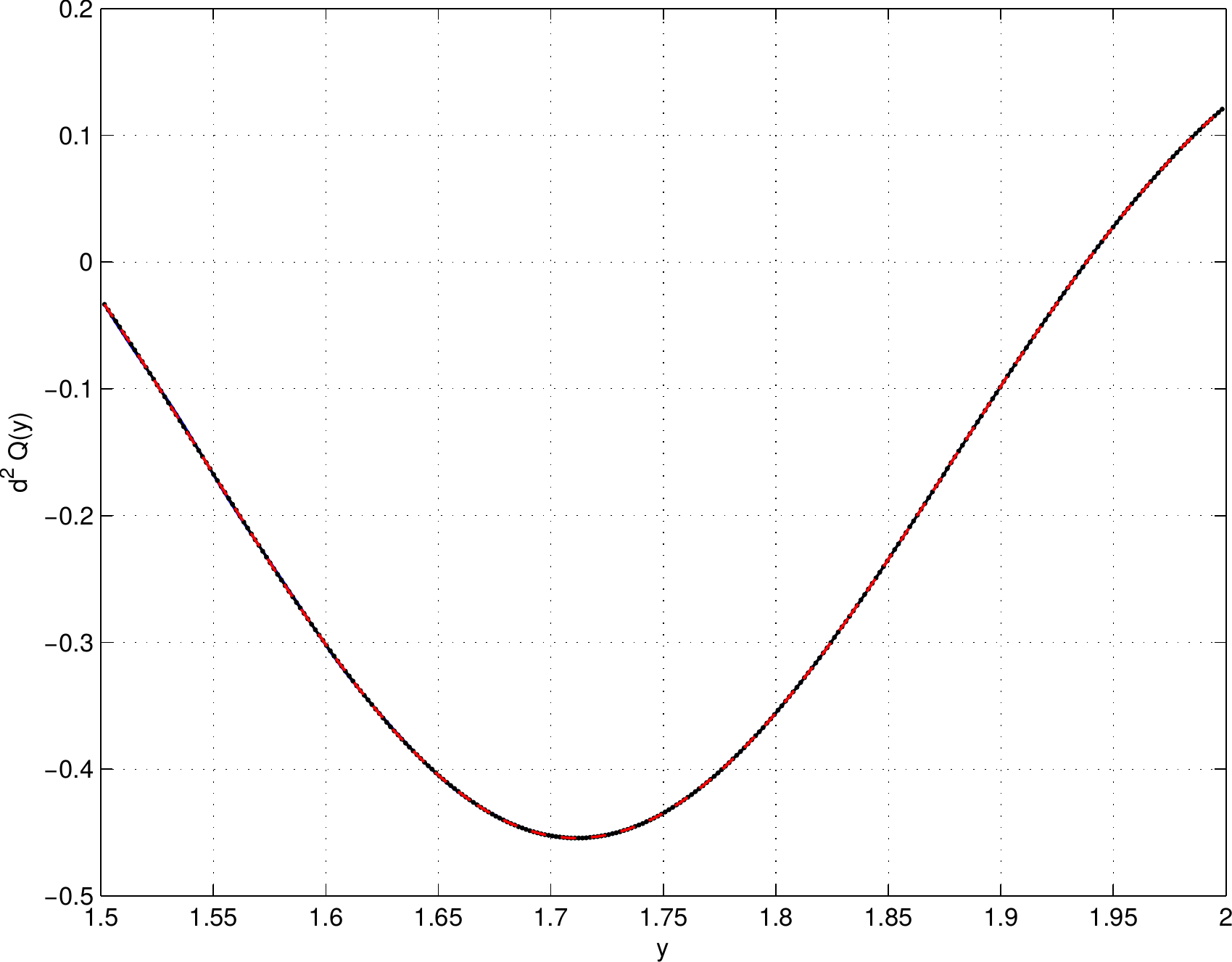}
                \end{minipage}
            \caption{Left column: QoI \eqref{QoIsmall0} with $\varphi_0(x,y) = x$, and its first and second derivative in $y$. Central column: QoI \eqref{QoIsmall0} with $\varphi_0(x,y) = x^2$. Right column: QoI \eqref{QoIbig0} with $\varphi_0(x,y) = x$.} 
            \label{fig:QoI1}
		\end{figure}

Note that when $\varphi_0$ is an
even-order polynomial in $x$, the QoI is not oscillatory for the example above. For instance, $\varphi_0(x,y) = x^2$ gives
%\begin{align*}
%	Q_0(t,y)	&=  \frac12 \int_\Real \cos\left(\frac{4xyt}{\eps}\right) g(t,x,y)\psi(x)\, dx,
%\end{align*}
$\varphi(t,x,y) = 4xyt$.
By the non-stationary phase lemma, for all compact $\Gamma_c\subset \Gamma$ there exist $c_s$ independent of $\eps$ such that 
\[
\sup_{\substack{y\in \Gamma_c\\ t\in [0,T]}}\left|\int_\Real \cos\left(\frac{4xyt}{\eps}\right) B_0(x+yt,y)B_0(x-yt,y)\psi(x)\, dx\right|\leq c_s\eps^s,
\] 
for all $s$ as $\eps\to 0$, and the same holds for its derivatives with respect to $y$. 
The QoI \eqref{QoIsmall0} with $\varphi_0(x,y) = x^2$ and its first and second derivatives in $y$ are plotted in Figure \ref{fig:QoI1}, central column, utilizing the same parameters as the previous example. No oscillations can be observed in the plot.

The different behavior of $\varphi_0(x,y) = x$ and $\varphi_0(x,y) = x^2$ in \eqref{B0phi} does not come as a surprise if one looks at the GB approximation \eqref{QoIsmall0GB} of \eqref{QoIsmall0}. 
Note that the left-going wave $u^-$ in \eqref{dalemb} is approximated solely by $u_k^-$ in \eqref{ukdef}. This is because all GBs $v_k^-$ in \eqref{vdef} move along the rays $(q^-,p^-)$ whose initial data are $q^-(0,y,z)=z$ and $p^-(0,y,z)=1$ by \eqref{GBini}. From \eqref{IC} this implies that $p^-(t,y,z) = 1$ and $q^-(t,y,z) = -yt+z$. Hence, 
as $y>0$ 
%we therefore have $q^-\to -\infty$ irrespectively of the starting point $z$; in other words 
all $v_k^-$ move to the left.
Similarly, $u^+$ is approximated merely by $u_k^+$. Therefore, the waves moving towards the origin (where the test function is supported) are from two different GB families. As stated above, a two-mode solution can thus give highly oscillatory QoIs.

%We will now show that the oscillations emerge due to the waves meeting in the origin being from two different families \eqref{ukdef}. %First observe that even-order polynomial phase can be modeled by one GB family only. %, the oscillatory term independent of $x$ disappears. 
%This in not the case for even-order polynomial phase.
In contrast, for $\varphi_0(x,y) = x^2$ we obtain $p^\pm(0,y,z) = p^\pm(t,y,z) = 2z$ and hence $q^{\pm}(t,y,z) = \pm y \frac{z}{|z|}t + z$. Therefore, both $q^+$ and $q^-$ can move in either direction depending on the starting point $z$. %Moreover, for the two-pulse initial amplitude $B_0$ in \eqref{B0phi} 
For our example, this implies that the two waves moving towards the origin belong to the same GB mode, $u_k^-$, and the two waves moving away belong to $u_k^+$. Since the test function $\psi$ is compactly supported around the origin, only $u^-_k$ will substantially contribute to the QoI \eqref{QoIsmall0GB}. Finally, by Theorem \ref{th:old}, the QoI \eqref{QoIsmall0GB} consisting of one GB mode solution is non-oscillatory.

\begin{remark}
Generally, a phase $\varphi_0=\varphi_0(x)$ whose derivative changes sign on $\Real$ allows for two waves approximated by the same mode moving in two different directions. In particular, this is true for even-order polynomials. Technically, $\varphi_0$ is not allowed to attain local extrema due to \ref{ASSUM3}. In practice however, it is enough to make sure that the support of $B_0$ and $B_1$ does not include the stationary point.
\end{remark}		

%%%%%%%%%%%%%%%%%%%%%%%%%%%%%%%%%%%%%%%%%%%%%%%%%%%%%%%%%%%%%%%%%%%%%%%%%%%
\subsection{New quantity of interest}

To avoid the oscillatory behavior
of $\widetilde{\mc Q}$ in (\ref{tt3}) we introduce the new QoI \eqref{QoIbig0},
in which $|u^\eps|^2\psi$
is integrated not only in $\mb x$ but also in time $t$,
%\be\label{QoI3}
%	\widetilde{\mc Q}(\mb y) = \int_{\Real}\int_{\Real^n} |u^\eps(t,\mb x,\mb y)|^2 \psi(t,\mb x) d\mb x \, dt,
%\ee
with $\psi\in C^\infty_c(\Real\times\Real^n)$. %We prove below in Theorem \ref{th:main} that the approximated QoI is regular in the sense \eqref{regularity}.
Let us first apply this QoI to the 1D oscillatory example from Section \ref{sec:whatcouldgowrong} with $\varphi_0(x,y) = x$,
\begin{align*}
	\mc Q(y) &= \int_\Real \int_\Real |u^+(t,x,y) + u^{-}(t,x,y)|^2 \psi(t,x) \, dx \, dt, \\
	&= \int_\Real\int_\Real \left(|u^+(t,x,y)|^2 + |u^-(t,x,y)|^2 + 2\Re (u^+(t,x,y)^* u^-(t,x,y))\right)\psi(t,x)\, dx\, dt\\
	&=: Q_+(y) + Q_-(y) + Q_0(y).
\end{align*}
Again, the first two terms yield
\[
	Q_\pm(y) = \int_\Real \int_\Real|u^\pm(t,x,y)|^2\psi(t,x) \, dx\, dt = \frac14 \int_\Real \int_{\Real}B_0^2(x\mp yt,y)\psi(t,x)\, dx \, dt,
\]
where the integrand is smooth, compactly supported in both $t$ and $x$ and independent of $\eps$, including all its derivatives in $y$. The last term reads
\[
	Q_0(y) =  \frac12 \int_\Real \int_{\Real} \cos\left(\frac{2yt}{\eps}\right) B_0(x+ yt,y)B_0(x-yt,y)\psi(t,x)\, dx\, dt,
\]
	%\frac14\int_\Real \int_{\Real} \left(f(t,x,y) + g(t,x,y)\sin\left(\frac{2yt}{\eps}\right)\right)\psi(t,x) \, dx\, dt,
and since the phase of $\cos\left(\frac{2yt}{\eps}\right)$ %= \frac{1}{2}\left(e^{\frac{2iyt}{\eps}}+e^{-\frac{2iyt}{\eps}}\right)$ 
has no stationary point in $t$, we can utilize the non-stationary phase lemma in $t$. As $\psi$ is compactly supported in both $t$ and $x$, we obtain the desired regularity: for all compact $\Gamma_c\subset \Gamma$, $\sup_{y\in \Gamma_c}|Q_0(y)|\leq c_s \eps^s$ for all $s$ as $\eps\to 0$, where $c_s$ is independent of $\eps$ and similarly for differentiation in $y$. The same then holds for ${\mc Q}(y)$.
%\ldots formula\ldots numerical example

To confirm this numerically, we use the initial data from the previous section and set
\[
\psi(t,x) = e^{-5x^2-300(t-t_s)^2},
\]
where $t_s = 1.75$. The rightmost column of Figure \ref{fig:QoI1} shows the QoI \eqref{QoIbig0} and its first and second derivatives with respect to $y$ for $\eps = (1/40,1/80,1/160)$. 
Compared to the first column
the oscillations are eliminated.
% and the QoI is independent of the wavelength.

%%%%%%%%%%%%%%%%%%%%%%%%%%%%%%%%%%%%%%%%%%%%%%%%%%%%%%%%%%%%%%%%%%%%%%%%%%%%%%%%%%%%%%%%%%%%%%%%%%%%%%%%%%%
\subsection{Stochastic regularity of ${\mc Q}^{p,\bs\alpha}$}
We now consider the general 
QoI $\mc Q^{p,\bs\alpha}$ in \eqref{QoIbig}
%\be\label{QoI4}
%	\widetilde{\mc Q}^{p,\bs\alpha}(\mb y) = \eps^{2(p+|\bs\alpha|)}\int_{\Real}\int_{\Real^n} |\partial_t^p\partial_{\mb x}^{\bs\alpha}u^\eps(t,\mb x,\mb y)|^2 \psi(t,\mb x) d\mb x \, dt,
%\ee
with $\psi$ as in \ref{ASSUM6} %In the next section, we will show that its GB approximation $\widetilde{\mc Q}^{p,\bs\alpha}_{\text{GB}}$ again possesses the stochastic regularity of type \eqref{regularity}.
and define its GB approximated version as 
\be\label{QoI3GB}
	{\mc Q}^{p,\bs\alpha}_{\text{GB}}(\mb y) = \eps^{2(p+|\bs\alpha|)}\int_{\Real}\int_{\Real^n} g(t,\mb x,\mb y)|\partial_t^p\partial_{\mb x}^{\bs\alpha}u_k(t,\mb x,\mb y)|^2 \psi(t,\mb x) d\mb x \, dt.
\ee
%As before, we assume that $u_k$ is a good approximation to $u^\eps$ and hence so is $\mc Q^{p,\bs\alpha}_{\text{GB}}$ to $\mc Q^{p,\bs\alpha}$.
We start off by defining the admissible cutoff parameter for the case of two-mode solutions.
\begin{proposition}\label{prop:admissible2}
Assume \ref{ASSUM1}--\ref{ASSUM3} hold. 
Then for all $T>0$, beam order $k$ and compact $\Gamma_c\subset \Gamma$, there is a GB cutoff width $\eta>0$ and constant $\delta >0$ such that for all $\mb x\in {\mc B}_{2\eta}$,
\be\label{eee}
\Im \Phi_k^\pm(t,\mb x,\mb y,\mb z)\geq \delta |\mb x|^2, \quad \forall t\in [0,T], \, \mb y\in \Gamma_c, \,\mb z\in K_0.
\ee
For the first order GB, $k=1$, we can take $\eta=\infty$
and (\ref{eee}) is valid for all ${\mb x}\in \Real^n$.
\end{proposition}
\begin{proof}
By Proposition \ref{prop:admissible}, for every $\Gamma_c$ there exist $\delta^+>0$ and $\eta^+>0$ such that for all $\mb x\in {\mc B}_{2\eta^+}$ we have $\Im \Phi_k^+(t,\mb x,\mb y,\mb z)\geq \delta^+ |\mb x|^2$, and analogously for $\Im \Phi_k^-$ with $\delta^-$ and $\eta^-$. Then choosing $\delta= \min\{\delta^+,\delta^-\}$ and $\eta = \min\{\eta^+,\eta^-\}$ yields the relation \eqref{eee} for all $\mb x\in {\mc B}_{2\eta}$. 
\end{proof}
\begin{definition}\label{def:admissible2}
The cutoff width $\eta$ used for the GB approximation is called admissible for a given $T$, $k$ and $\Gamma_c$ if it is small enough in the sense of Proposition \ref{prop:admissible2}.
\end{definition}
%In this section we prove the following theorem: {\bf Theorem}

\begin{remark}
 As in Section \ref{sec:onemodeprel}, we assume that $\eta<\infty$ without loss of
generality. %When $k=1$ one can choose $\eta=\infty$, which amounts to removing the cutoff functions $\rho_\eta$ in (18) altogether. This is convenient in computations, but there are some technical issues with having $\eta=\infty$ in the proofs below. We note, however, that, in any finite time interval $[0,T]$, 
We note that also for the two-mode solutions, the Gaussian beam superposition \eqref{umode} with no cutoff is identical to the one with a large enough cutoff, because of the compact support of the test function $\psi(t,\mb x)$.
\end{remark}

We will now prove the main theorem, which shows 
that the QoI \eqref{QoI3GB} is indeed non-oscillatory.
\begin{theorem}\label{th:main}
Assume \ref{ASSUM1}--\ref{ASSUM6} hold. Moreover, let $\eta<\infty$ be admissible %in the sense of Definition \ref{def:admissible2} 
for $T>0$, $k$ and a compact $\Gamma_c\subset \Gamma$. 
Then 
for all $p\in{\mathbb N}$ and ${\bs\alpha}\in{\mathbb N_0^N}$,
there exist $C_{\bs\sigma}$ such that
\[\sup_{\mb y\in \Gamma_c}\left|\frac{\partial^{\bs\sigma} {\mc Q}^{p,\bs\alpha}_{\text{GB}}(\mb y)}{\partial \mb y^{\bs\sigma}}\right|\leq C_{\bs\sigma}, \quad \forall \bs\sigma\in\mathbb N_0^N,\]
where $C_{\bs\sigma}$ is independent of $\eps$ but depends on $T$, $k$ and $\Gamma_c$. 
\end{theorem}

%\bigskip

In the proof 
we will use the following notation.
Let $\Wcal_\mu$ and $\Sigma_\mu$, for $\mu<\infty$, denote the spaces
\begin{align*}
\Wcal_\mu &= \Bigl\{f\in C^\infty: \text{supp} \, f(t,\cdot,\mb y,\mb z,\mb z')\subset \Sigma_\mu(t,\mb y,\mb z,\mb z'), \; \forall t\in \Real, \, \mb y\in \Gamma, \, \mb z,\mb z' \in \Real^n\Bigr\},\\
%\mc W_\eta &= \Bigl\{f\in C^\infty: \text{supp } f(t,\cdot,\mb y,\mb z,\mb z')\subset \Lambda^s_\eta(t,\mb y,\mb z,\mb z'), \: \forall t\in \Real, \, \mb y\in \Gamma, \, \mb z,\mb z'\in \Real^n\Bigr\},\\
&\ \ \ \text{where}\:
\Sigma_\mu(t,\mb y,\mb z,\mb z') := \{ \mb x\in \Real^n: |\mb x-\mb q^+(t,\mb y,\mb z)|\leq 2\mu \quad \text{and } \quad |\mb x-\mb q^-(t,\mb y,\mb z')|\leq 2\mu\}.
%&\Lambda^s_\eta(t,\mb y,\mb z,\mb z') := \{\mb x\in \Real^n: |\mb x-\Delta \mb q|\leq 2\eta \quad \text{and} \quad |\mb x+\Delta \mb q|\leq 2\eta\}.
\end{align*}
%
%
%
%\begin{itemize}
%\item $\Sigma_\mu = \Sigma_\mu(t,\mb y,\mb z,\mb z') = \{ \mb x\in \Real^n: |\mb x-\mb q^+(t,\mb y,\mb z)|\leq 2\mu \quad \text{and } \quad |\mb x-\mb q^-(t,\mb y,\mb z')|\leq 2\mu\}$.
%\item $\Wcal_\mu = \{f\in \mc Y: \text{supp} \, f(t,\cdot,\mb y,\mb z,\mb z')\subset \Sigma_\mu(t,\mb y,\mb z,\mb z'), \; \forall t\in \Real, \, \mb y\in \Gamma, \, \mb z,\mb z' \in \Real^n\}.$
%%\cap \text{supp}\, \psi(t,\cdot),\; \text{supp} \, f(\cdot,\mb x,\mb y,\mb z,\mb z')\subset [0,T]\}.$ 
%\end{itemize}
Note that the space $\Sigma_\mu$ is similar to $\Lambda_\mu$ introduced in Section~\ref{sec:th:new}. Instead of containing $\mb x$ that are close enough to two beams from the same mode, it contains $\mb x$ that lie at a distance at most $2\mu$ from two beams from different modes. We also note that there exist two spaces $\mc S^\pm_\mu$ as defined in Section~\ref{sec:onemodeprel} since we have two modes of $\Phi_k^\pm$ and that Lemma~\ref{lemma:PSprop} holds for both.
%The function $\psi\in C_c^\infty(\Real\times\Real^n)$ in the definition of $\mc W_\mu$ is as in \eqref{QoI3}.
%We also define $\Sigma_\infty = \Real^n$ and $\mc W_\infty = \mc Y$.

For the remainder of the proof we fix
the final time $T>0$, the beam order $k$ and
the compact set $\Gamma_c\subset \Gamma$. Moreover, we select $\eta<\infty$ admissible in the sense of Definition \ref{def:admissible2}. 
%Note that $\eta = \infty$ in case of the first order GB, $k=1$ {\color{cyan} (Check that all is ok in that case)}. 
%The proof of Theorem \ref{th:main} uses the following two lemmas.
An important part of the proof relies on the non-stationary
phase lemma:
\begin{lemma}[Non-stationary phase lemma]\label{lemma:nonstat}
Suppose  $\Theta\in C^\infty(\Real)$ and $f\in C_c^\infty(\Real)$ with $\text{supp} f\subset [0,T].$ If $\partial_t\Theta(t) \ne 0$ for all $t\in[0,T]$ then the following estimate holds true for all $K\in \mathbb N_0$,
\[
	\left|\int_\Real f(t) e^{i\Theta(t)/\eps}dt \right|
\leq C_K(1+\|\Theta\|_{C^{K+1}([0,T])})^K \eps^K \sum_{m\leq K} \int_\Real \frac{|\partial_t^m f(t)|}{|\partial_t \Theta(t)|^{2K-m}} e^{-\Im \Theta(t)/\eps} \, dt,
\]
where $C_K$ depends on $K$ but is independent of $\eps, f, \Theta$, $T$, and
	\[
		\|\Theta\|_{C^{K+1}([0,T])} = \sum_{k=0}^{K+1}\sup_{t\in [0,T]}\left|\Theta^{(k)}(t)\right|.
	\]

\end{lemma}
The proof of this lemma is classical. See  e.g. \cite{hormander83}.
Upon keeping careful track of the constants in this proof we
get 
the precise dependence on $\|\Theta\|$ in the right hand
side of the estimate.

\begin{lemma}\label{lemma:partialI}
Define
\[
	I(\mb y,\mb u) = f(\mb y,\mb u) e^{i\Theta(\mb y,\mb u)/\eps},
\]
for $f,\Theta\in C^\infty(\Gamma\times \Real^d)$, where $\text{supp} \, f(\mb y,\cdot)\subset D\subseteq\Real^d$, $\forall \mb y\in \Gamma$. Then there exist functions $ f_{j\bs\sigma}\in C^\infty(\Gamma\times\Real^d)$ with $\text{supp} \, f_{j\bs\sigma}(\mb y,\cdot)\subset D$, $\forall \mb y\in \Gamma$ such that,
\be\label{partialI}
\frac{\partial^{\bs\sigma} I(\mb y,\mb u)}{\partial \mb y^{\bs\sigma}} =  \sum_{j=0}^{|\bs\sigma|} \eps^{-j} f_{j\bs\sigma}(\mb y,\mb u) e^{i\Theta(\mb y,\mb u)/\eps}.
\ee
\end{lemma}
\begin{proof}
We will carry out the proof by induction. For $\bs\sigma = \mb 0$, we choose $f_{0\mb 0} = f$ and the lemma holds. Let us assume \eqref{partialI} is true for a fixed $\bs\sigma$. Then for $\tilde{\bs\sigma} = \bs\sigma + \mb e_k$ where $\mb e_k$ is the $k$-th unit vector we have
\begin{align*}
	\frac{\partial^{\tilde{\bs\sigma}} I(\mb y,\mb u)}{\partial \mb y^{\tilde{\bs\sigma}}} &=  \frac{\partial }{\partial y_k}\sum_{j=0}^{|\bs\sigma|} \eps^{-j} f_{j\bs\sigma}(\mb y,\mb u) e^{i\Theta(\mb y,\mb u)/\eps} \\
&= \sum_{j=0}^{|\bs\sigma|} \eps^{-j} \left(\frac{\partial f_{j\bs\sigma}(\mb y,\mb u)}{\partial y_k}+ f_{j\bs\sigma}(\mb y,\mb u)\frac{i}{\eps}\frac{\partial \Theta(\mb y,\mb u)}{\partial y_k}\right)e^{i\Theta(\mb y,\mb u)/\eps}.
%&= \left(\sum_{j=0}^{|\bs\sigma|} \eps^{-j} g_{j\bs\sigma}(\mb y,\mb u)+ \sum_{j=0}^{|\bs\sigma|+1} \eps^{-j} h_{j\bs\sigma}(\mb y,\mb u)\right)e^{i\Theta(\mb y,\mb u)/\eps},
\end{align*}
Hence we can take
\[
	f_{j\tilde{\bs\sigma}} = \left\{\begin{array}{ll}
					\frac{\partial f_{0\bs\sigma}}{\partial y_k}, & j = 0,\\
					\frac{\partial f_{j\bs\sigma}}{\partial y_k} + if_{j-1\bs\sigma}\frac{\partial \Theta}{\partial y_k}, & 1\leq j\leq |\tilde{\bs \sigma}|-1,\\
	if_{j-1\bs\sigma}\frac{\partial \Theta}{\partial y_k}, & j = |\tilde{\bs\sigma}|.
			\end{array}\right.
\]
%with $g_{j\bs\sigma} = \frac{\partial f_{j\bs\sigma}}{\partial y_k}$ and $h_{j\bs\sigma} = i f_{j\bs\sigma}\frac{\partial \Theta}{\partial y_k}$ and hence both $g_{j\bs\sigma},h_{j\bs\sigma}\in C^\infty(\Gamma\times\Real^d)$ with $\text{supp} \, g_{j\bs\sigma}(\mb y,\cdot)\subset D$ and $\text{supp} \, h_{j\bs\sigma}(\mb y,\cdot)\subset D$. Since $|\tilde{\bs\sigma}| = |\bs\sigma| + 1$, we can find a $\tilde f_{j\tilde{\bs\sigma}}$ such that
Clearly, we have $f_{j\tilde{\bs\sigma}}\in C^\infty(\Gamma\times\Real^d)$ with $\text{supp} \, f_{j\tilde{\bs\sigma}}(\mb y,\cdot)\subset D$ for all $\mb y\in \Gamma$. The proof is complete.
%\[
%\frac{\partial^{\tilde{\bs\sigma}} I(\mb y,\mb u)}{\partial \mb y^{\tilde{\bs\sigma}}} = \sum_{j=0}^{|\tilde{\bs\sigma}|} \eps^{-j} \tilde f_{j\tilde{\bs\sigma}}(\mb y,\mb u) e^{i\Theta(\mb y,\mb u)/\eps},
%\]
%where $\tilde f_{j\tilde{\bs\sigma}}\in C^\infty(\Gamma\times\Real^d)$ and $\text{supp} \, \tilde f_{j\tilde{\bs\sigma}}(\mb y,\cdot)\subset D$, 
%which completes the proof.
\end{proof}

Recalling the definition of $u_k$ in \eqref{umode}, $\mc Q^{p,\bs\alpha}_{\text{GB}}$ in \eqref{QoI3GB} becomes
\begin{align}
	\mc Q^{p,\bs\alpha}_{\text{GB}}(\mb y) &= \eps^{2(p+|\bs\alpha|)}\int_\Real \int_{\Real^n} g(t,\mb x,\mb y)\left|\partial_t^p \partial_{\mb x}^{\bs\alpha}u^+_k(t,\mb x,\mb y)+ \partial_t^p \partial_{\mb x}^{\bs\alpha} u^-_k(t,\mb x,\mb y)\right|^2 \, \psi(t,\mb x) d\mb x \, dt\nonumber \\
&= \eps^{2(p+|\bs\alpha|)} \int_\Real \int_{\Real^n} g(t,\mb x,\mb y)[|\partial_t^p \partial_{\mb x}^{\bs\alpha} u^+_k(t,\mb x,\mb y)|^2+|\partial_t^p \partial_{\mb x}^{\bs\alpha} u^-_k(t,\mb x,\mb y)|^2 \nonumber \\
& \qquad \qquad \qquad \qquad + 2\Re(\partial_t^p \partial_{\mb x}^{\bs\alpha} u^+_k(t,\mb x,\mb y)^* \partial_t^p \partial_{\mb x}^{\bs\alpha} u^-_k(t,\mb x,\mb y))]\, \psi(t,\mb x) d\mb x \, dt \nonumber\\
&=: Q_1(\mb y) + Q_2(\mb y) + 2\Re(Q_3(\mb y)),\label{Qnew}
\end{align}
where $\psi\in C_c^\infty(\Real\times\Real^n)$ is as in \ref{ASSUM6} and $g\in C^\infty(\Real\times\Real^n\times \Gamma)$.
The first two terms of \eqref{Qnew}, $Q_1$ and $Q_2$, possess the required stochastic regularity as a consequence of Theorem \ref{th:new}. %Indeed, if $\text{supp} \, \psi(\cdot,\mb x)\subset [0,T]$, $\forall \mb x\in \Real^n$,
Indeed, as $\psi$ is only supported for $t\in[0,T]$
we can write
$$
Q_1(\mb y) = \int_0^T \widetilde{Q}_1(t,{\mb y}) dt,
$$
where the reduced QoI $\tilde{Q}_1$
satisfies the assumptions of Theorem \ref{th:new}. 
(Note that when $\eta$ is admissible it
admissible for both $\Phi_k^+$ and $\Phi_k^-$ individually.)
Then
\be\label{supQ1}
	\sup_{\mb y\in\Gamma_c} \left|\partial_{\mb y}^{\bs\sigma} Q_1(\mb y)\right| \leq \int_0^T \sup_{\substack{\mb y\in\Gamma_c\\t\in[0,T]}}\left|\partial_{\mb y}^{\bs\sigma} \widetilde{Q}_1^{p,\bs\alpha}(t,\mb y)\right| \, dt \leq T C_{\bs\sigma},
\ee
and analogously for $Q_2$.

We will now prove that $Q_3$ satisfies the same regularity condition owing to the absence of stationary points of the phase. % and then recall the non-stationary phase lemma \ref{lemma:nonstat}. 
Let us examine the quantity
\begin{align}
\partial_{\mb y}^{\bs\sigma} Q_3(\mb y) &= \eps^{2(p+|\bs\alpha|)}\partial_{\mb y}^{\bs\sigma} \int_\Real \int_{\Real^n} g(t,\mb x,\mb y)\partial_t^p \partial_{\mb x}^{\bs\alpha}u^+_k(t,\mb x,\mb y)^* \partial_t^p \partial_{\mb x}^{\bs\alpha}u^-_k(t,\mb x,\mb y)\,\psi(t,\mb x) d\mb x \, dt \nonumber \\
&= \left(\frac{1}{2\pi \eps}\right)^{n} \int_{K_0\times K_0}\int_{K_1}  \partial^{\bs\sigma}_{\mb y} I(\mb x,\mb y,\mb z,\mb z') \, d\mb x\, d\mb z \, d\mb z', \label{partialQ3}
\end{align}
where
\begin{align*}
	I(\mb x,\mb y,\mb z,\mb z') = \eps^{2(p+|\bs\alpha|)} \int_\Real & \partial_t^p \partial_{\mb x}^{\bs\alpha} w_k^+(t,\mb x-\mb q^+(t,\mb y,\mb z),\mb y,\mb z)^* \partial_t^p \partial_{\mb x}^{\bs\alpha}w_k^-(t,\mb x-\mb q^-(t,\mb y,\mb z'),\mb y,\mb z')\\
	&\times g(t,\mb x,\mb y)\psi(t,\mb x)\, dt,
\end{align*}
with
\[
	w_k^\pm(t,\mb x,\mb y,\mb z) = A_k^\pm(t,\mb x,\mb y,\mb z) \varrho_\eta(\mb x) e^{i\Phi^\pm_k(t,\mb x,\mb y,\mb z)/\eps}.
\]
Recalling Lemma \ref{lemma:partialwk}, we can find $s_k^\pm\in \mc S_\eta^\pm$ such that
\begin{align*}
	I(\mb x,\mb y, \mb z,\mb z') &= \int_\Real s_k^+(t,\mb x-\mb q^+(t,\mb y,\mb z),\mb y,\mb z)^* s_k^-(t,\mb x-\mb q^-(t,\mb y,\mb z'),\mb y,\mb z') g(t,\mb x,\mb y)\psi(t,\mb x)\, dt\\
	&= \sum_{\ell=0}^{L_1}\sum_{m=0}^{L_2}\eps^{\ell+m}\int_\Real a_{\ell m}(t,\mb x,\mb y,\mb z,\mb z') \psi(t,\mb x) e^{i\vartheta_k(t,\mb x,\mb y,\mb z,\mb z')/\eps} \, dt,
\end{align*}
where 
\[
a_{\ell m}(t,\mb x,\mb y,\mb z,\mb z') = g(t,\mb x,\mb y)p_{\ell}^+(t,\mb x-\mb q^+(t,\mb y,\mb z),\mb y,\mb z)^* p_m^-(t,\mb x-\mb q^-(t,\mb y,\mb z'),\mb y,\mb z'),
\] 
with $p_\ell^+, p_m^-\in \mc P_\eta$, and
\be\label{Thetak}
	\vartheta_k(t,\mb x,\mb y,\mb z,\mb z')=\Phi_k^-(t,\mb x-\mb q^-(t,\mb y,\mb z'),\mb y,\mb z')-\Phi_k^+(t,\mb x-\mb q^+(t,\mb y,\mb z),\mb y,\mb z)^*.\ee
%\begin{comment}
%where by \eqref{ukdef} and \eqref{vdef},
%\begin{align*}
%	\tilde I(\mb x,\mb y,\mb z,\mb z') & = \int_\Real v_k^+(t,\mb x,\mb y,\mb z)^*v_k^-(t,\mb x,\mb y,\mb z')\varrho_\eta(\mb x-\mb q^+(t,\mb y,\mb z)) \varrho_\eta(\mb x-\mb q^-(t,\mb y,\mb z'))\psi(t,\mb x)\,dt\\
%	& = \int_\Real a_k(t,\mb x,\mb y,\mb z,\mb z') \psi(t,\mb x) e^{i\vartheta_k(t,\mb x,\mb y,\mb z,\mb z')/\eps} \, dt,
%\end{align*}
%\begin{align*}
%	a_k(t,\mb x,\mb y,\mb z,\mb z')&= A_k^+(t,\mb x-\mb q^+(t,\mb y,\mb z),\mb y, \mb z)^* A_k^-(t,\mb x-\mb q^-(t,\mb y,\mb z'),\mb y, \mb z')\times\\
%&\times \varrho_\eta(\mb x-\mb q^+(t,\mb y,\mb z))\varrho_\eta(\mb x-\mb q^-(t,\mb y,\mb z')).
%\end{align*}
%\end{comment}
By Proposition \ref{prop:smooth}, we have $\vartheta_k\in C^\infty$, and $a_{\ell m}\in \Wcal_{\eta}$ because both 
$p_\ell^+,p_m^-$ are supported in the ball ${\mc B}_{2\eta}$. %For $\eta = \infty$, we have $a_{\ell m}\in \mc Y = \mc W_\infty$ since both $P_\ell^+,P_m^-\in \mc X$.
%To estimate $Q_3$ and its $\mb y$-derivatives, let us first consider the $\bs\sigma$-th derivative of $\tilde I$ in \eqref{tildeI} with respect to the random variable, $\mb y$.
Therefore, by Lemma \ref{lemma:partialI}, there exist functions $f_{\ell mj\bs\sigma}\in \Wcal_{\eta}$ such that
\be\label{partialI2}
\partial_{\mb y}^{\bs\sigma} I(\mb x,\mb y,\mb z,\mb z') =  \sum_{j=0}^{|\bs\sigma|}\sum_{\ell=0}^{L_1}\sum_{m=0}^{L_2} \eps^{\ell+m-j} \int_\Real f_{\ell mj\bs\sigma}(t,\mb x,\mb y,\mb z,\mb z') \psi(t,\mb x) e^{i\vartheta_k(t,\mb x,\mb y,\mb z,\mb z')/\eps}dt.
\ee
%Let us define $g_{\ell mj\bs\sigma} = f_{\ell mj\bs\sigma}\psi$ and note that $g_{\ell mj\bs\sigma}\in\mc Y$, since $f_{\ell mj\bs\sigma}\in\mc Y$ and $\psi\in C_c^\infty(\Real\times\Real^n)$; and $\text{supp}\, g_{\ell mj\bs\sigma}(\cdot,\mb x,\mb y,\mb z,\mb z')\subset [0,T]$, $\forall \mb x,\mb y,\mb z,\mb z'$, and $\text{supp}\, g_{\ell mj\bs\sigma}(t,\cdot,\mb y,\mb z,\mb z')\subset K_1\cap \Sigma_\eta$, $\forall t,\mb y,\mb z,\mb z'$.

%We will show that there exists $\mu>0$ such that $\Theta_k$ does not have any stationary points in $t$, i.e. $\partial_t\Theta_k\ne 0$, for $\mb x\in \Omega_\mu$ as introduced in Definition \ref{def:Omega}. 
The following proposition shows that $\vartheta_k$ has no stationary points in $t\in[0,T]$ for all $\mb x\in \Sigma_\mu$ with a small enough $\mu$. Note that this is true even for $\mb z=\mb z'$.
\begin{proposition}\label{prop:nostatpt2}
There exist $0<\mu\leq 1$ and $\nu>0$ such that for all $\mb y\in \Gamma_c$, $\mb z\in K_0$, $\mb z' \in K_0$, $t\in[0,T]$ and for all $\mb x\in \Sigma_\mu$,
\be\label{partialtThetak}
	|\partial_t \vartheta_k(t,\mb x,\mb y,\mb z,\mb z')|\geq \nu.
\ee
\end{proposition}
\begin{proof}
%Let us fix a compact $\Gamma_c\subset \Gamma$.
Differentiating \eqref{Thetak} with respect to $t$ and using \eqref{phidef} and \eqref{IC}, we obtain
\be\label{pat}
\partial_t \vartheta_k= -\partial_t\mb q^-\cdot \mb p^- +\partial_t\mb q^+\cdot \mb p^+ + R_k = -c(\mb q^-,\mb y)|\mb p^-|-c(\mb q^+,\mb y)|\mb p^+| + R_k,
\ee
where $R_k = R_k(t,\mb x,\mb y,\mb z,\mb z')$ reads
\begin{align*}
R_k &= (\mb x-\mb q^-)\cdot \partial_t\mb p^-- (\mb x-\mb q^+)\cdot \partial_t \mb p^+ -\partial_t \mb q^- \cdot M^-(\mb x-\mb q^-)+\partial_t\mb q^+ \cdot(M^+)^*(\mb x-\mb q^+)\nonumber \\
&+ \frac12(\mb x-\mb q^-)\cdot \partial_t M^- (\mb x-\mb q^-)+ \frac12(\mb x-\mb q^+)\cdot (\partial_t M^+)^* (\mb x-\mb q^+)\nonumber \\
&+ \sum_{|\bs\beta|=3}^{k+1} \frac{1}{\bs\beta!}\left(\partial_t \phi_{\bs\beta}^- (\mb x-\mb q^-)^{\bs\beta} + \phi_{\bs\beta}^- \partial_t(\mb x-\mb q^-)^{\bs\beta}\right)\\
&- \sum_{|\bs\beta|=3}^{k+1} \frac{1}{\bs\beta!}\left(\partial_t \phi_{\bs\beta}^+ (\mb x-\mb q^+)^{\bs\beta} + \phi_{\bs\beta}^+ \partial_t(\mb x-\mb q^+)^{\bs\beta}\right)^*.
\end{align*}
Since $\mb q^\pm, \mb p^\pm, M^\pm, \phi_{\bs\beta}^\pm$ are smooth in all variables by Proposition \ref{prop:smooth}, their time derivative is uniformly bounded in the compact set $[0,T]\times\Gamma_c\times K_0$. If $\mb x\in\Sigma_\mu$ for some $0<\mu\leq 1$, then both $|\mb x-\mb q^-|\leq 2\mu$ and $|\mb x-\mb q^+|\leq 2\mu$ and we arrive at
\[
|R_k|\leq C_k\mu,
\]
with $C_k$ 
%depending on $\Gamma_c, T, k$, but 
independent of $\mu$. % {\color{magenta} (But $\mu$ is dependent on $t,\mb y,\mb z,\mb z'$. Fix $\tilde\mu = \sup_{t,\mb y,\mb z,\mb z'}\mu$ instead?)}.

Next, we note that $H(\mb p^+,\mb q^+,\mb y) = c(\mb q^+,\mb y)|\mb p^+|$ is conserved along the ray,
\[
	c(\mb q^+(t,\mb y,\mb z),\mb y)|\mb p^+(t,\mb y,\mb z)| = c(\mb q^+(0,\mb y,\mb z),\mb y)|\mb p^+(0,\mb y,\mb z)| = c(\mb z,\mb y)|\nabla \varphi_0(\mb z,\mb y)|,
\]
and therefore by \ref{ASSUM1} and \ref{ASSUM3} we obtain a uniform lower bound on $c(\mb q^+,\mb y)|\mb p^+|$, for all $t\in \Real$, $\mb y\in\Gamma_c$ and $\mb z\in K_0$, 
\[c(\mb q^+(t,\mb y,\mb z),\mb y)|\mb p^+(t,\mb y,\mb z)| \geq c_{\text{min}}\inf_{\substack{\mb z\in K_0\\ \mb y\in \Gamma_c}}|\nabla \varphi_0(\mb z,\mb y)|\geq \gamma >0,\]
and similarly, from the conservation of $H(\mb p^-,\mb q^-,\mb y)$ we obtain $c(\mb q^-(t,\mb y,\mb z'),\mb y)|\mb p^-|\geq \gamma >0$.
Thus from \eqref{pat} we get
\[
|\partial_t \vartheta_k|\geq c(\mb q^-,\mb y)|\mb p^-|+c(\mb q^+,\mb y)|\mb p^+|- |R_k|\geq 2 \gamma - C_k\mu \geq \nu>0,
\]
for all $\mb x\in \Sigma_\mu$ 
upon taking $\mu$ small enough.
%with a small enough $\mu$. %, and all $\mb y\in \Gamma_c$, $\mb z,\mb z' \in K_0$, $t\in[0,T]$.
%Setting $\nabla_{\mb x}\Theta_1=0$, we obtain $\mb p^- = \mb p^+ + O(\sqrt \eps)$ which gives $\partial_t \Theta_1= -[c(\mb q^-)+c(\mb q^+)]|\mb p^+| + O(\sqrt \eps)$. For small $\eps$, a necessary condition for $\partial_t \Theta_1 = 0$ is that the leading term is 0. Since $c$ is a positive function, the only choice remains that $|\mb p^+| = 0$. However, we started with the assumption {\color{cyan}(State the condition in introduction)} that $|\mb p|\ne 0$ for all $t,\mb y,\mb z$ and no stationary point therefore exists such that $ \nabla_{\mb x} \Theta_1=0$ and $\partial_t \Theta_1=0$ are satisfied at the same time.
\end{proof}

%We are now ready to prove the main theorem.
%\subsection{Proof of Theorem \ref{th:main}}
%By Theorem \ref{th:old} {\color{orange} (not really, integration in time as well, check assumptions)} $Q_1$ and $Q_2$ in \eqref{Qnew} posses the required stochastic regularity, i.e. there exist $C_{\bs\sigma,1}$ and $C_{\bs\sigma,2}$ independent of $\eps$ such that for all compact $\Gamma_c\subset \Gamma$
%\[
%\sup_{\mb y\in\Gamma_c}\left|\frac{\partial^{\bs\sigma} Q_1(\mb y)}{\partial \mb y^{\bs\sigma}}\right|\leq C_{\bs\sigma,1},\quad \text{and}\quad \sup_{\mb y\in\Gamma_c}\left|\frac{\partial^{\bs\sigma} Q_2(\mb y)}{\partial \mb y^{\bs\sigma}}\right|\leq C_{\bs\sigma,2}.
%\]
%We prove a similar estimate for $Q_3$ in \eqref{Qnew}.

We are now ready to finalize the proof of Theorem~\ref{th:main}.
We 
first choose $0<\mu\leq\eta<\infty$ such that Proposition \ref{prop:nostatpt2} holds. % and the cutoff parameter $\eta$ admissible in the sense of Definition \ref{def:admissible2}. 
Furthermore, note that the admissibility condition implies that for all $\mb x$ satisfying $|\mb x-\mb q^\pm|\leq 2\eta$ we have $\Im \Phi_k^\pm(t,\mb x-\mb q^\pm,\mb y,\mb z)\geq \delta |\mb x-\mb q^\pm|^2$. We can therefore estimate 
$\Im \vartheta_k$ with $\vartheta_k$ as in \eqref{Thetak} as
\begin{align}
 \Im \vartheta_k(t,\mb x,\mb y,\mb z,\mb z') &= \Im \Phi_k^-(t,\mb x-\mb q^-(t,\mb y,\mb z'),\mb y,\mb z') + \Im \Phi_k^+(t,\mb x-\mb q^+(t,\mb y,\mb z),\mb y,\mb z)\nonumber \\
 & \geq  \delta |\mb x-\mb q^-(t,\mb y,\mb z')|^2 + \delta |\mb x-\mb q^+(t,\mb y,\mb z)|^2,\label{ImThetak}
\end{align}
for all $\mb x\in \Sigma_{\eta}$. 
To estimate $|\partial_{\mb y}^{\bs\sigma} Q_3|$ we recall \eqref{partialQ3},
\be\label{partialQ3est}
\left|\partial_{\mb y}^{\bs\sigma} Q_3(\mb y)\right| \leq \left(\frac{1}{2\pi \eps}\right)^n \int_{K_0\times K_0}\int_{K_1}\left|\partial_{\mb y}^{\bs\sigma} I(\mb x,\mb y,\mb z,\mb z')\right| \, d\mb x \, d\mb z \, d\mb z',
\ee
and by \eqref{partialI2} and \ref{ASSUM5} one has
\be\label{partialI2est}
\left|\partial_{\mb y}^{\bs\sigma} I(\mb x,\mb y,\mb z,\mb z')\right| \leq \sum_{j=0}^{|\bs \sigma|}\sum_{\ell=0}^{L_1}\sum_{m=0}^{L_2}\eps^{-|\bs\sigma|}\left|\int_\Real f_{\ell mj\bs\sigma}(t,\mb x,\mb y,\mb z,\mb z')\psi(t,\mb x)e^{i\vartheta_k(t,\mb x,\mb y,\mb z,\mb z')/\eps}\, dt\right|.
\ee
Let us introduce the function
\[
	g_1(t,\mb x,\mb y,\mb z,\mb z') = \varrho_{\mu}(\mb x-\mb q^+(t,\mb y,\mb z))\varrho_{\mu}(\mb x-\mb q^-(t,\mb y,\mb z')),
\]
so that $g_1\in \mc W_\mu$. Then for $g_2 := 1 - g_1 \in C^\infty$ and $\text{supp }g_2(t,\cdot,\mb y,\mb z,\mb z')\subset \Real^n\setminus \Sigma_{\mu/2}$ for all $t,\mb y,\mb z,\mb z'$. We will now regard \eqref{partialI2est} one term at a time, and use the partition of unity $1 = g_1 + g_2$, %(omitting variables),
\[
	\int_\Real f_{\ell mj\bs\sigma}\psi e^{i\vartheta_k/\eps}\, dt = \int_\Real f_{\ell mj\bs\sigma}\psi(g_1 + g_2)e^{i\vartheta_k/\eps}\, dt = {\textcircled{\raisebox{-0.9pt}{1}}} + {\textcircled{\raisebox{-0.9pt}{2}}}.
\]
Let us first estimate the term {\textcircled{\raisebox{-0.9pt}{1}}}.
We have $\Sigma_{\mu/2}\subset\Sigma_{\eta}$ and therefore for $g_{\ell mj \bs\sigma}:= f_{\ell mj \bs\sigma}\psi g_1$ we have $\text{supp }g_{\ell mj\bs\sigma}(\cdot,\mb x,\mb y,\mb z,\mb z')\subset [0,T]$, $\forall \mb x,\mb y,\mb z, \mb z'$, and $\text{supp }g_{\ell mj\bs\sigma}(t,\cdot,\mb y,\mb z,\mb z')\subset \Sigma_{\mu/2}(t,\mb y,\mb z,\mb z') \cap K_1$, $\forall t,\mb y,\mb z, \mb z'$.
We now restrict $(t,{\mb y}, {\mb z},{\mb z}')$ to
the compact set $[0,T]\times \Gamma_c\times K_0\times K_0$.
Since the gradient $\partial_t\vartheta_k$ does not vanish for $\mb x\in \Sigma_{\mu/2}$ on this set
by Proposition \ref{prop:nostatpt2} we can employ the non-stationary phase Lemma \ref{lemma:nonstat}, % to {\textcircled{\raisebox{-0.9pt}{1}}},%each of the terms in \eqref{partialI2},
\begin{align*}
|{\textcircled{\raisebox{-0.9pt}{1}}}| &\leq 
\left|\int_\Real g_{\ell mj\bs\sigma}(t,\mb x,\mb y,\mb z,\mb z')e^{i\vartheta_k(t,\mb x,\mb y,\mb z,\mb z')/\eps}\, dt\right| \\
& \leq C_K D_K\eps^K \sum_{q=0}^K \int_\Real \frac{|\partial_t^{q}g_{\ell mj\bs\sigma}(t,\mb x,\mb y,\mb z,\mb z')|}{|\partial_t \vartheta_k(t,\mb x,\mb y,\mb z,\mb z')|^{2K-q}} e^{-\Im \vartheta_k(t,\mb x,\mb y,\mb z,\mb z')/\eps} \, dt,
\end{align*}
for every $K\in \mathbb N_0$.
Here,
$C_K$ only depends on $K$ and
$$
D_K= \left(1+\left\|
\vartheta_k(\,\cdot\,,\mb x,\mb y,\mb z,\mb z')
\right\|_{C^{K+1}([0,T])}\right)^K\leq \tilde{D}_K,
$$
since $\vartheta\in C^\infty$ and $(\mb x,\mb y,\mb z,\mb z')$
belongs to the compact set $K_1\times\Gamma_c\times K_0\times K_0$.
Similarly, since $g_{\ell mj\bs\sigma}\in C^\infty$, its time derivatives are uniformly bounded: for all $t\in [0,T]$, $\mb y\in \Gamma_c$, $\mb z,\mb z'\in K_0$ and ${\mb x}\in K_1$,
\[
% \sup_{\mb x\in \Sigma_{\mu/2}\cap K_1}|\partial_t^{q}g_{\ell mj\bs\sigma}(t,\mb x,\mb y,\mb z,\mb z')| \leq \sup_{\mb x\in K_1}
 |\partial_t^{q}g_{\ell mj\bs\sigma}(t,\mb x,\mb y,\mb z,\mb z')|\leq C_{\ell mj\bs\sigma q}.
\]
Therefore, using the fact that $\Im \vartheta_k \geq 0$ from \eqref{ImThetak} and recalling \eqref{partialtThetak} we obtain
\[
|{\textcircled{\raisebox{-0.9pt}{1}}}| \leq C_K \eps^K \sum_{q=0}^K \int_{0}^T \frac{C_{\ell mj\bs\sigma q}}{{\nu}^{2K-q}} \, dt \leq \widetilde C_{K \ell mj \bs\sigma} \eps^K,
\]
where $\widetilde C_{K \ell mj \bs\sigma}$ also depends on $T,\mu,\eta,\Gamma_c,k,\nu,p,\bs\alpha$, but is independent of $\eps$.

Secondly, let us estimate the term {\textcircled{\raisebox{-0.9pt}{2}}}. Since $\text{supp }g_2(t,\cdot,\mb y,\mb z,\mb z')\subset \Real^n\setminus\Sigma_{\mu/2}(t,\mb y,\mb z,\mb z')$, {\textcircled{\raisebox{-0.9pt}{2}}} is only nonzero for either $|\mb x-\mb q^{+}(t,\mb y,\mb z)|>2\mu$ or $|\mb x-\mb q^{-}(t,\mb y,\mb z')|>2\mu$ (or both) and therefore by \eqref{ImThetak},
\[
 \Im \vartheta_k(t,\mb x,\mb y,\mb z,\mb z') \geq \delta \mu^2,
\]
whenever $t\in [0,T]$, $\mb y\in \Gamma_c$, $\mb z,\mb z'\in K_0$
and ${\mb x}$ is in the support of $g_2$.
As $h_{\ell mj\bs\sigma}:= f_{\ell mj\bs\sigma} \psi g_2\in C^\infty$, {\textcircled{\raisebox{-0.9pt}{2}}} can be estimated as
\begin{align*}
|{\textcircled{\raisebox{-0.9pt}{2}}}|&\leq \int_0^T \left|h_{\ell mj\bs\sigma}(t,\mb x,\mb y,\mb z,\mb z')\right| e^{-\Im \vartheta_k(t,\mb x,\mb y,\mb z,\mb z')/\eps}\, dt \\
&\leq T \widetilde C_{\ell mj\bs\sigma} e^{-\delta \mu^2/\eps},
\end{align*}
for all $\mb y\in \Gamma_c$, $\mb z,\mb z'\in K_0$
and ${\mb x}\in K_1$.
%
%
%
%
%
%The function $h_{\ell mj\bs\sigma}:= f_{\ell mj\bs\sigma} \psi g_2\in C^\infty$ is supported and bounded on the same set, 
%
%is bounded on the same compact set,
%%
%%hence for all $t\in [0,T]$, $\mb y\in \Gamma_c$, $\mb z,\mb z'\in K_0$,
%\[
%\sup_{\mb x\in \Real^n\setminus \Sigma_{\mu/2} \cap K_1}|h_{\ell mj\bs\sigma}(t,\mb x,\mb y,\mb z,\mb z')|\leq \sup_{\mb x\in K_1}|h_{\ell mj\bs\sigma}(t,\mb x,\mb y,\mb z,\mb z')|\leq  \widetilde C_{\ell mj\bs\sigma}.
%\]
%Therefore, {\textcircled{\raisebox{-0.9pt}{2}}} can be estimated as
%\begin{align*}
%|{\textcircled{\raisebox{-0.9pt}{2}}}|&\leq \int_0^T \left|h_{\ell mj\bs\sigma}(t,\mb x,\mb y,\mb z,\mb z')\right| e^{-\Im \vartheta_k(t,\mb x,\mb y,\mb z,\mb z')/\eps}\, dt \\
%&\leq T \widetilde C_{\ell mj\bs\sigma} e^{-\delta \mu^2/\eps},
%\end{align*}
%for all $\mb y\in \Gamma_c$, $\mb z,\mb z'\in K_0$
%and ${\mb x}\in K_1$.
%
%
Collecting {\textcircled{\raisebox{-0.9pt}{1}}} and {\textcircled{\raisebox{-0.9pt}{2}}} together, we obtain from \eqref{partialI2est}
\begin{align*}
\left|\partial^{\bs\sigma}_{\mb y} I(\mb x,\mb y,\mb z,\mb z')\right| &\leq  \sum_{j=0}^{|\bs\sigma|} \sum_{\ell=0}^{L_1} \sum_{m=0}^{L_2} \eps^{-|\bs\sigma|} \left(|{\textcircled{\raisebox{-0.9pt}{1}}}| + |{\textcircled{\raisebox{-0.9pt}{2}}}|\right)\\
&\leq \max_{j,\ell,m}\eps^{-|\bs\sigma|} \left(\widetilde C_{K\ell mj \bs\sigma} \eps^K + T \widetilde C_{\ell m j \bs\sigma} e^{-\delta\mu^2/\eps}\right).
\end{align*}
Finally, by \eqref{partialQ3est} we have
\begin{align*}
\left|\partial_{\mb y}^{\bs\sigma} Q_3(\mb y)\right| %&\leq \left(\frac{1}{2\pi \eps}\right)^n \int_{K_0\times K_0}\int_{\Sigma_\eta\backslash\Sigma_\mu \cap K_1}\left|\partial_{\mb y}^{\bs\sigma} \tilde I(\mb x,\mb y,\mb z,\mb z')\right| \, d\mb x \, d\mb z \, d\mb z'\\
&\leq (2\pi)^{-n} \eps^{-|\bs\sigma|-n} |K_0|^2|K_1| \max_{j,\ell,m} \left(\widetilde C_{K\ell mj \bs\sigma} \eps^K + T \widetilde C_{\ell m j \bs\sigma} e^{-\delta\mu^2/\eps}\right).
\end{align*}
That is, choosing $K\geq n+|\bs\sigma|$, the first term is bounded
in $\eps$. Since $\delta>0$, the second term decays fast as a function of $\eps$ for any $\bs\sigma$. Therefore, there exists an upper bound $C_{\bs\sigma}$ such that 
\[\sup_{\mb y\in \Gamma_c}\left|\partial^{\bs\sigma}_{\mb y} Q_3(\mb y)\right| \leq C_{\bs\sigma},\]
where $C_{\bs\sigma}$ depends on $T,\mu,\eta,\Gamma_c,k,\delta,L_1,L_2,p,\bs\alpha$, but is uniform in $\eps$.
Recalling \eqref{Qnew} and \eqref{supQ1} we then arrive at
\[
	\sup_{\mb y\in\Gamma_c} \left|\partial_{\mb y}^{\bs\sigma} {\mc Q}_{\text{GB}}^{p,\bs\alpha}(\mb y)\right| \leq \sup_{\mb y\in\Gamma_c} \left|\partial_{\mb y}^{\bs\sigma} Q_1(\mb y)\right| + \sup_{\mb y\in\Gamma_c} \left|\partial_{\mb y}^{\bs\sigma} Q_2(\mb y)\right| + 2 \sup_{\mb y\in\Gamma_c} \left|\partial_{\mb y}^{\bs\sigma} Q_3(\mb y)\right| \leq \widetilde C_{\bs\sigma},
\]
with $C_{\bs\sigma}$ dependent on $T,\mu,\eta,\Gamma_c,k,K,\delta,\nu, L_1,L_2,p,\bs\alpha$, but independent of $\eps$, which concludes the proof 
of Theorem~\ref{th:main}.

%%%%%%%%%%%%%%%%%%%%%%%%%%%%%%%%%%%%%%%%%%%%%%%%%%%%%%%%%%%%%%%%%%%%%%%%%%%%%%%%%%%%%%%%%%%%%%%%%%%%%%%%%%%%%%%%%%%%%%%%%%%%%

\subsection{Numerical example}
A numerical example was presented in Section \ref{sec:whatcouldgowrong} comparing the QoIs $\widetilde{\mathcal Q}$ in \eqref{QoIsmall0} and $\mathcal Q$ in \eqref{QoIbig0}. We were able to obtain the exact solution since the speed was constant and the spatial variable was one-dimensional. In higher dimensions, however, caustics can appear and the exact solution is typically no longer available. Instead, we make use of the GB approximations $\widetilde{\mathcal Q}_{\text{GB}}$ in \eqref{QoIsmall0GB} and $\mathcal Q_{\text{GB}}:= \mathcal Q_{\text{GB}}^{0,\textbf 0}$ in \eqref{QoIbigGB}.

Let us consider a 2D wave equation \eqref{waveeq} with $\mb x = [x_1,x_2]$. The initial data include two random parameters $\mb y = [y_1,y_2]$,
\begin{align*}
	B_0(\mb x,\mb y) &= e^{-10((x_1+1)^2 + (x_2-y_1)^2)} + e^{-10((x_1-1)^2 + (x_2-y_1)^2)},  && B_1(\mb x,\mb y) = 0, \\
	\varphi_0(\mb x,\mb y) &= |x_1|+(x_2-y_1)^2,  && c(\mb x,\mb y) = y_2. %1 - y_1 e^{-y_2 x_1^2 - y_3 x_2^2}.
\end{align*}
The test function is chosen as
\[
 \psi(\mb x) = \left\{ \begin{array}{ll}
 			e^{-\frac{|\mb x|^2}{1-|\mb x|^2}}, & \text{for } |\mb x|\leq 1,\\
 			0, & \text{otherwise}.
 			\end{array}\right.
\]
This setup corresponds to two pulses centered in $(\pm 1,y_1)$ at $t=0$, moving along the $x_1$ axis, while spreading or contracting in the $x_2$ direction, see Figure \ref{fig:solforqoi1b}, where we plot the modulus of the first-order GB solution $|u_{1}(t,\mb x,\mb y)|$ at $t=1$ for various combinations of $y$. The central circle denotes the support of the test function $\psi$.

%The stochastic regularity we showed above allows us to use fast quadrature method, in particular, the sparse stochastic collocation method provides a remedy to the curse of dimensionality if $N$ the dimension of the stochastic space is moderately high. 

    	\begin{figure}[h!]
            \centering
                \begin{minipage}[b]{0.32\textwidth}
                    \includegraphics[width=\textwidth]{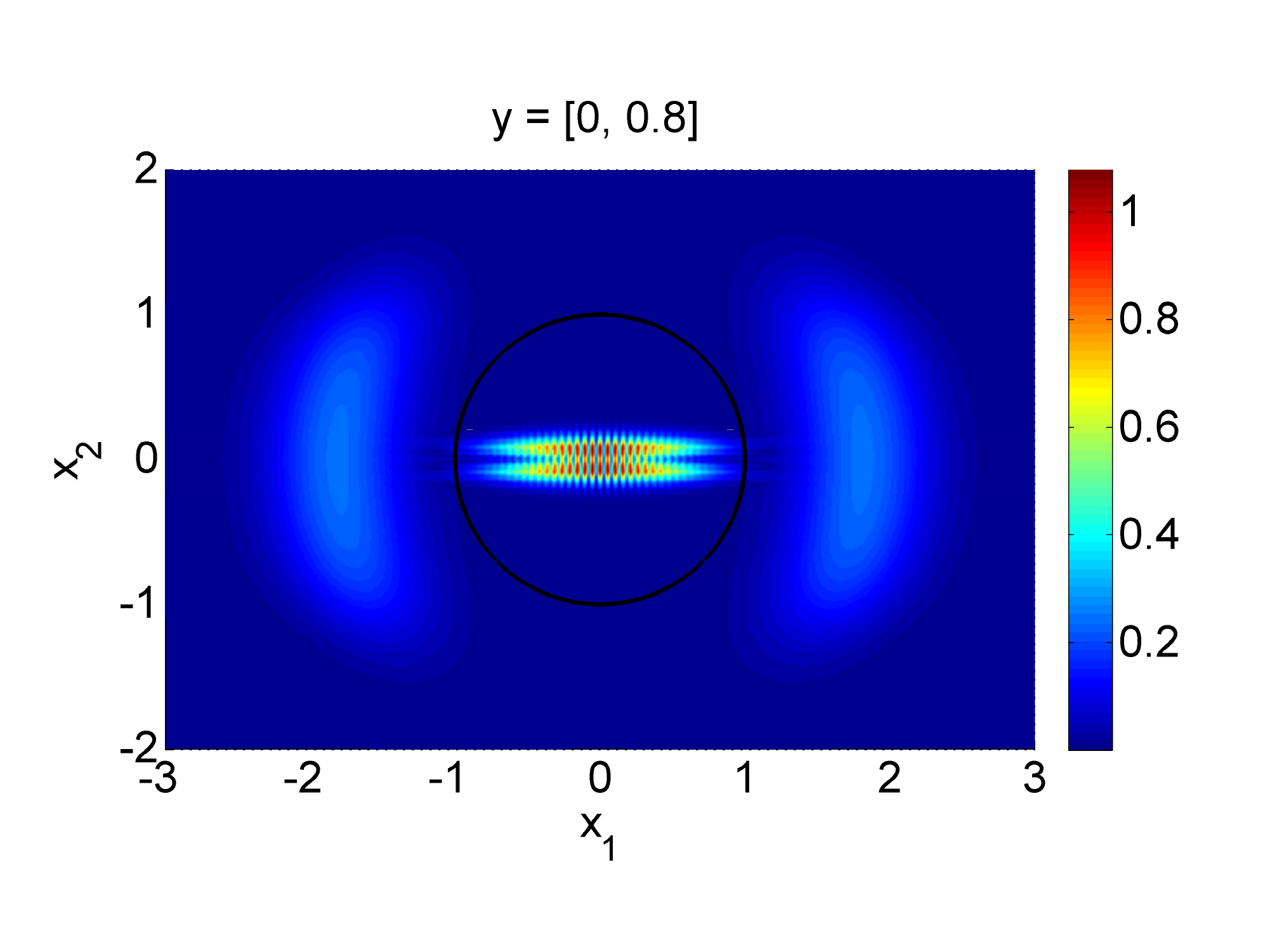}
                \end{minipage}
                \begin{minipage}[b]{0.32\textwidth}
                    \includegraphics[width=\textwidth]{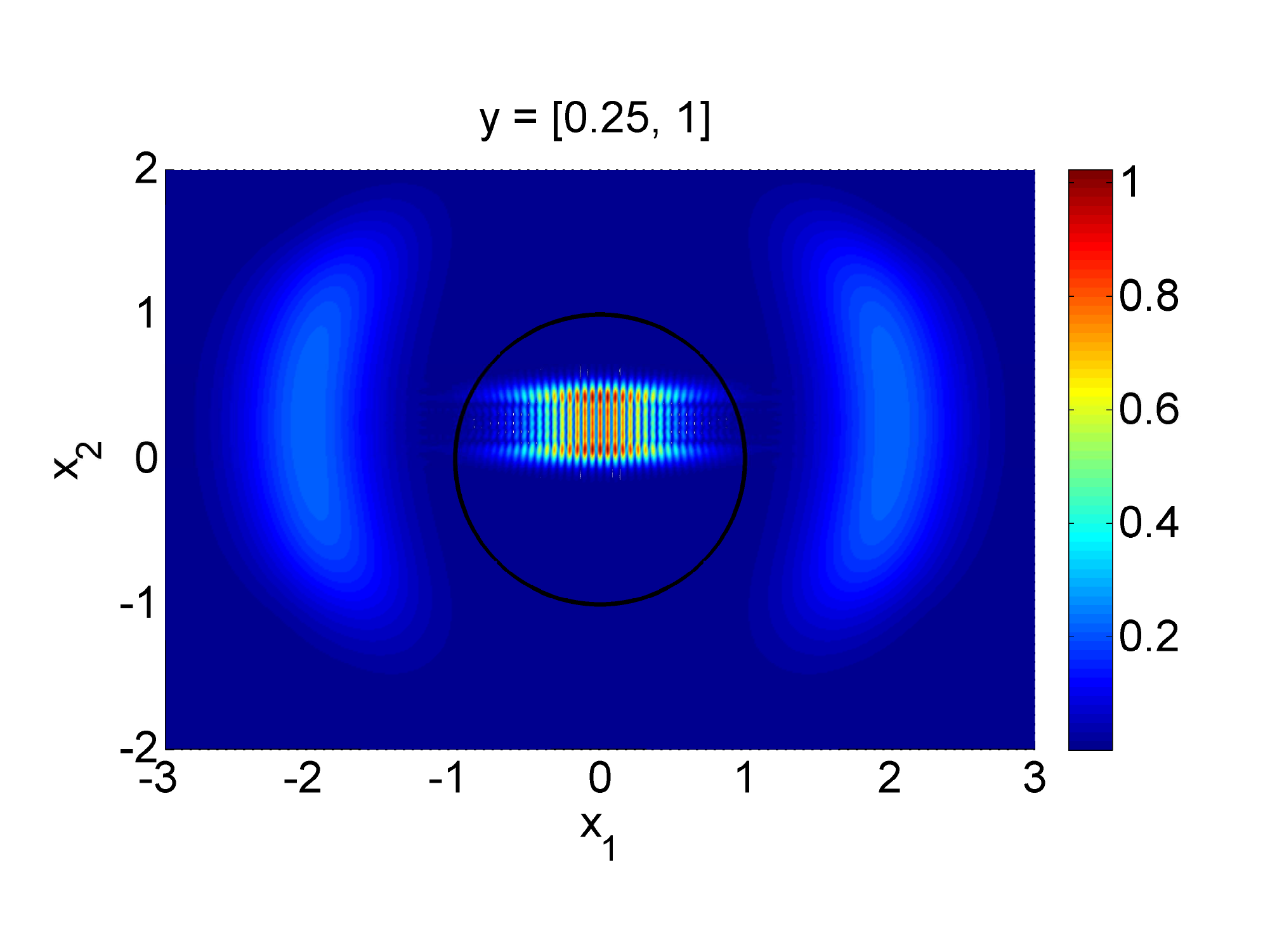}
                \end{minipage}
                \begin{minipage}[b]{0.32\textwidth}
                    \includegraphics[width=\textwidth]{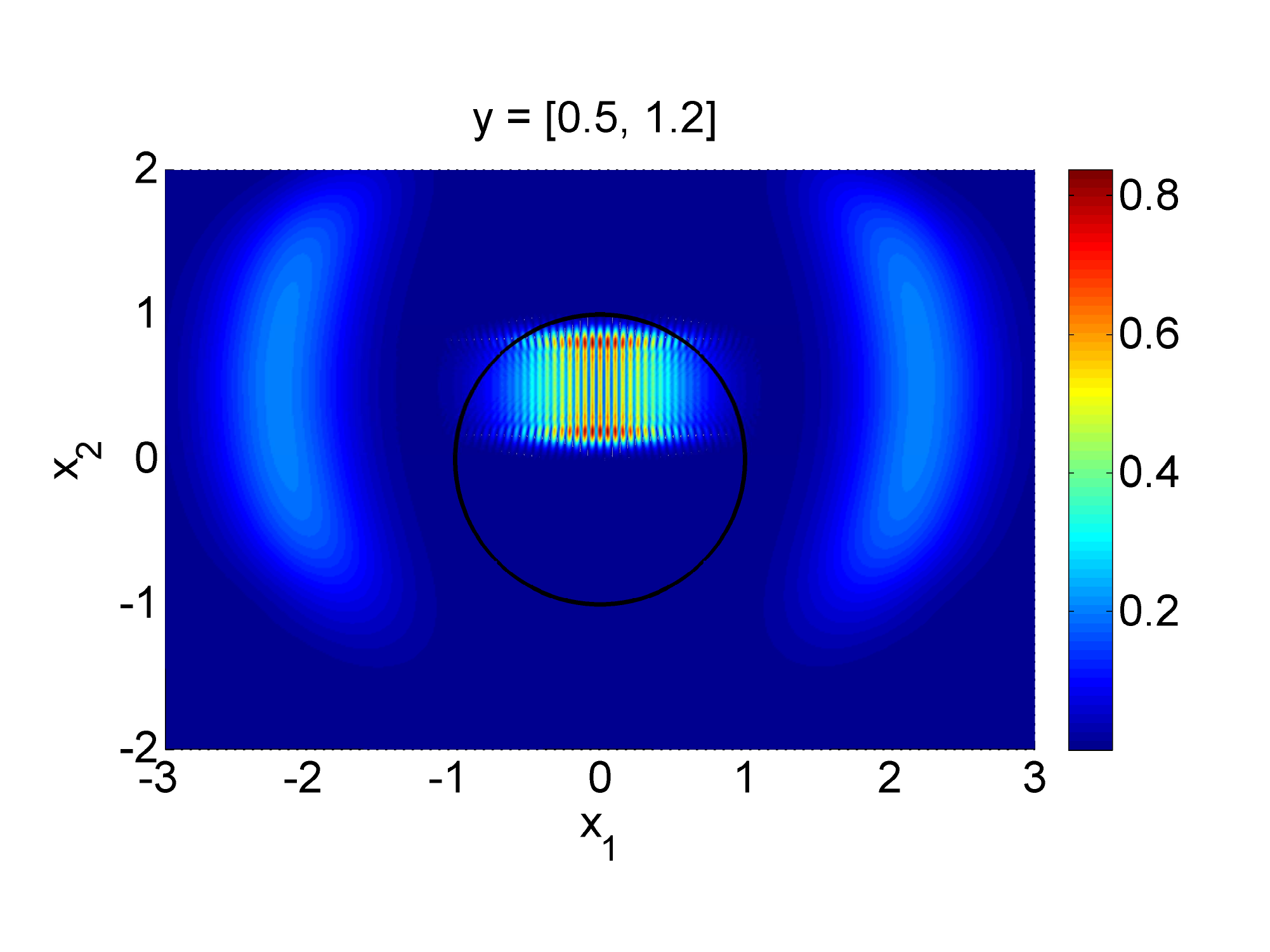}
                \end{minipage}
            \caption{The modulus of the GB solution $|u_1(t,\mb x,\mb y)|$ for $\eps = 1/60$ and $\varphi_0(\mb x,\mb y) = |x_1|+(x_2-y_1)^2$, at time $t=1$, for various $y$. The circle denotes the support of the test function $\psi$.} 
            \label{fig:solforqoi1b}
		\end{figure}
		
By analogous arguments as in Section \ref{sec:whatcouldgowrong}, the part of the solution overlapping in the origin is from the same GB mode. Hence, the QoI $\widetilde{\mathcal Q}_{\text{GB}}$ with the test function supported around the origin should not oscillate. This is indeed the case, as seen in the left column of Figure \ref{fig:qoiGB}, where the random variables are chosen as $y_1\in[0,0.5]$, $y_2\in[0.8,1.2]$ and we define $r\in [0,1]$, such that $[y_1,y_2] = [0,0.8] + r[0.5,0.4]$ (i.e. the diagonal parameter). We plot $\widetilde{\mathcal Q}_{\text{GB}}$ and its first and second derivatives with respect to $r$ at time $t=1$ as a function of $r$.

    	\begin{figure}[h!]
            \centering
                \begin{minipage}[b]{0.32\textwidth}
                    \includegraphics[width=\textwidth]{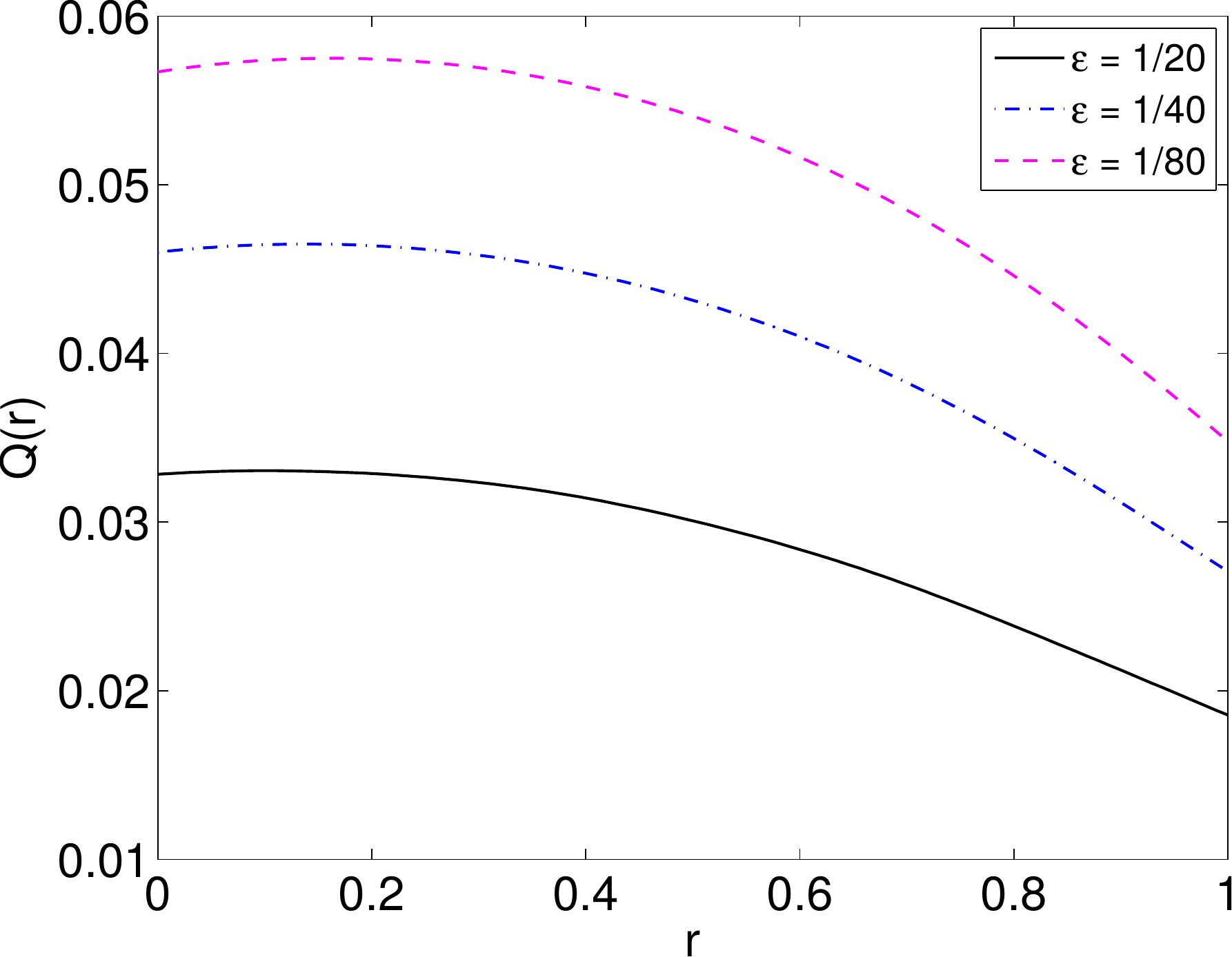}
                \end{minipage}
                \begin{minipage}[b]{0.32\textwidth}
                    \includegraphics[width=\textwidth]{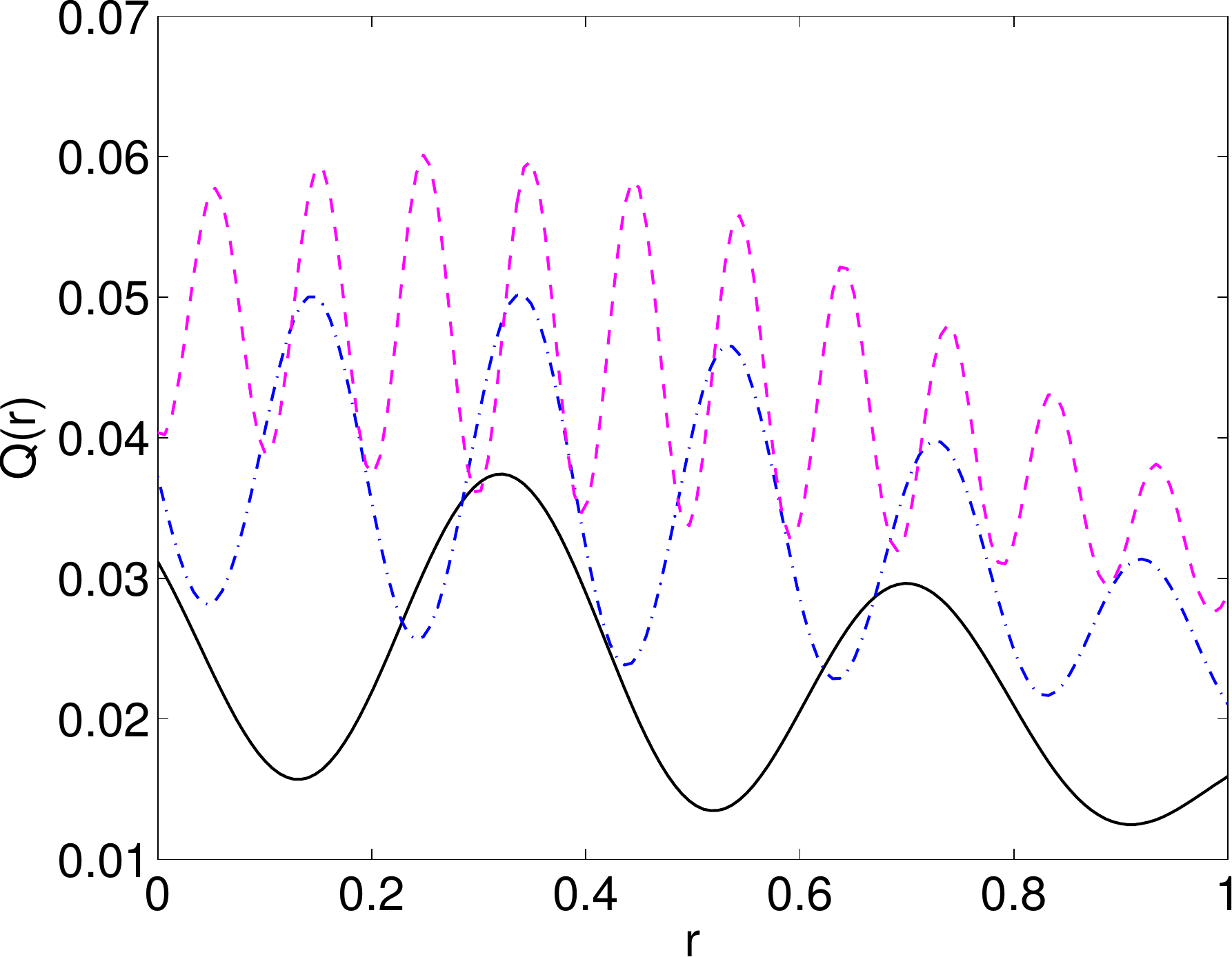}
                \end{minipage}
                \begin{minipage}[b]{0.32\textwidth}
                    \includegraphics[width=\textwidth]{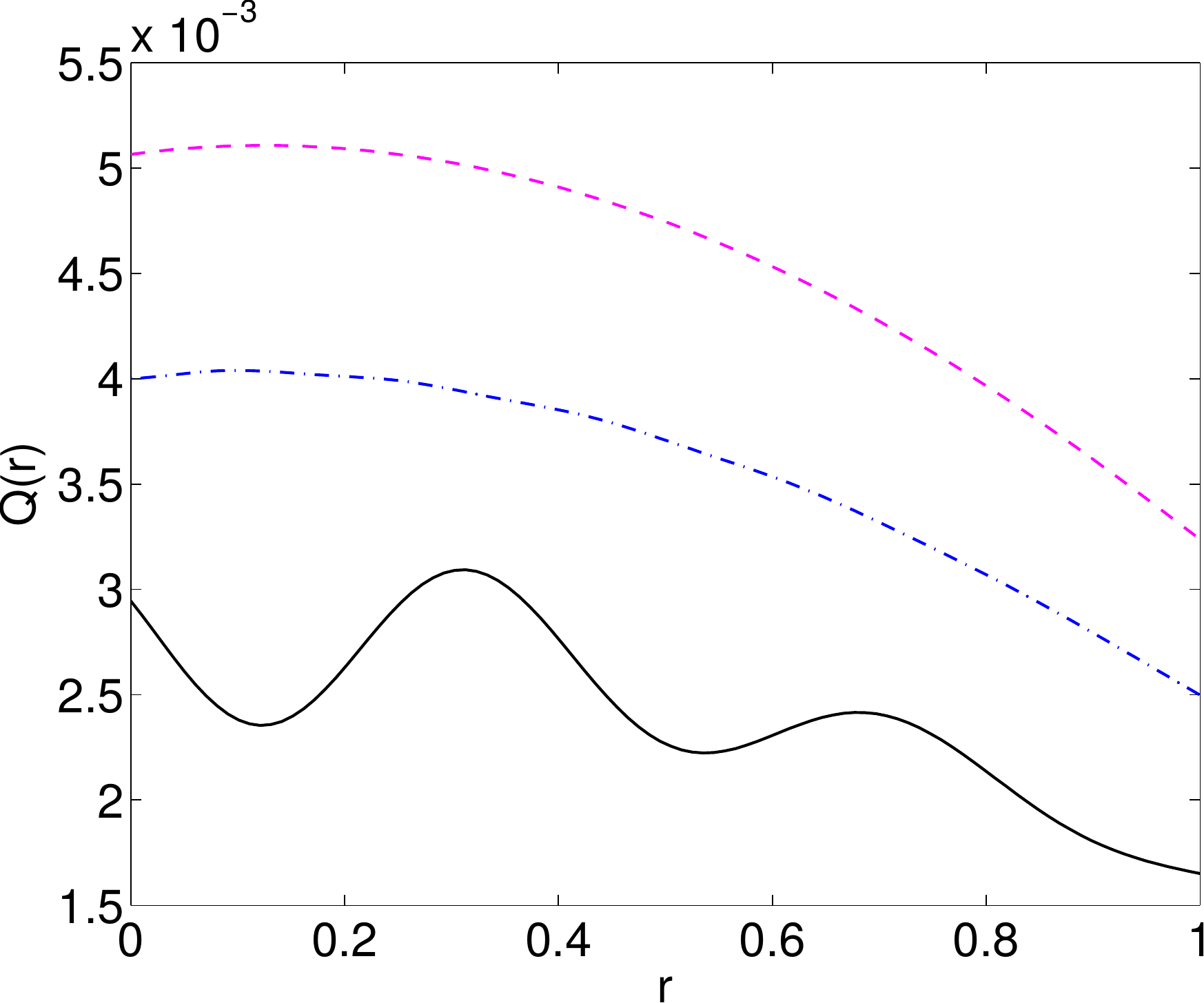}
                \end{minipage}                
                
                \begin{minipage}[b]{0.32\textwidth}
                    \includegraphics[width=\textwidth]{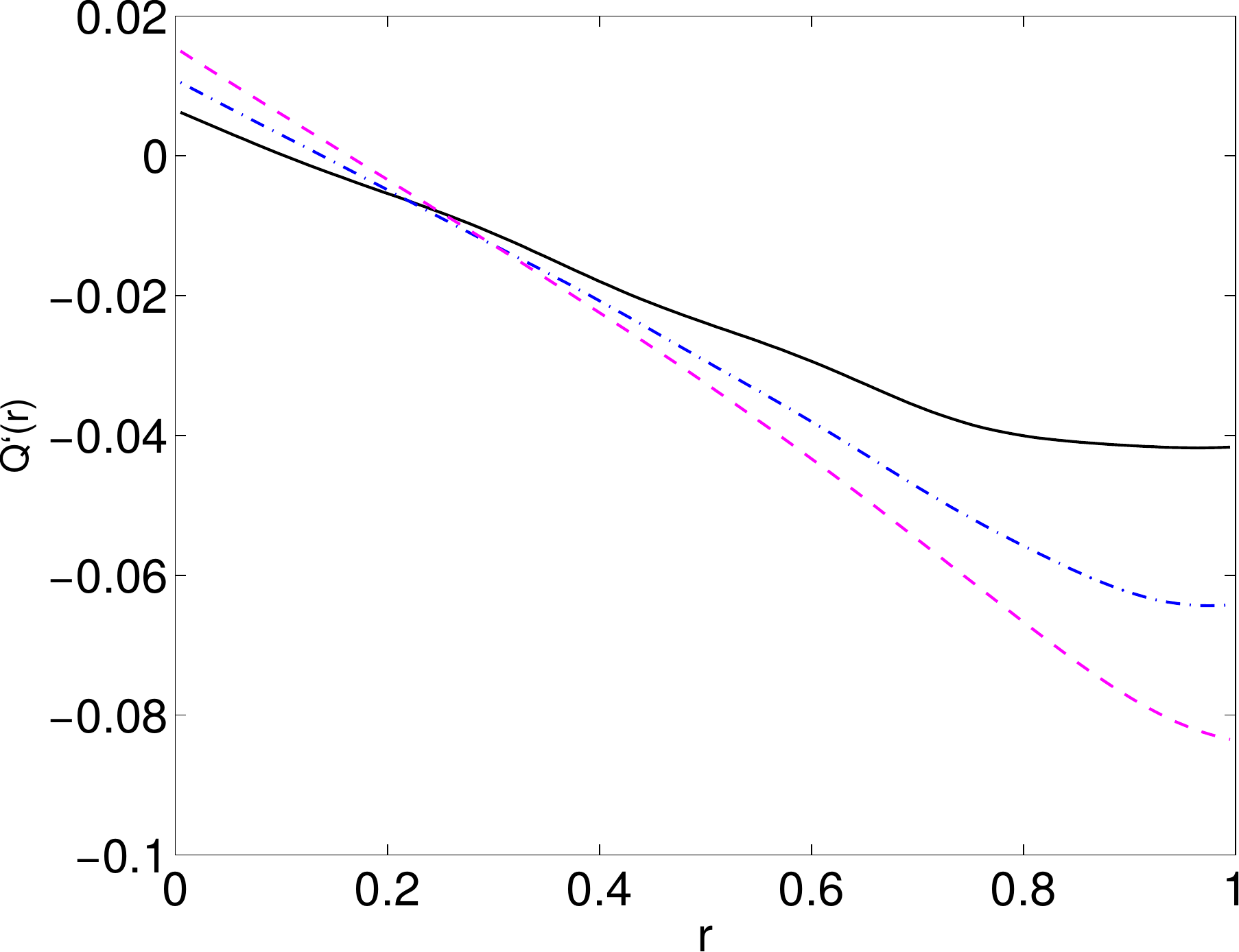}
                \end{minipage}
                \begin{minipage}[b]{0.32\textwidth}
                    \includegraphics[width=\textwidth]{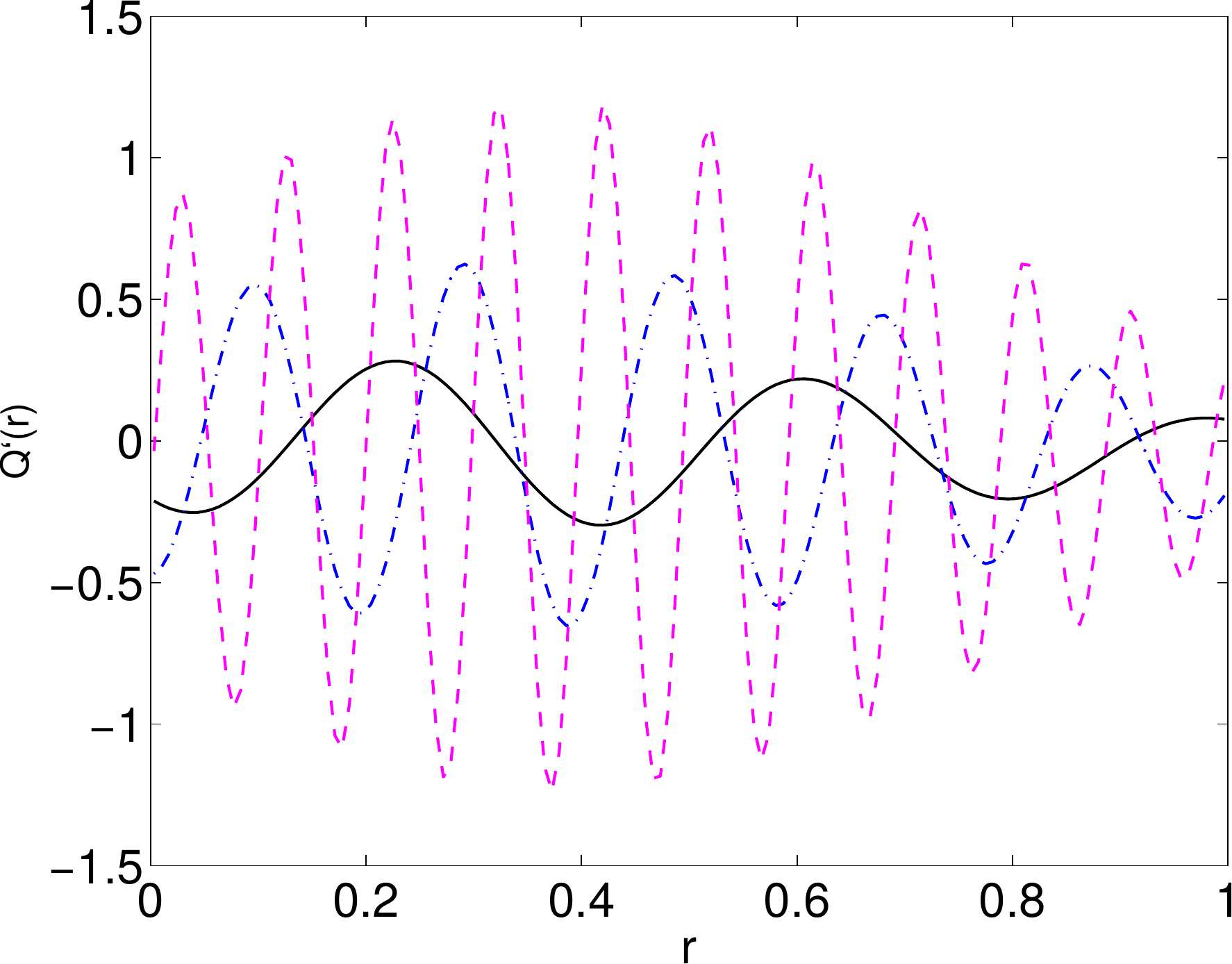}
                \end{minipage}
                \begin{minipage}[b]{0.32\textwidth}
                    \includegraphics[width=\textwidth]{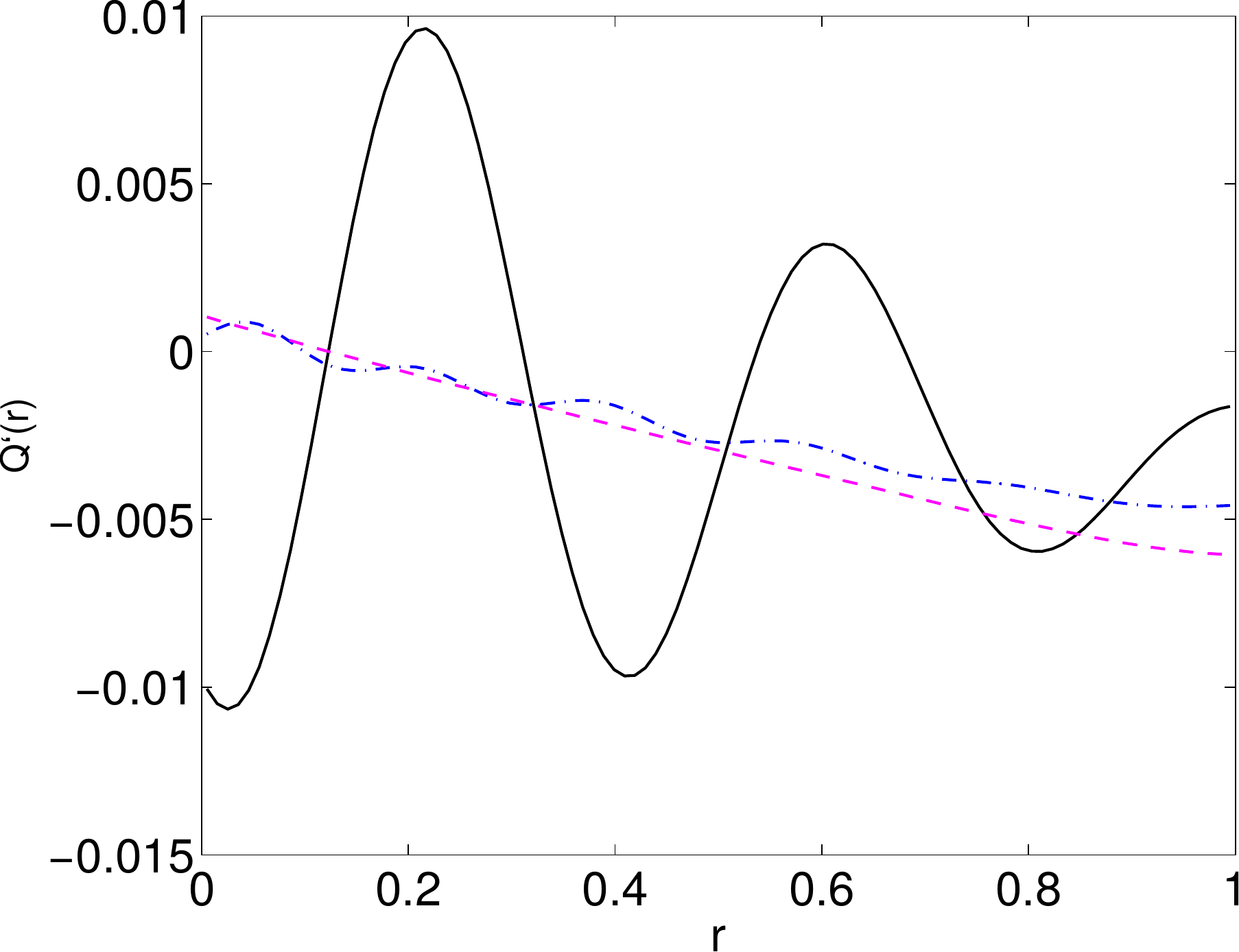}
                \end{minipage}                
                
                \begin{minipage}[b]{0.32\textwidth}
                    \includegraphics[width=\textwidth]{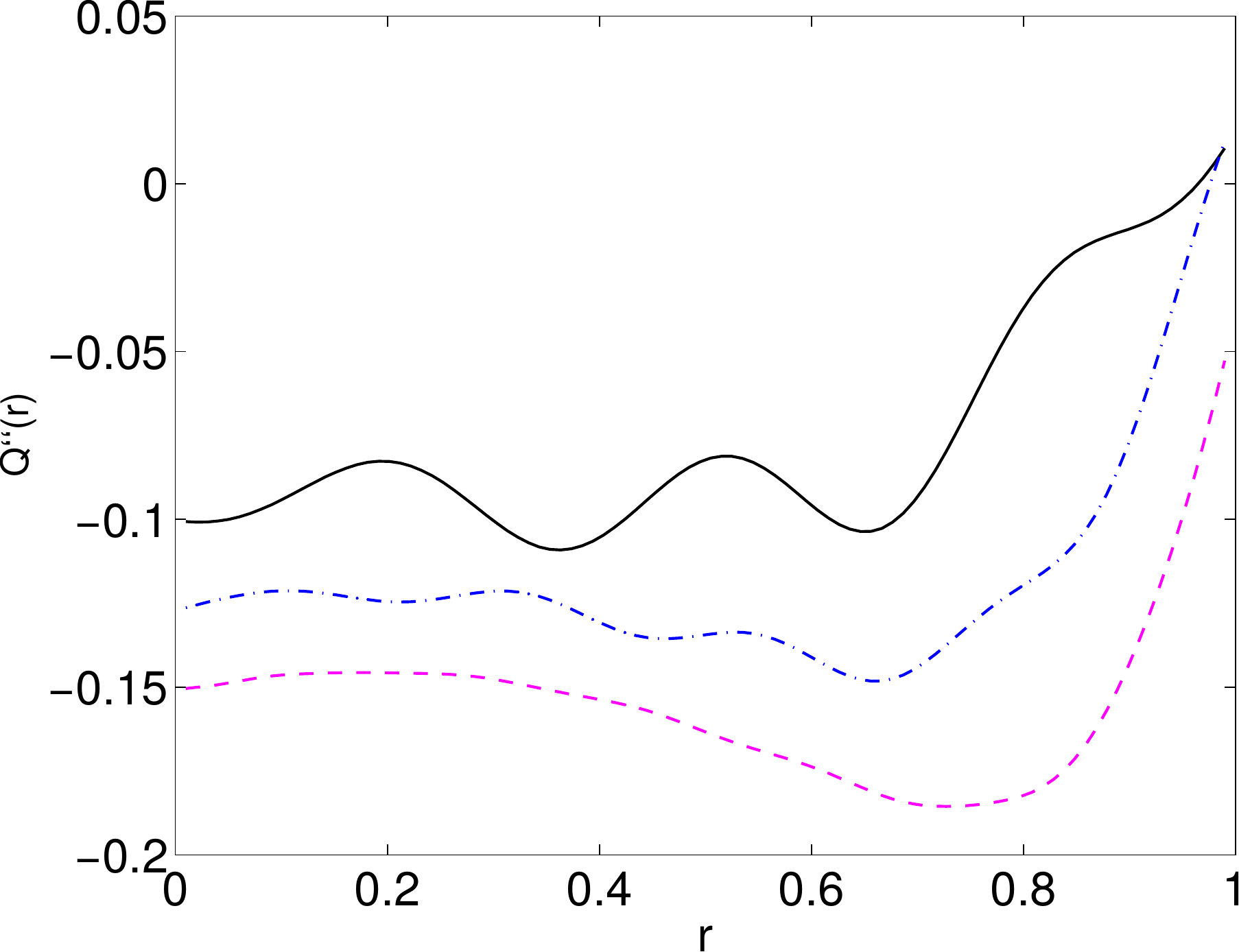}
                \end{minipage}
                \begin{minipage}[b]{0.32\textwidth}
                    \includegraphics[width=\textwidth]{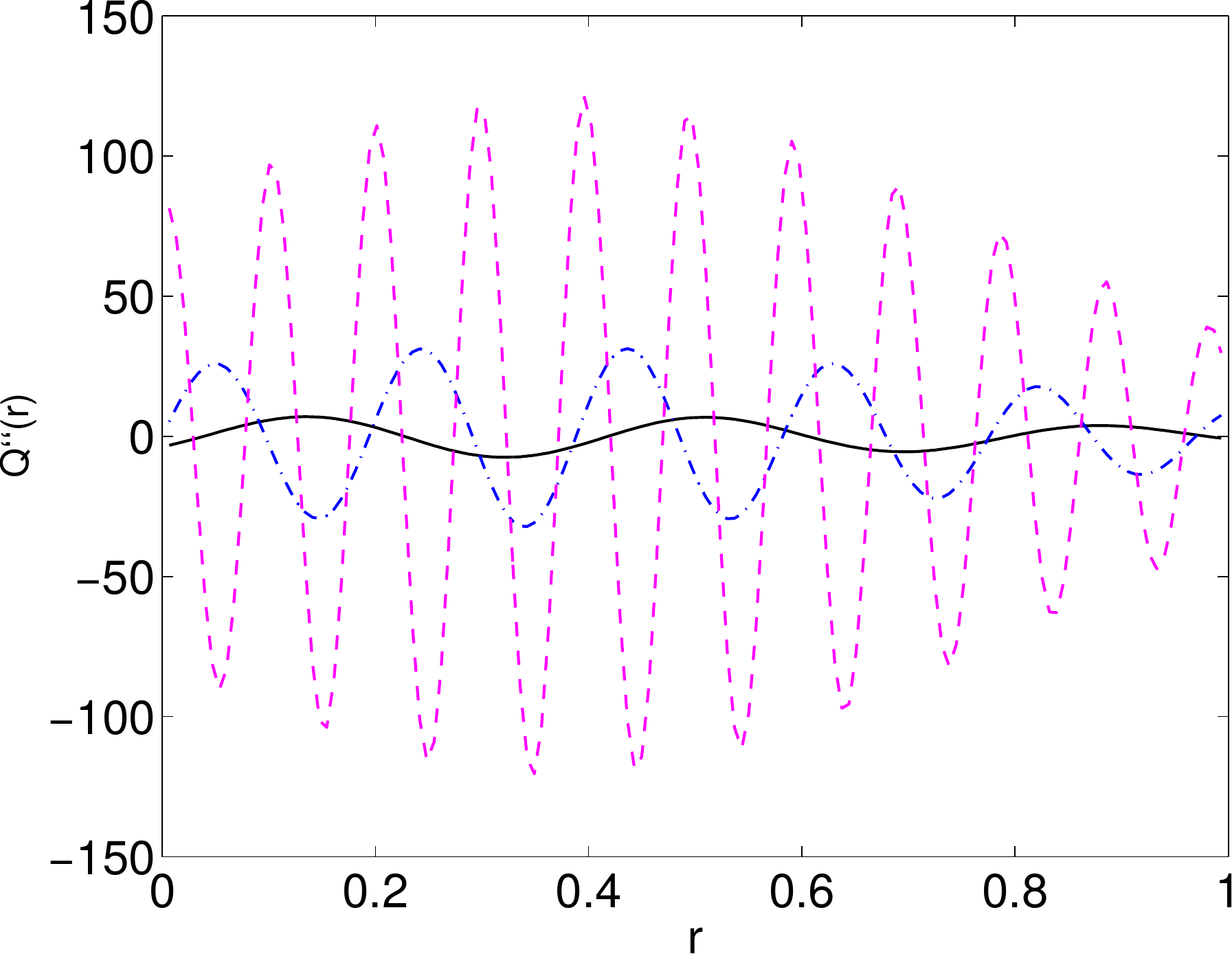}
                \end{minipage}  
                \begin{minipage}[b]{0.32\textwidth}
                    \includegraphics[width=\textwidth]{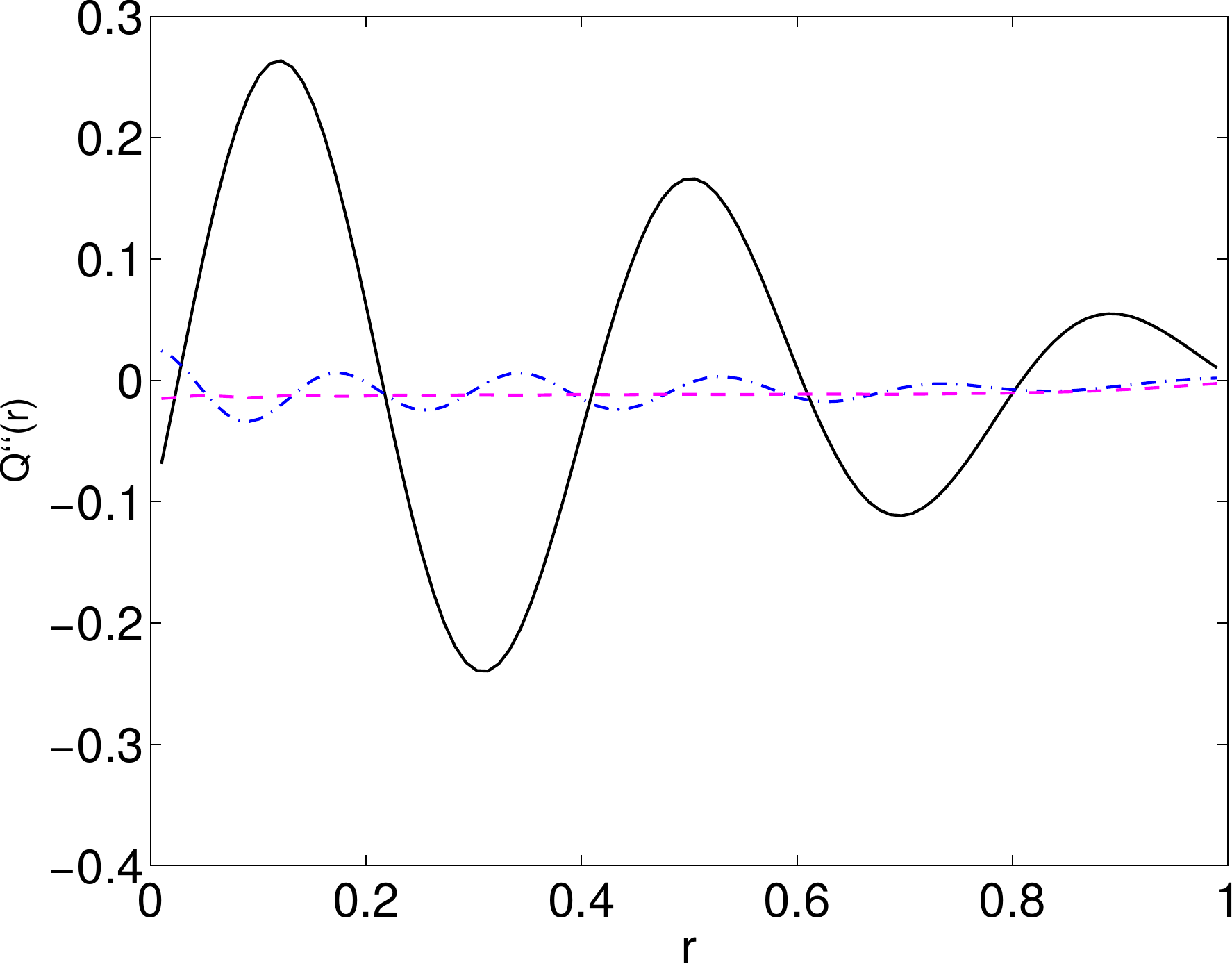}
                \end{minipage}                               
            \caption{Left column: $\widetilde{\mathcal Q}_{\text{GB}}$ and its first and second derivatives for one-mode solution. Central column: $\widetilde{\mathcal Q}_{\text{GB}}$ and its first and second derivatives for two-mode solution. Right column: $\mathcal Q_{\text{GB}}$ and its first and second derivatives for two-mode solution.} 
            \label{fig:qoiGB}
		\end{figure}

    	\begin{figure}[h!]
            \centering
                \begin{minipage}[b]{0.32\textwidth}
                    \includegraphics[width=\textwidth]{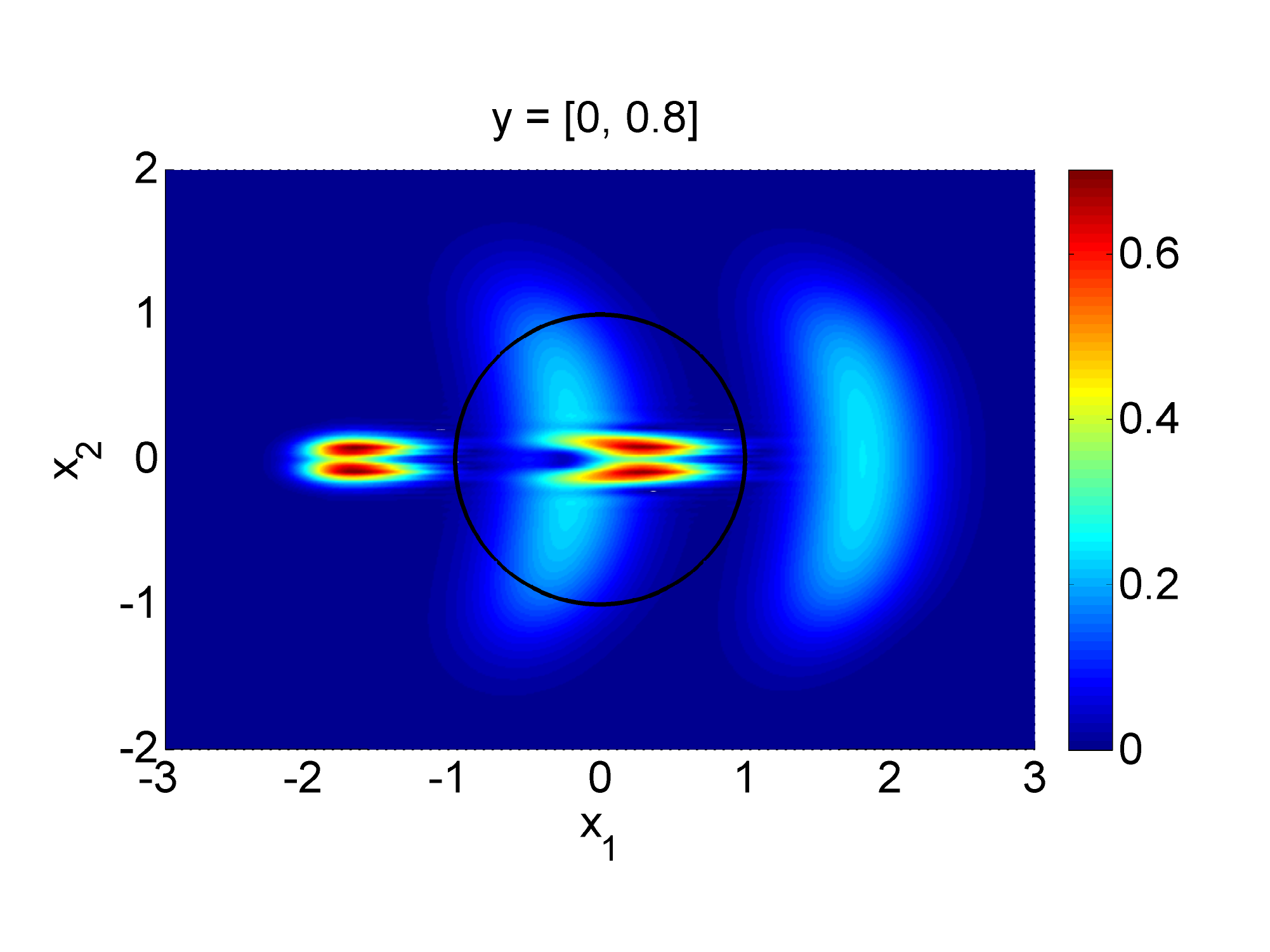}
                \end{minipage}
                \begin{minipage}[b]{0.32\textwidth}
                    \includegraphics[width=\textwidth]{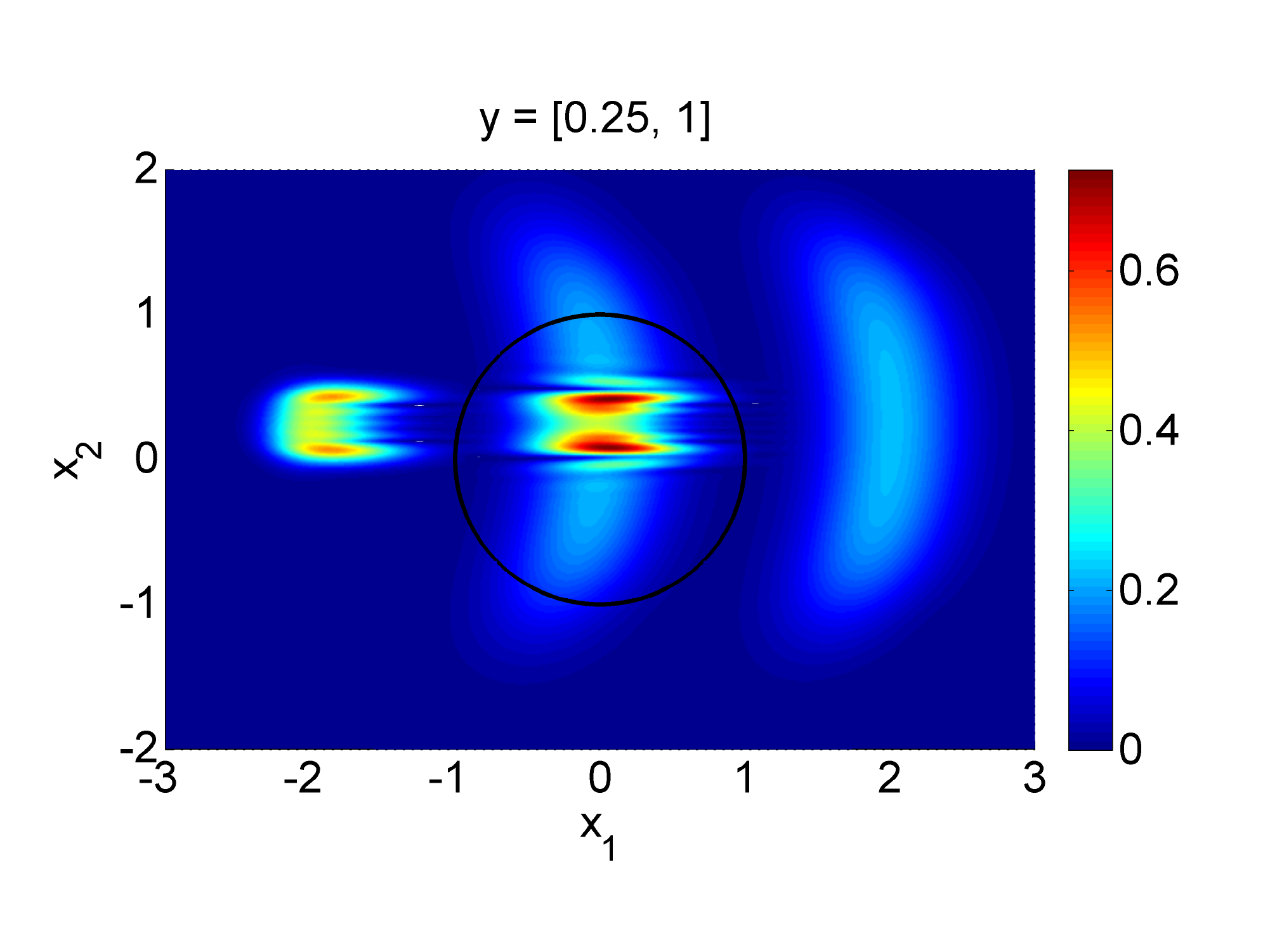}
                \end{minipage}
                \begin{minipage}[b]{0.32\textwidth}
                    \includegraphics[width=\textwidth]{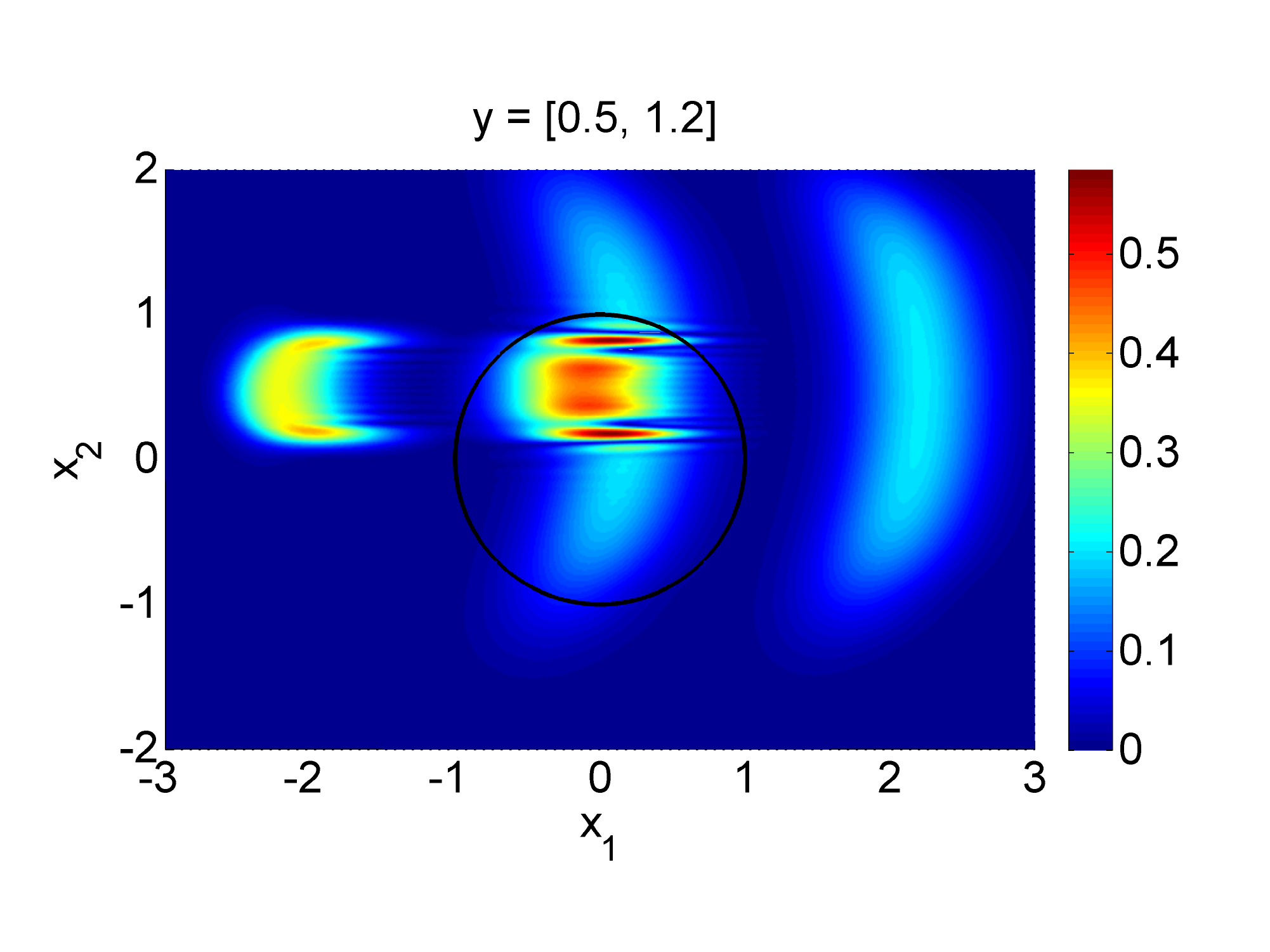}
                \end{minipage}
            \caption{The modulus of the GB solution $|u_1(t,\mb x,\mb y)|$ for $\eps=1/60$ and $\varphi_0(\mb x,\mb y) = x_1+(x_2-y_1)^2$ at time $t=1$, for various $y$. The circle denotes the support of the test function $\psi$.} 
            \label{fig:solforqoi2b}
		\end{figure}
		
Let us now consider the same setup only changing the initial phase function to
\[
	\varphi_0(\mb x,\mb y) = x_1 + (x_2-y_1)^2.
\]
Three realizations of $|u_1(t,\mb x,\mb y)|$ at $t=1$ are shown in Figure \ref{fig:solforqoi2b}.
It is no longer the case that the two branches moving towards the center can be described by the same GB mode.
A numerical test plotted in Figure \ref{fig:qoiGB}, central column, confirms the presence of two GB modes since the QoI cannot be bounded by a constant independent of $\eps$. Here, we again plot $\widetilde{\mathcal Q}_{\text{GB}}$ and its first and second derivatives with respect to $r$ at time $t=1$ as a function of $r$. Oscillations with increasing amplitudes can be observed.

To get rid of the oscillations, we need to consider the time-integrated QoI $\mathcal Q_{\text{GB}}$. We introduce the test function
\[
 \psi(\mb x) = \left\{ \begin{array}{ll}
 			e^{-\frac{|\mb x|^2}{1-|\mb x|^2}-10\frac{(t-1)^2}{0.2^2-(t-1)^2}}, & \text{for } |\mb x|\leq 1, \text{and } |t-1|\leq 0.2,\\
 			0, & \text{otherwise},
 			\end{array}\right.
\]
and integrate over both $\mb x$ and $t$. %$t\in[0.8,1.2]$. 
The QoI and its first and second derivatives are shown in Figure \ref{fig:qoiGB}, right column. The oscillations do not disappear entirely, but their amplitude decrease rapidly as $\eps\to 0$. This illustrates the difference between $\mathcal Q_{\text{GB}}$ and $\widetilde{\mathcal Q}_{\text{GB}}$.
		
%%%%%%%%%%%%%%%%%%%%%%%%%%%%%%%%%%%%%%%%%%%%%%%%%%%%%%%%%%%%%%%%%%%%%%%%%%%%%%%%%%%%%%%%%%%%%%%%%%%%%%%%%%%%%%%%%%%%%%%%%%%%%
%\newpage
\appendix
\section{Proof of Theorem \ref{th:oldmodified}}
To simplify the expressions, we first introduce the symmetrizing variables
\begin{equation}\label{qbar_deltaq}
\bar{\bf q} = \bar{\bf q}(t,\bold y,\mathbf z, \mathbf z') = \frac{{\bf q}(t,\bold y,\mathbf z) + {\bf q}(t,\bold y,\mathbf z')}{2}, \qquad \Delta {\bf q} = \Delta {\bf q}(t,\bold y,\mathbf z, \mathbf z') =  \frac{{\bf q}(t,\bold y,\mathbf z) - {\bf q}(t,\bold y,\mathbf z')}{2},
\end{equation}
and the symmetrized version of the 
space ${\mc T}_\eta$ 
used in Section~\ref{sec:th:new}% 
\begin{align*}
\mc T^s_\eta &:= \Bigl\{f\in C^\infty: \text{supp } f(t,\cdot,\mb y,\mb z,\mb z')\subset \Lambda^s_\eta(t,\mb y,\mb z,\mb z'), \: \forall t\in \Real, \, \mb y\in \Gamma, \, \mb z,\mb z'\in \Real^n\Bigr\},\\
&\ \ \ \text{where}\:
\Lambda^s_\eta(t,\mb y,\mb z,\mb z') := \{\mb x\in \Real^n: |\mb x-\Delta \mb q|\leq 2\eta \quad \text{and} \quad |\mb x+\Delta \mb q|\leq 2\eta\}.
\end{align*}
%
%
%We also let $\Omega_\mu$, $\Ucal_\mu$ and $\Vcal$ denote the following spaces:
%\begin{itemize}
%\item $\Omega_\mu = \Omega_\mu(t,\mb y,\mb z,\mb z'):=\{\mb x\in \Real^n: |\mb x-\Delta \mb q|\leq 2\mu \quad \text{and} \quad |\mb x+\Delta \mb q|\leq 2\mu\},$
%\item $\Ucal_\mu := \{f\in \mc Y :\: \text{supp}\: f(t,\,\cdot,\,\mb y,\mb z,\mb z') \subset \Omega_{\mu}(t,\mb y,\mb z,\mb z'), \: \forall t\in \Real,\, \mb y\in \Gamma, \, \mb z,\mb z' \in \Real^n\},$
%\item $\Vcal := C^{\infty}(\Real\times \Gamma \times  \Real^n\times \Real^n).$
%\end{itemize}
Then $I_0$ in \eqref{I0} can be written as
\be\label{I_eps_mod}
      I_0(t,{\mb y},{\mb z}, {\mb z}')  =
       \int_{\Real^n} h(t,\mb x,\mb y,\mb z,\mb z') (\mb x-\Delta \mb q)^{\bs\alpha}\: (\mb x+\Delta \mb q)^{\bs\beta}\: e^{i  \Psi_k(t, \mb x, \mb y,\mb z,\mb z') / \varepsilon}  \, d\mb x,
\ee
where $\Psi_k(t, \mb x, \mb y,\mb z,\mb z') = \Theta_k(t,\mb x + \bar{\mb q},\mb y,\mb z,\mb z')$ and $h(t,\mb x,\mb y,\mb z,\mb z') = f(t,\mb x+\bar{\mb q},\mb y,\mb z,\mb z')$ so that $h\in {\mc T}^s_{\eta}$ since $f\in \mathcal T_\eta$.
The following auxiliary lemma is a compilation of Lemma 3 and the differentiated version of Lemma 4 in \cite{malenova2017stochastic}.
\begin{lemma}\label{lemma:temp}
There exists $f_{\bs\mu,\bs\nu}\in C^\infty$ such that
	  \begin{equation*}%\label{xplusq_xminusq}
      (\mb x - \Delta \mb q)^{\bs\alpha} \, (\mb x + \Delta \mb q)^{\bs\beta} = \sum_{|\bs\mu + \bs\nu| = |\bs\alpha + \bs\beta|}
      f_{\bs\mu,\bs\nu}(t,\mb y, \mb z, \mb z') (\mb z-\mb z')^{\bs\mu} \mb x^{\bs\nu}.
  \end{equation*}
For the $k$-th order symmetrized Gaussian beam phase $\Psi_k$, there exist $a_{\bs\alpha,\bs\beta,m} \in C^\infty$ such that
\begin{equation*}%\label{partial1psi}
    \partial_{y_m}\Psi_k(t,\mb x,\mb y,\mb z,\mb z') = \sum_{2\leq|\bs\alpha+\bs\beta|\leq k+1} a_{\bs\alpha,\bs\beta,m}(t,\mb y,\mb z,\mb z') \, (\mb z-\mb z')^{\bs\alpha} \, \mb x^{\bs\beta}. %\quad k+1 \geq |\alpha_j + \gamma_j| \geq 2,
\end{equation*}
\end{lemma}
The following proposition is an update of \cite[Proposition 3]{malenova2017stochastic} adapted to our case.
%We henceforth fix the cutoff width parameter $\eta$ such that it is admissible in the sense of Definition \ref{def:admissible}.
\begin{proposition}\label{prop:bigtheorem}
There exist functions $g_{\bs\mu,\bs\nu,\bs\sigma,\ell} \in {\mc T}^s_{\eta}$ and $L_{\bs\sigma},M_{\bs\sigma}\geq 0$ such that the derivatives of $I_0$ in \eqref{I_eps_mod} with respect to $\mb y$ read
    \begin{equation}\label{bigtheorem}
        \partial^{\bs{\bs\sigma}}_{\mb y} I_0(t,\mb y,\mb z,\mb z') = %\sum_{(\ell,\bs\alpha,\bs\beta)\in \mathcal J^{\bs{\bs\sigma}}}
\sum_{\ell=-|\bs{\bs\sigma}|}^{L_{\bs\sigma}} \sum_{|\bs\mu+\bs\nu|+2\ell=0}^{M_{\bs\sigma}}
        \varepsilon^{\ell}
        (\mb z-\mb z')^{\bs\mu} \int_{\mathbb R^n} \mb x^{\bs\nu} g_{\bs\mu,\bs\nu,\bs{\bs\sigma},\ell}(t,\mb x,\mb y,\mb z,\mb z') e^{i\Psi_k(t,\mb x,\mb y,\mb z,\mb z')/\varepsilon}  d\mb x.
    \end{equation}
\end{proposition}
\begin{proof}
Recalling Lemma \ref{lemma:temp}, \eqref{I_eps_mod} can be reformulated as
\begin{equation*}
      I_0(t,{\mb y},{\mb z}, {\mb z}')  =
       \sum_{|\bs\mu+\bs\nu|=|\bs\alpha+\bs\beta|}(\mb z-\mb z')^{\bs\mu} \int_{\Real^n}
      \mb x^{\bs\nu}\: g_{\bs\mu,\bs\nu}(t,\mb x,\mb y,\mb z,\mb z') \: e^{i  \Psi_k(t, \mb x, \mb y,\mb z,\mb z') / \varepsilon}  \, d\mb x,
\end{equation*}
with $g_{\bs\mu,\bs\nu}(t,\mb x,\mb y,\mb z,\mb z') = h(t,\mb x,\mb y,\mb z,\mb z') f_{\bs\mu,\bs\nu}(t,\mb y,\mb z,\mb z')$. Therefore, since $h\in \mc T^s_{\eta}$ and $f_{\bs\mu,\bs\nu}\in C^\infty$ we have $g_{\bs\mu,\bs\nu}\in {\mc T}^s_{\eta}$.
We will now prove \eqref{bigtheorem} by induction. First, the statement is valid for $\bs\sigma=\mb 0$ since we can choose $L_{\bs 0}=0$, $M_{\bs 0}=|\bs\alpha+\bs\beta|$ and
\[
g_{\bs\mu,\bs\nu,\mb 0,0} = \left\{\begin{array}{ll}
		g_{\bs\mu,\bs\nu}, & \text{for }  |\bs \mu+\bs\nu|=|\bs\alpha +\bs\beta|,\\
		0, & \text{otherwise}.
	\end{array}\right.
\]
For the induction step let $L_{\bs\sigma},M_{\bs\sigma}\geq 0$ and $g_{\bs\mu,\bs\nu,\bs\sigma,\ell}\in {\mc T}^s_{\eta}$ be such that \eqref{bigtheorem} holds. Then for $\tilde{\bs\sigma} = \bs\sigma + \mb e_m,$ where $\mb e_m$ is the $m$-th unit vector, we have $\partial^{\tilde{\bs\sigma}}_{\mb y} I_0= \partial_{y_m} \partial_{\mb y}^{\bs\sigma} I_0$. Using \eqref{bigtheorem}, we can write %(omitting the arguments for now),
    \begin{align*}
        \partial^{\tilde{\bs\sigma}}_{\mb y} I_0 & = \sum_{\ell=-|\bs{\bs\sigma}|}^{L_{\bs\sigma}} \sum_{|\bs\mu+\bs\nu|+2\ell=0}^{M_{\bs\sigma}} \varepsilon^{\ell}
        (\mb z-\mb z')^{\bs\mu} \int_{\mathbb R^n} \mb x^{\bs\nu} \left(\partial_{y_m} g_{\bs\mu,\bs\nu,\bs{\bs\sigma},\ell} + g_{\bs\mu,\bs\nu,\bs{\bs\sigma},\ell}\, i\eps^{-1} \partial_{y_m} \Psi_k \right) \,e^{i\Psi_k/\varepsilon}  d\mb x \\ %\label{t1}\\
        & = {\textcircled{\raisebox{-0.9pt}{1}}} + {\textcircled{\raisebox{-0.9pt}{2}}}. 
%&= \sum_{\ell=-|\bs{\bs\sigma}|}^{L_{\bs\sigma}} \sum_{|\bs\mu+\bs\nu|+2\ell=0}^{M_{\bs\sigma}} \varepsilon^{\ell}
 %       (\mb z-\mb z')^{\bs\mu} \int_{\mathbb R^n} \mb x^{\bs\nu} \,\partial_{y_m} g_{\bs\mu,\bs\nu,\bs{\bs\sigma},\ell} \,e^{i\Psi_k/\varepsilon}  d\mb x\label{t1}\\
%&+ \sum_{\ell=-|\bs{\bs\sigma}|}^{L_{\bs\sigma}} \sum_{|\bs\mu+\bs\nu|+2\ell=0}^{M_{\bs\sigma}} \varepsilon^{\ell}
%        (\mb z-\mb z')^{\bs\mu} \int_{\mathbb R^n} \mb x^{\bs\nu}\,g_{\bs\mu,\bs\nu,\bs{\bs\sigma},\ell}\, i\eps^{-1} \partial_{y_m} \Psi_k \, e^{i\Psi_k/\varepsilon}  d\mb x.\label{t2}
    \end{align*}
Since $\partial_{y_m} g_{\bs\mu,\bs\nu,\bs\sigma,\ell}\in {\mc T}^s_{\eta}$, {\textcircled{\raisebox{-0.9pt}{1}}} is of the form \eqref{bigtheorem} with $L_{\tilde{\bs\sigma}}=L_{\bs\sigma}$, $M_{\tilde{\bs\sigma}} = M_{\bs\sigma}$ and
\[
	g_{\bs\mu,\bs\nu,\tilde{\bs\sigma},\ell} = \left\{\begin{array}{ll}
	\partial_{y_m}g_{\bs\mu,\bs\nu,\bs\sigma,\ell}, & \text{for } \ell \geq -|\bs\sigma|,\\
	0, & \text{for } \ell = -|\bs\sigma|-1.\end{array}\right.
\]
Regarding the remaining terms {\textcircled{\raisebox{-0.9pt}{2}}}, let us express the derivative $\partial_{y_m}\Psi_k$ by Lemma \ref{lemma:temp}. Then {\textcircled{\raisebox{-0.9pt}{2}}}  reads
    \be\label{tempp}
\sum_{\ell=-|\bs{\bs\sigma}|}^{L_{\bs\sigma}} \sum_{|\bs\mu+\bs\nu|+2\ell=0}^{M_{\bs\sigma}} \sum_{|\bs\gamma+\bs\delta|=2}^{k+1} \eps^{\ell-1}
        (\mb z-\mb z')^{\bs\mu+\bs\gamma} \int_{\mathbb R^n} \mb x^{\bs\nu+\bs\delta} h_{\bs\mu,\bs\nu,\bs\gamma,\bs\delta,\ell}\, e^{i\Psi_k/\varepsilon}  d\mb x,
    \ee
with $h_{\bs\mu,\bs\nu,\bs\gamma,\bs\delta,\ell} = i a_{\bs\gamma,\bs\delta,m} \, g_{\bs\mu,\bs\nu,\bs{\bs\sigma},\ell}\in {\mc T}^s_{\eta}$ since $g_{\bs\mu,\bs\nu,\bs{\bs\sigma},\ell}\in \mc T^s_{\eta}$ and $a_{\bs\gamma,\bs\delta,m}\in C^\infty$. Each of the terms in \eqref{tempp} is therefore of the form 
\[
\eps^{\tilde\ell}
        (\mb z-\mb z')^{\tilde{\bs\mu}} \int_{\mathbb R^n} \mb x^{\tilde{\bs\nu}} h_{\tilde{\bs\mu},\tilde{\bs\nu},\tilde{\ell}}(t,\mb x,\mb y,\mb z,\mb z')\, e^{i\Psi_k(t,\mb x,\mb y,\mb z,\mb z')/\varepsilon}  d\mb x,
\]
where 
\[
	-|\tilde{\bs\sigma}|\leq \tilde\ell = \ell-1 \leq L_{\bs\sigma}-1=: L_{\tilde{\bs\sigma}},
\]
and
\[
	0 \leq|\tilde{\bs\mu}+\tilde{\bs\nu}| + 2\tilde\ell = |\bs\mu +\bs\nu|+2\ell + |\bs\gamma+\bs\delta|-2\leq M_{\bs\sigma}+k-1=: M_{\tilde{\bs\sigma}},
\]
which finalizes the induction argument and concludes Proposition \ref{prop:bigtheorem}.
\end{proof}

The rest of the proof of \cite[Theorem 1]{malenova2017stochastic} can be used as it is. In particular, if $\eta<\infty$, then
\cite[Lemma~5]{malenova2017stochastic} and \cite[Lemma~6]{malenova2017stochastic} are valid without any alteration. Ultimately, we are using the fact that $0\leq|\bs\mu+\bs\nu|+2\ell$ in \eqref{bigtheorem} which is still the case due to Proposition \ref{prop:bigtheorem}. Finally, since all estimates in \cite{malenova2017stochastic} are uniform in $t$, the constant $C_{\bs\sigma}$ is uniform in $[0,T]$ as well.
This completes the proof of Theorem \ref{th:oldmodified}.

%{\color{cyan}
%\begin{remark}
%We note that the latter part of Lemma \ref{lemma:temp} holds also for $\partial_t \Psi_k$. Therefore, Theorem \ref{th:oldmodified} holds for $\partial_t^\sigma I_0$ replacing $\partial_{\mb y}^{\bs\sigma} I_0$
%and hence $\widetilde{\mathcal Q}^{p,\bs\alpha}_{\text{GB}}$ is also regular in time, i.e. 
%\[
%\sup_{\substack{\mb y\in \Gamma_c\\ t\in [0,T]}}\left|\frac{\partial^\sigma \widetilde{\mathcal Q}^{p,\bs\alpha}_{\text{GB}}(t,\mb y)}{\partial t^\sigma}\right|\leq C_\sigma, \qquad \forall\sigma\in \mathbb N_0,
%\]
%provided the assumptions of Theorem \ref{th:new} are valid.
%\end{remark}
%}

%%%%%%%%%%%%%%%%%%%%%%%%%%%%%%%%%%%%%%%%%%%
%\newpage
\bibliographystyle{plain}
\bibliography{refs3}

\end{document}